\documentclass[reqno,11pt]{amsart}
\usepackage{amsmath, latexsym, amsfonts, amssymb, amsthm, amscd, dsfont}
\usepackage{mathrsfs}
\usepackage[lofdepth,lotdepth]{subfig}
\usepackage{graphicx,epsf,psfrag}
\usepackage{MnSymbol}

\numberwithin{equation}{section}

\newtheorem{theorem}{Theorem}[section]
\newtheorem{lemma}[theorem]{Lemma}
\newtheorem{proposition}[theorem]{Proposition}

\newtheorem{remark}[theorem]{Remark}
\newtheorem{definition}[theorem]{Definition}

\newcommand{\wrt}{w.r.t.\ }
\newcommand{\iid}{i.i.d.\ }

\renewcommand{\div}{{\rm div}}
\def\N#1{\left\|\,#1\,\right\|}
\def\Ninf#1{\left\|\,#1\,\right\|_{\infty}}

\newcommand{\cro}[2]{\ensuremath{\left\langle#1\, ,\, #2\right\rangle}}

\newcommand{\bE}{\ensuremath{\mathbf{E}}}

\newcommand{\bP}{\ensuremath{\mathbf{P}}}

\newcommand{\bV}{\ensuremath{\mathbf{V}}}
\newcommand{\bW}{\ensuremath{\mathbf{W}}}

\newcommand{\bbR}{\ensuremath{\mathbb{R}}}
\newcommand{\bbS}{\ensuremath{\mathbb{S}}}

\newcommand{\bbX}{\ensuremath{\mathbb{X}}}
\newcommand{\bbY}{\ensuremath{\mathbb{Y}}}
\newcommand{\bbZ}{\ensuremath{\mathbb{Z}}}

\newcommand{\cC}{\ensuremath{\mathcal{C}}}

\newcommand{\cE}{\ensuremath{\mathcal{E}}}
\newcommand{\cF}{\ensuremath{\mathcal{F}}}

\newcommand{\cH}{\ensuremath{\mathcal{H}}}
\newcommand{\cI}{\ensuremath{\mathcal{I}}}
\newcommand{\cJ}{\ensuremath{\mathcal{J}}}
\newcommand{\cK}{\ensuremath{\mathcal{K}}}
\newcommand{\cL}{\ensuremath{\mathcal{L}}}
\newcommand{\cM}{\ensuremath{\mathcal{M}}}

\newcommand{\cP}{\ensuremath{\mathcal{P}}}
\newcommand{\cQ}{\ensuremath{\mathcal{Q}}}
\newcommand{\cR}{\ensuremath{\mathcal{R}}}
\newcommand{\cS}{\ensuremath{\mathcal{S}}}
\newcommand{\cT}{\ensuremath{\mathcal{T}}}
\newcommand{\cU}{\ensuremath{\mathcal{U}}}
\newcommand{\cV}{\ensuremath{\mathcal{V}}}
\newcommand{\cW}{\ensuremath{\mathcal{W}}}

\newcommand{\tk}{\ensuremath{\tilde{k}}}

\newcommand{\tx}{\ensuremath{\tilde{x}}}

\newcommand{\tB}{\ensuremath{\tilde{B}}}

\newcommand{\tcH}{\ensuremath{\widetilde{\cH}}}

\newcommand{\ttheta}{\ensuremath{\tilde{\theta}}}
\newcommand{\tTheta}{\ensuremath{\tilde{\Theta}}}
\newcommand{\tomega}{\ensuremath{\tilde{\omega}}}
\newcommand{\ttau}{\ensuremath{\tilde{\tau}}}
\newcommand{\tpi}{\ensuremath{\tilde{\pi}}}

\newcommand{\btheta}{\ensuremath{\bar\theta}}

\newcommand{\dd}{\,\text{\rm d}}             

\DeclareMathSymbol{\leqslant}{\mathalpha}{AMSa}{"36} 
\DeclareMathSymbol{\geqslant}{\mathalpha}{AMSa}{"3E} 
\DeclareMathSymbol{\eset}{\mathalpha}{AMSb}{"3F}     
\renewcommand{\leq}{\;\leqslant\;}                   
\renewcommand{\geq}{\;\geqslant\;}                   

\renewcommand{\div}{{\rm div}}

\newcommand{\nuN}[1]{\nu_{N,\, #1}}
\newcommand{\bnuN}[1]{\bar{\nu}_{N,\, #1}}
\newcommand{\etaN}[1]{\eta_{N,\, #1}}

\DeclareMathOperator*{\Supp}{Supp}
\newcommand{\com}[1]{}

\begin{document}

\title[Fluctuations for mean-field diffusions in spatial interaction]{Transition from Gaussian to non-Gaussian fluctuations for mean-field diffusions in spatial interaction}
\author{Eric Lu\c con$^1$}
\author{Wilhelm Stannat$^{2, 3}$}
\footnotetext[1]{ Laboratoire MAP5, Universit\'e Paris Descartes and CNRS, Sorbonne Paris Cit\'e, 45 rue des Saints P\`eres, 75270 Paris Cedex 06, France. email: eric.lucon@parisdescartes.fr.}
\footnotetext[2]{Institut f\"ur Mathematik, Technische Universit\"at Berlin, Stra{\ss}e des 17. Juni 136, D-10623 Berlin, Germany. email: stannat@math.tu-berlin.de.}
\footnotetext[3]{Bernstein Center for Computational Neuroscience Berlin, Philippstr. 13, D-10115 Berlin, Germany.}

\date{\today}
\keywords{Weakly interacting diffusions, spatially-extended particle systems, weighted empirical measures, fluctuations, Kuramoto model, neuronal models, stochastic partial differential equations}

\begin{abstract}
We consider a system of $N$ disordered mean-field interacting diffusions within spatial constraints : each particle $ \theta_{ i}$ is attached to one site $x_{ i}$ of a periodic lattice and the interaction between particles $ \theta_{ i}$ and $ \theta_{ j}$ decreases as $ \left\vert x_{ i} - x_{ j} \right\vert^{ - \alpha}$ for $ \alpha\in[0, 1)$. In a previous work \cite{LucSta2014}, it was shown that the empirical measure of the particles converges in large population to the solution of a nonlinear partial differential equation of McKean-Vlasov type. The purpose of the present paper is to study the fluctuations associated to this convergence. We exhibit in particular a phase transition in the scaling and in the nature of the fluctuations: when $ \alpha\in[0, \frac{ 1}{ 2})$, the fluctuations are Gaussian, governed by a linear SPDE, with scaling $ \sqrt{N}$ whereas the fluctuations are deterministic with scaling $ N^{ 1- \alpha}$ in the case $ \alpha\in( \frac{ 1}{ 2}, 1)$.

2010 \textit{Mathematics Subject Classification: 60F05, 60G57, 60H15, 82C20, 92B25.}
\end{abstract}

\maketitle


\section{Introduction}
The aim of the paper is to study the large population fluctuations of disordered mean-field interacting diffusions within spatial interaction. A general instance of the model may be given by the following system of $N$ coupled stochastic differential equations in $\bbX:=\bbR^{ m}$ ($m\geq1$):
\begin{equation}
\label{eq:odegene}
\dd\theta_{i, t}= c(\theta_{i,t}, \omega_{i}) \dd t + \frac{1}{|\Lambda_{N}|}\sum_{\substack{j\in\Lambda_{N}\\ j\neq i}} \Gamma\left(\theta_{i,t}, \omega_{i}, \theta_{j, t}, \omega_{j} \right)\Psi(x_{i}, x_{j}) \dd t + \dd B_{i, t},\ i\in \Lambda_{ N},\ t\in[0, T]
\end{equation}
where $T>0$ is a fixed but arbitrary time horizon. In \eqref{eq:odegene}, $c(\theta_{ i, t}, \omega_{ i})$ models the local dynamics of the particle $ \theta_{ i}$, $ \Gamma( \theta_{ i, t}, \omega_{ i}, \theta_{ j, t}, \omega_{ j})$ governs the interaction between particles $ \theta_{ i}$ and $ \theta_{ j}$ and $(B_{ i})_{ i\in \Lambda_{ N}}$ is a collection of independent standard Brownian motions modeling thermal noise in the system. Both local dynamics and interaction are perturbed by an independent random environment that is a sequence of i.i.d. random variables $(\omega_{i})_{ i\geq1}$ in $\bbY:= \bbR^{n}$ ($n\geq1$) modeling some local inhomogeneity for each particle.

The novelty of \eqref{eq:odegene} is that some geometry is imposed on the interactions: the particles $ \theta_{ i}$ in \eqref{eq:odegene} are regularly positioned on a periodic lattice (of dimension $1$ for simplicity) and the interaction between two particles depend on the distance between them. Namely, define $\Lambda_{ N}:= \left\lbrace -N, \ldots, N\right\rbrace$ with $-N$ and $N$ being identified, with cardinal $ \left\vert \Lambda_{ N} \right\vert=2N$ and suppose that for all $i\in \Lambda_{ N}$, the particle $ \theta_{ i}$ is at the fixed position $x_{ i}:= \frac{ i}{ 2N}\in\bbS_{ N}$, where $\bbS_{ N}:= \left\lbrace x_{ i}:= \frac{ i}{ 2N},\ i\in \Lambda_{ N}\right\rbrace$ is a subset of the one-dimensional circle $\bbS:= \bbR/\bbZ$. Making the obvious identification between $x \in \left[ -\frac{ 1}{ 2}, \frac{ 1}{ 2}\right]$ and its equivalence class $\bar x\in \bbS$, the euclidean norm $ \left\vert \cdot \right\vert$ on $\left[ -\frac{ 1}{ 2}, \frac{ 1}{ 2}\right]$ induces a distance $d(\cdot, \cdot)$ on $\bbS$ by:
\begin{equation}
\label{eq:dist_d}
\forall \bar x= x[1], \bar y= y[1]\in \bbS,\ d(\bar x, \bar y)= \min( \left\vert x-y \right\vert, 1 - \left\vert x-y \right\vert).
\end{equation}
The interaction in \eqref{eq:odegene} is supposed to decrease polynomially with the distance between particles through the spatial weight
\begin{equation}
\label{eq:def_psi}
\Psi(x, y):= d(x, y)^{ - \alpha},\ x, y\in \bbS,
\end{equation}
where $ \alpha$ is a parameter in $[0, 1)$.
\begin{remark}
We will often drop the notation $j \neq i$ in sums as in \eqref{eq:odegene}, since, with no loss of generality, one can define $ \Psi(x, x):=0$ for all $x$.
\end{remark}
\subsection{Weakly interacting diffusions}
\subsubsection{Mean-field models in physics and neuroscience}
In the particular case where $ \alpha=0$ and $\omega_{ i}=0$, the geometry and the disorder in \eqref{eq:odegene} become irrelevant and we retrieve the usual class of \emph{weakly interacting diffusions} (studied since McKean, Oelschl\"ager and Sznitman \cite{McKean1967, Gartner, Oelsch1984,SznitSflour}). Such systems are encountered in the context of statistical physics and biology (synchronization of large populations of individuals, collective behavior of social insects, emergence of synchrony in neural networks \cite{22657695, MR2998591,1211.0299}) and for particle approximations of nonlinear partial differential equations (see e.g. \cite{MR1410117, Malrieu2003} and references therein). Usual questions for this class of models concern their large population behavior (propagation of chaos, existence of a continuous limit for the empirical measure of the particles \cite{Oelsch1984,Gartner,Lucon2011}, fluctuations around this limit \cite{Fernandez1997,MR865013,Lucon2011} and large deviations \cite{daiPra96}) as well as their long-time behavior \cite{MR2731396,Malrieu2003}.

A recent interest for similar mean-fields diffusions has been developed for the modeling of the spiking activity of neurons in a noisy environment (e.g. Hodgkin-Huxley and FitzHugh-Nagumo oscillators \cite{22657695,MR2674516,MR2998591}). In this context, $\theta_{ i}$ represents the electrical activity of one single neuron, $ \Gamma$ captures synaptic connections between neurons and the disorder $ \omega_{ i}$ models an inhomogeneous behavior between inhibition and excitation. We refer to \cite{22657695} for more details on application of these models to neuroscience.

Another illustration of weakly interacting diffusions concerns statistical physics and models of synchronization of oscillators. In particular, the Kuramoto model (see e.g. \cite{Acebron2005, Lucon2011,MR3207725, Strogatz1991}) describes the synchronization of rotators $\theta_{i}$ in $\bbR/2\pi\bbZ$ with inhomogeneous frequencies $ \omega_{ i}$:
\begin{equation}
\label{eq:Kurintro}
\dd\theta_{i}(t)= \omega_{i} \dd t + \frac{K}{N}\sum_{j=1}^{N} \sin\left(\theta_{j}-\theta_{i}\right)\dd t + \sigma\dd B_{i}(t),\ t\geq0,\  i=1,\ldots, N.
\end{equation}
The system \eqref{eq:Kurintro} is well-known to exhibit a phase transition between incoherence to synchrony as the interaction strength $K>0$ increases. We refer to the mentioned references for further details on the dynamical properties of \eqref{eq:Kurintro}.

\subsubsection{Diffusions in spatial interaction}
The motivation of going beyond pure mean-field interactions comes from the biological observation that neurons do not interact in a mean-field way (see \cite{Bullmore:2009aa, PhysRevLett.110.118101} and references therein). There has been recently a growing interest in models closer to the topology of real neuronal networks \cite{PhysRevE.85.066201,PhysRevE.85.026212,PhysRevLett.110.118101}.  Even though the analysis of such models seems to be difficult in general, it is quite natural to expect that properties valid in the pure mean-field case (the first of them being the existence of a continuous limit in an infinite population) still hold for perturbations of the mean-field case, namely for systems where the interactions are not strictly identical, but where the number of connections is sufficiently large to ensure some self-averaging as the population size increases.

In this perspective, the main motivation of \eqref{eq:odegene} comes from works of Rogers and Wille \cite{PhysRevE.54.R2193} and of Gupta, Potters and Ruffo \cite{PhysRevE.85.066201} (see also \cite{PhysRevE.66.011109,PhysRevE.82.016205}) where a spatial version of the Kuramoto model is introduced:
\begin{equation}
\label{eq:Kurintrospace}
\dd\theta_{i, t}= \omega_{i} \dd t + \frac{K}{N}\sum_{j=1}^{N} \sin\left(\theta_{j, t}-\theta_{i, t}\right) \frac{ 1}{ \left\vert x_{ i} - x_{ j}\right\vert^{ \alpha}}\dd t + \sigma\dd B_{i, t},\ t\geq0,\  i=1,\ldots, N.
\end{equation}
In \eqref{eq:Kurintrospace} (and more generally in \eqref{eq:odegene}), the particles are still interacting on the complete graph but the strength of interaction decreases polynomially with the distance between particles.

System \eqref{eq:Kurintrospace} has to be related to analogous models of statistical physics with long-range interactions (e.g. the Ising model or XY-model with interaction in $r^{ - \alpha}$, see \cite{Aizenman:1988aa} and references therein). The influence of the spatial decay in \eqref{eq:Kurintrospace} on the synchronization properties of the system (in particular the existence of critical exponents and finite size effects) does not seem to be completely understood so far (see the aforementioned references for further details).
\subsubsection{Empirical measure and McKean-Vlasov limit}
Note that in \eqref{eq:odegene}, $ \theta_{ i}\in\bbX:= \bbR^{m}$, but one should also think to the case of $\bbX$ being a compact domain of $\bbR^{ m}$, see for example $\bbX=\bbR/2\pi\bbZ$ in the Kuramoto case \eqref{eq:Kurintrospace}. For any vectors $u$ and $v$ in $\bbX$ (or $\bbY$), $ \left\vert u \right\vert$ stands for the euclidean norm of $u$ and $u\cdot v$ is the scalar product between $u$ and $v$.

We endow \eqref{eq:odegene} with an initial condition: the particles $ (\theta_{ i, 0})_{ i=1, \ldots N}$ are independent and identically distributed according to some law $ \zeta$ on $\bbX$. The disorder $ (\omega_{ i})_{ i=1, \ldots, N}$ is a sequence of independent random variables, identically distributed according to some law $ \mu$ on $\bbY$, independently from $(\theta_{ i, 0})_{ i\in \Lambda_{ N}}$ and $(B_{ i})_{ i\in \Lambda_{ N}}$. 

All the statistical information of \eqref{eq:odegene} is contained in the empirical measure of the particles, disorder and positions (that is a random process in the set of probabilities on $\bbX\times\bbY\times\bbS$):
\begin{equation}
\label{eq:emp_measure}
\nu_{ N, t}:= \frac{1}{ \left\vert \Lambda_{N} \right\vert} \sum_{j\in\Lambda_{N}}\delta_{(\theta_{j,t}, \omega_{j}, x_{j})},\ N\geq 1,\ t\in[0, T].
\end{equation}
The object of a previous work \cite{LucSta2014} was to show that, under mild assumptions on the model (see \cite{LucSta2014}, Theorem~2.18), the empirical measure $ \nu_{ N}$ converges, as $N\to\infty$, to the unique solution $ \nu$ of the following weak McKean-Vlasov equation:
\begin{equation}
\label{eq:nut}
 \int f \dd\nu_{t} = \int f \dd \nu_{ 0} + \int_{0}^{t} \int \left(\frac{1}{2}\Delta_{\theta}f + \nabla_{\theta}f \cdot \left\lbrace c + \int \Gamma(\cdot, \ttheta, \tomega) \Psi(\cdot, \tx) \nu_{ s}(\dd \ttheta, \dd \tomega, \dd \tx) \right\rbrace\right) \dd\nu_{s}\dd s,
\end{equation}
for any regular test function $(\theta, \omega, x) \mapsto f(\theta, \omega, x)$ and where the initial condition is 
\begin{equation}
\label{eq:cond_ini_meanfield}
\nu_{ 0}(\dd \theta, \dd \omega, \dd x):= \zeta(\dd \theta) \mu(\dd \omega) \dd x.
\end{equation}
The main difficulty in the analysis of \eqref{eq:nut} is that it involves the singular kernel $ \Psi$ so that even the well-posedness of such equation is unclear. We refer to \cite{LucSta2014} for more details. Equation \eqref{eq:nut} is a spatial generalization of the standard McKean-Vlasov equation in the pure mean-field case (see e.g. \cite{Oelsch1984,Gartner,Lucon2011}).

\subsubsection{Fluctuations around the mean-field limit}
The purpose of this work is to address the question of the fluctuations of the empirical measure $ \nu_{ N}$ \eqref{eq:emp_measure}  around its limit $ \nu$ \eqref{eq:nut}. We are in particular interested in the influence of the geometry on the speed of convergence of $ \nu_{ N}$ towards $ \nu$.
\begin{definition}
\label{def:first_order_fluct}
Define the fluctuation process 
\begin{equation}
\label{eq:fluct}
\eta_{N, t}:= a_{N} \left(\nu_{N, t}- \nu_{t}\right),\ N\geq 1,\ t\in[0, T]
\end{equation}
where $a_{ N}$ is an appropriate renormalizing factor.
\end{definition}
The first contribution of the paper is to exhibit a phase transition in the size of the fluctuation of \eqref{eq:odegene}: the correct scaling $a_{ N}$ in \eqref{eq:fluct} depends on the value of the spatial parameter $ \alpha\in[0, 1)$ in \eqref{eq:def_psi} with respect to the critical value $ \frac{ 1}{ 2}$.
\begin{definition}[Fluctuation renormalization]
\label{def:aN}
Fix $0\leq\alpha<1$ and define for $N\geq 1$
\begin{equation}
\label{eq:aN}
a_{N}:=\begin{cases}
N^{ \frac{1}{2}}, & \text{if } 0\leq\alpha< \frac{1}{2},\\
N^{1-\alpha}, & \text{if } \frac{1}{2}<\alpha< 1.
\end{cases}
\end{equation}
\end{definition}
\noindent Let us show intuitively that Definition~\ref{def:aN} provides the correct scaling for \eqref{eq:fluct}. The convergence of $ \nu_{ N}$ towards $ \nu$ is due to the competition of two effects: the convergence of the empirical distribution of the initial condition $ \theta_{ i, 0}$ and the Brownian motions $B_{ i}$ (which scales typically as $ \sqrt{N}$) and the convergence with respect to the spatial variable $x_{ i}$. To fix ideas, set $c \equiv 0$ and $ \Gamma\equiv 1$ in \eqref{eq:odegene}. In this case, everything boils down to the approximation of the integral $ \int_{\bbS} \frac{ 1}{ \left\vert x - x_{ i} \right\vert^{ \alpha}} \dd x$ by the Riemann sum $ \frac{ 1}{ N} \sum_{ j=1}^{ N} \frac{ 1}{ \left\vert x_{ j} - x_{ i}\right\vert^{ \alpha}}$. A simple estimate (see Lemma~\ref{lem:fluct_Psi} below) shows that the rate of this last convergence is exactly $N^{ 1- \alpha}$. Definition~\ref{def:aN} simply chooses the predominant scaling in both cases.

This intuition also suggests that when $ \alpha< \frac{ 1}{ 2}$, the randomness prevails and one should obtain Gaussian fluctuations as $N\to\infty$, whereas when $ \alpha> \frac{ 1}{ 2}$ the randomness disappears under the scaling $N^{ 1- \alpha}$ and one should obtain a deterministic limit for $ \eta_{ N}$. The main result of the paper is precisely to make this intuition rigorous: we show that the fluctuation process $ \eta_{ N}$ converges to the unique solution of a linear stochastic partial differential equation when $ \alpha< \frac{ 1}{ 2}$ (see Theorem~\ref{theo:conv_subcrit}) and that $ \eta_{ N}$ has a deterministic limit in the supercritical case $ \alpha> \frac{ 1}{ 2}$ (see Theorem~\ref{theo:conv_supercrit}).

What makes the analysis difficult here is the singularity of the spatial kernel $ \Psi$ in \eqref{eq:odegene}. An important aspect of the paper is the introduction of an auxiliary weighted fluctuation process $\cH_{ N}$ (in addition to $ \eta_{ N}$) that is necessary to capture the spatial variations of the system and to cope with the singularity of the weight $ \Psi$. We refer to Section~\ref{sec:main_results} for more details.

\subsubsection{On the quenched fluctuations} The main results (Theorems~\ref{theo:conv_subcrit} and~\ref{theo:conv_supercrit}) are \emph{averaged} with respect to the disorder $(\omega_{ i})_{ i}$. Although we did not go in this direction for the simplicity of exposition, analogous results also hold in the quenched set-up, that is when we only integrate in \eqref{eq:odegene} w.r.t. the Brownian noise and the initial condition, and \emph{not} w.r.t. the disorder. This has been carried out in a previous paper \cite{Lucon2011} in the non spatial case. We let the reader adapt a similar strategy to the present situation.

\subsubsection{Existing literature}
The use of weighted empirical processes such as $\cH_{ N}$ (see Section~\ref{sec:second_order_process}) in the context of interacting particle systems is reminiscent of previous works. One should mention in particular the articles of Kurtz and Xiong \cite{MR1705602, MR2118848} on particle approximations for nonlinear SPDEs.

The present paper uses Hilbertian techniques already introduced by Fernandez and M\'el\'eard \cite{Fernandez1997} (see also \cite{MR865013,Lucon2011}) who proved a similar central limit theorem for weakly interacting diffusions without spatial geometry (that is $ \alpha=0$ in the framework of \eqref{eq:odegene}). It is shown in \cite{Fernandez1997} that the fluctuations are governed by a linear SPDE and that the convergence holds in some appropriate weighted Sobolev space of distributions. One can see the first result of the paper (Theorem~\ref{theo:conv_subcrit}) as a generalization of the result of Fernandez and M\'el\'eard to the spatial case: when $ \alpha< \frac{ 1}{ 2}$, the spatial damping on the interactions in \eqref{eq:odegene} is not strong enough to have an effect on the behavior of the fluctuations of the system and the fluctuations remain Gaussian.

The supercritical case (Theorem~\ref{theo:conv_supercrit}) may also be related to a class of models previously studied in the literature, that is \emph{moderately interacting diffusions} (see \cite{Jourdain1998,MR876258}). This class of models also exhibits deterministic fluctuations (see \cite{Jourdain1998} p.~755), but one should point out that the precise scaling $a_{ N}$ is not explicitly known in this case (see \cite{Jourdain1998}, Remark~3.15).

\subsubsection{Comments and perspective} 
The main conclusion of Theorem~\ref{theo:conv_supercrit} is that, in the case $ \alpha> \frac{ 1}{ 2}$, the leading term in the asymptotic expansion of the empirical measure $ \nu_{ N}$ around $ \nu$ is deterministic, of order $ \frac{ 1}{ N^{ 1- \alpha}}$. A natural question would be to ask about the existence and the nature of the next term in this expansion. Concerning the dependence in the spatial variable, one easily sees that the term following $ \frac{ 1}{ N^{ 1- \alpha }}$ in the asymptotic expansion of $ \frac{ 1}{ N} \sum_{ j=1}^{ N} \frac{ 1}{ \left\vert x_{ j} \right\vert^{ \alpha}}$ is of order $ \frac{ 1}{ N}$, which is in any case smaller than the Gaussian scaling $ \frac{ 1}{ \sqrt{N}}$. Consequently, one should expect the next term in the expansion of $ \nu_{ N}$ to be Gaussian, of order $ \frac{ 1}{ \sqrt{N}}$. The precise form of this term remains unclear, though. Note that a similar analysis has been made by Oelschl\"ager in \cite{MR876258} concerning moderately interacting diffusions.

Another natural question would be to ask what happens at the critical case $ \alpha= \frac{ 1}{ 2}$ (that is when the spatial and Gaussian fluctuations are exactly of the same order). Although it is natural to think that the correct scaling is exactly $ \sqrt{N}$, the present work only provides partial answers. A closer look to the proof below shows that the scaling is \emph{at least} $ \frac{ \sqrt{N}}{ \ln N}$ (see Section~\ref{sec:fluct_crit}). To derive the correct scaling and limit for the fluctuations in this case seems to require alternative techniques.

The behavior as $ N\to\infty$ of systems similar to \eqref{eq:odegene} in the case $ \alpha>1$ is also of interest and is the object of an ongoing work. The existence of a continuous limit at the level of the law of large number is doubtful in this case.

It is likely that the results presented here should be generalized to other models of interacting diffusions, especially to systems with random inhomogeneous connectivities (in the spirit of \cite{arous1995large, ben1997symmetric}) which are of particular interest in the context of neuroscience \cite{22657695}. This will be the object of a future work.
\subsubsection{Outline of the paper}
In Section~\ref{sec:main_results}, we specify the assumptions on the model and state the main results (Theorem~\ref{theo:conv_subcrit} and~\ref{theo:conv_supercrit}). Section~\ref{sec:tightness_results} is devoted to prove tightness of the fluctuation process in an appropriate space of distributions. The identification of the limits is done in Section~\ref{sec:identification_limits}. Section~\ref{sec:proofs} contains the proofs of technical propositions.

\section{Main results}
\label{sec:main_results}
In the rest of the paper, $ \tau= \left(\theta, \omega, x\right)$ stands for an element of $\bbX\times\bbY\times\bbS$. In particular, for any $s\in[0, T]$, $i\in \Lambda_{ N}$, $ \tau_{ i,s}=(\theta_{ i,s}, \omega_{ i}, x_{ i})$. We use also the duality notation
\begin{equation}
\label{eq:cro_f_lambda}
\left\langle \lambda\, ,\, f\right\rangle:=\int_{ \bbX\times\bbY\times\bbS} f(\theta, \omega, x) \lambda(\dd\theta, \dd\omega, \dd x),
\end{equation}
where $ \lambda$ is a measure on $\bbX\times\bbY\times\bbZ$ and $(\theta, \omega, x)\mapsto f(\theta, \omega, x)$ is a test function on $\bbX\times\bbY\times\bbS$. We also use functions of \emph{two variables} defined on $(\bbX\times\bbY\times\bbS)^{ 2}$ and introduce notations for the corresponding partial integrals: for any such function $(\tau, \ttau) \mapsto g(\tau, \ttau)$, define
\begin{align}
\Big[ g\, ,\, \lambda\Big]_{ 1}(\tau)&= \Big[ g\, ,\, \lambda\Big]_{ 1}(\theta, \omega, x):=\int_{ \bbX\times\bbY\times\bbS} g(\theta, \omega, x, \ttheta, \tomega, \tx) \lambda(\dd\ttheta, \dd\tomega, \dd \tx),\label{eq:cro_1}\\
\Big[ g\, ,\, \lambda\Big]_{ 2}(\ttau)&= \Big[ g\, ,\, \lambda\Big]_{ 2}(\ttheta, \tomega, \tx):=\int_{ \bbX\times\bbY\times\bbS} g(\theta, \omega, x, \ttheta, \tomega, \tx) \lambda(\dd\theta, \dd\omega, \dd x).\label{eq:cro_2}
\end{align}
With a slight abuse of notations, we will often drop the subscript whenever it is clear from the context and write $ \Big[ \cdot\, ,\, \cdot \Big]$ instead of $ \Big[ \cdot\, ,\, \cdot \Big]_{ i}$, $i=1, 2$. With these notations, one can write the dynamics \eqref{eq:odegene} only in terms of the empirical measure~\eqref{eq:emp_measure}
\begin{equation}
\dd\theta_{i, t}= c(\theta_{i,t}, \omega_{i}) \dd t +  \Big[ \Gamma \Psi\, ,\, \nu_{ N, t}\Big](\theta_{ i, t}, \omega_{ i}, x_{ i}) \dd t + \dd B_{i, t},\ i\in \Lambda_{ N},\ t\in[0, T]\label{eq:odegene_short}
\end{equation}
where $ \Big[ \Gamma \Psi\, ,\, \nu_{ N, t}\Big]= \Big[ \Gamma \Psi\, ,\, \nu_{ N, t}\Big]_{ 1}$ stands for \eqref{eq:cro_1} with the choice of $g(\theta, \omega, x, \ttheta, \tomega, \tx):= \Gamma(\theta, \omega, \ttheta, \tomega) \Psi(x, \tx)$ and $ \lambda:= \nu_{ N,t}$. Note also that the McKean-Vlasov equation \eqref{eq:nut} may be written as
\begin{equation}
\label{eq:nutLm}
\cro{\nu_{t}}{f} = \left\langle \nu_{ 0} \, ,\,  f \right\rangle + \int_{0}^{t}\cro{\nu_{s}}{L[\nu_{s}]f}\dd s,
\end{equation}
where, for any measure $ \lambda$ on $\bbX\times\bbY\times\bbS$, the generator $L[\lambda]$ is given by
\begin{equation}
\label{eq:propagL}
L[\lambda]f(\tau):= \frac{1}{2}\Delta_{\theta}f(\tau) + \nabla_{\theta} f(\tau) \cdot \left\{c(\theta, \omega) + \Big[\Gamma\Psi, \lambda\Big](\tau)\right\}.
\end{equation}
\subsection{Decomposition of the fluctuation process}
\subsubsection{Two-particle fluctuation process}
\label{sec:second_order_process}
The main difficulty in the analysis of \eqref{eq:fluct} comes from the singularity of the spatial kernel $ \Psi$ \eqref{eq:def_psi}. In particular, if one studies the process $ \eta_{ N}$ alone, one would need to consider test functions $(\theta, \omega, x) \mapsto f(\theta, \omega, x)$ with singularities w.r.t. $x$ whereas the embeddings techniques used in this work require a minimal regularity on the test functions (see Section~\ref{sec:Hilbert_spaces} below). Hence, the main idea is to introduce an auxiliary process $\cH_{ N}$ in order to bypass the lack of regularity of the kernel $ \Psi$. This process, that we call \emph{two-particle fluctuation process}, is the key object in order to understand the influence of the positions $(x_{ i}, x_{ j})$ of the particles on the fluctuations of the whole system.
\begin{definition}
\label{def:second_order_fluct}
Define the two-particle fluctuation process by 
\begin{equation}
\label{eq:secorderfluc}
\cH_{ N, t}:= a_{ N} \left( \frac{1}{ \left\vert \Lambda_{ N} \right\vert^{ 2}} \sum_{i, j\in \Lambda_{ N}} \Psi(x_{ i}, x_{ j}) \delta_{(\tau_{ i, t}, \tau_{ j, t})} - \nu_{ N, t} \otimes \nu_{ t} \left(\Psi\cdot\right)\right)(\dd \tau, \dd \ttau),\ t\in[0, T]
\end{equation}
that is, for any test function $(\theta, \omega, x, \ttheta, \tomega, \tx)= (\tau, \ttau)\mapsto g(\tau, \ttau)$
\begin{equation}
\label{eq:HN_g}
\cro{ \cH_{ N, t}}{ g}:= a_{ N}\left( \frac{1}{ \left\vert \Lambda_{ N} \right\vert^{ 2}} \sum_{ i, j\in \Lambda_{ N}} \Psi(x_{ i}, x_{ j}) g( \tau_{ i, t}, \tau_{ j, t}) - \frac{1}{ \left\vert \Lambda_{ N} \right\vert} \sum_{ i\in \Lambda_{ N}}\int \Psi(x_{ i}, \tx)g( \tau_{ i, t}, \ttau) \nu_{ t}(\dd \ttau)\right)
\end{equation}
\end{definition}
The process $\cH_{ N}$ captures the mutual fluctuations of \emph{two particles} $(\theta_{ i}, \theta_{ j})$ instead of one. One can easily see that $\cH_{ N}$ captures the correct fluctuations induced by the space variables (especially in the supercritical case $ \alpha> \frac{ 1}{ 2}$): taking $g\equiv 1$ in \eqref{eq:HN_g}, one obtains that $ \left\langle \cH_{ N, t}\, ,\, 1\right\rangle= \frac{ 1}{ \left\vert \Lambda_{ N} \right\vert}\sum_{ i\in \Lambda_{ N}} \left\lbrace a_{ N}\left( \frac{1}{ \left\vert \Lambda_{ N} \right\vert} \sum_{j\in \Lambda_{ N}} \Psi(x_{ i}, x_{ j}) - \int_{ \bbS} \Psi(x_{ i}, \tx) \dd \tx\right)\right\rbrace$, which, by Lemma~\ref{lem:fluct_Psi} below and Definition~\ref{def:aN}, is exactly of order $1$ when $ \alpha> \frac{ 1}{ 2}$. 

What makes the use of $\cH_{ N}$ critical is that it enables to separate the issue of the singularity of the spatial kernel $ \Psi$ from the issue of the regularity of the test functions $g$: it is the process $\cH_{ N}$ itself that carries the singularity in $(x, \tx)$, not the test functions. Hence, we are allowed to (and we will in the sequel) consider test functions as regular as required in \emph{all} variables $(\theta, \omega, x)$, which is crucial for the Sobolev embeddings techniques used in the paper. Note that it also necessary to consider the process $\cH_{ N}$ in the subcritical case $ \alpha< \frac{ 1}{ 2}$ even if $\cH=\lim_{ N\to\infty}\cH_{ N}$ does not appear in the final convergence result (see Theorem~\ref{theo:conv_subcrit}).
\subsubsection{Relations between $ \eta_{ N}$ and $ \cH_{ N}$}
\label{sec:relation_HN_etaN}
It is immediate to see from \eqref{eq:fluct} and \eqref{eq:secorderfluc} that for all test functions $(\tau, \ttau) \mapsto g(\tau, \ttau)$
\begin{equation}
\label{eq:rel_etaN_HN}
\left\langle \cH_{ N, t}\, ,\, g\right\rangle= \left\langle \eta_{ N, t}\, ,\, \left\langle \nu_{ N, t}\, ,\, \Psi g\right\rangle\right\rangle, \quad t\in[0, T],
\end{equation}
where by $ \left\langle \eta_{ N, t}\, ,\, \left\langle \nu_{ N, t}\, ,\, \Psi g\right\rangle\right\rangle$ we mean $\left\langle \eta_{ N, t}(\dd \ttau)\, ,\, \left\langle \nu_{ N, t}(\dd \tau)\, ,\, \Psi(x, \tx) g(\tau, \ttau)\right\rangle\right\rangle$. A natural question would be to ask if an equality similar to \eqref{eq:rel_etaN_HN} holds in the limit as $N\to \infty$
\begin{equation}
\label{eq:rel_eta_H}
\left\langle \cH_{ t}\, ,\, g\right\rangle= \left\langle \eta_{ t}\, ,\, \left\langle \nu_{ t}\, ,\, \Psi g\right\rangle\right\rangle, \quad t\in[0, T],
\end{equation}
for any possible limits $ \eta_{ N, t}\to_{ N\to \infty} \eta_{ t}$ and $\cH_{ N, t}\to_{ N\to\infty} \cH_{ t}$. Supposing that \eqref{eq:rel_eta_H} is true, a consequence is that the limiting process $\cH$ becomes \emph{a posteriori} useless for the determination of $ \eta$: we will show in Theorem~\ref{theo:conv_subcrit} (see also Remark~\ref{rem:uncoupling_eta_H}) that, using \eqref{eq:rel_eta_H}, one can characterize the limit of $ \eta_{ N}$ as the unique solution to a linear SPDE \eqref{eq:SPDE_subcrit_eta} involving only $ \eta$ and not the auxiliary process $ \cH$.

Equality \eqref{eq:rel_eta_H} is certainly true in the case without space (that is the case considered in \cite{Fernandez1997,Lucon2011}, equivalent to $ \alpha=0$ in \eqref{eq:odegene}). One result of the paper is to show that \eqref{eq:rel_eta_H} remains true in the subcritical case $ 0\leq\alpha< \frac{ 1}{ 2}$ (although the proof for this equality is not straightforward, see Section~\ref{sec:repr_H_eta} and Proposition~\ref{prop:repres_H_eta}).
\medskip

On the contrary, we will show that \eqref{eq:rel_eta_H} does not hold in the supercritical case $ \alpha> \frac{ 1}{ 2}$: the limiting processes $ \eta$ and $\cH$ found in Theorem~\ref{theo:conv_supercrit} lead to different expressions in \eqref{eq:rel_eta_H}: when $g\equiv 1$ in \eqref{eq:rel_eta_H}, one obtains that $\left\langle \cH_{ t}\, ,\, 1\right\rangle= \left\langle \cH_{ 0}\, ,\, 1\right\rangle= \chi(\alpha)\neq0$ (see \eqref{eq:SPDE_supercrit} and \eqref{eq:H0_large_alpha}). On the other hand, $ \left\langle \nu_{ t}\, ,\, \Psi(\cdot, \tx)\right\rangle= \int_{ \bbS} \Psi(x, \tx) \dd x$ is a constant $C$ (equal to $\frac{ 2^{ \alpha}}{ 1- \alpha}$, recall Lemma~\ref{lem:fluct_Psi}) and we see from \eqref{eq:SPDE_supercrit} that $ \left\langle \eta_{ t}\, ,\, C\right\rangle=0$. 

The reason for this difference is that, when $ \alpha> \frac{ 1}{ 2}$, the predominant scaling (i.e. $ N^{ 1- \alpha}$) comes from the rate of convergence of the discrete Riemann sum $\frac{ 1}{ \left\vert \Lambda_{ N} \right\vert} \sum_{ j\in \Lambda_{ N}} \Psi(x_{ i}, x_{ j})$ towards the integral $\int_{ \bbS} \Psi(x_{ i}, \tx) \dd \tx$ (Lemma~\ref{lem:fluct_Psi}). This scaling is strongly related to the singularity of the kernel $ \Psi$: if we had replaced $ \Psi$ by a regular (e.g. $\cC^{ 1}$) function, the rate of convergence of this Riemann sum  would become $ \frac{ 1}{ N}$ and the predominant scaling of the whole system would still be $\sqrt{N}$. If one naively replaces in \eqref{eq:rel_etaN_HN} the empirical measure $ \nu_{ N}$ by its McKean-Vlasov limit $ \nu$, the singularity in space disappears: the function $\tx \mapsto \left\langle \nu_{ t}\, ,\, \Psi(\cdot, \tx) g(\cdot, \ttau)\right\rangle$ is continuous (provided $g$ is). In particular, in \eqref{eq:rel_etaN_HN}, the term $\left\langle \eta_{ N, t}\, ,\, \left\langle \nu_{t}\, ,\, \Psi g\right\rangle\right\rangle$ does not contribute to the scaling $ N^{ 1- \alpha}$. It only appears in $ \left\langle \eta_{ N, t}\, ,\, \left\langle \nu_{ N, t}- \nu_{ t}\, ,\, \Psi g\right\rangle\right\rangle$, which is a function of $\cH_{ N, t}$ but not of $ \eta_{ N, t}$ alone. Hence, there is no hope to have a closed formula for $ \eta$ when $ \alpha> \frac{ 1}{ 2}$: we obtain as $N\to\infty$ a system of coupled deterministic equations in $( \eta, \cH)$, see \eqref{eq:SPDE_supercrit}.

\subsubsection{Semi-martingale representation of the fluctuation processes}
\label{sec:semimart_repr_etaN_HN}
The starting point of the analysis is to write a semi-martingale decomposition for both processes $ \eta_{ N}$ and $ \cH_{ N}$. We see in \eqref{eq:semimart} the use of the process $\cH_{ N}$: the singular part in the semi-martingale decomposition of $ \eta_{ N}$  is completely expressed in terms of $\cH_{ N}$.
\begin{proposition}[Semi-martingale representation of $\etaN{s}$]
\label{prop:semimart}
For every test function $(\theta, \omega, x)\mapsto f(\theta, \omega, x)$, for all $t\in[0, T]$, one has
\begin{equation}
\label{eq:semimart}
\cro{\etaN{t}}{f} = \cro{\etaN{0}}{f} + \int_{0}^{t} \cro{\etaN{s}}{L[\nu_{ s}]f}\dd s+ \int_{0}^{t} \cro{\cH_{ N, s}}{ \Phi[f]} \dd s+ \cM_{ N, t}^{( \eta)}f
\end{equation}
where $L[\nu_{s}]$ is the propagator defined in \eqref{eq:propagL} and
\begin{equation}
\label{eq:Phif}
\Phi[f](\tau, \ttau)= \Phi[f]( \theta, \omega, x, \ttheta, \tomega, \tx):= \nabla_{\theta}f(\theta, \omega, x)\cdot \Gamma( \theta, \omega, \ttheta, \tomega).
\end{equation}
and where the martingale term is given by
\begin{equation}
\label{eq:def_MN_eta}
\cM_{ N, t}^{( \eta)}f := \frac{a_{N}}{ \left\vert \Lambda_{N} \right\vert} \sum_{i\in\Lambda_{N}} \int_{0}^{t} \nabla_{\theta} f(\theta_{i,s}, \omega_{i}, x_{i})\cdot \dd B_{i, s},\ t\in[0, T].
\end{equation}
\end{proposition}
\begin{proposition}[Semi-martingale representation for $\cH_{ N}$]
\label{prop:semimartI}
For any regular and bounded test function $(\tau, \ttau)\mapsto g(\tau, \ttau)$
\begin{equation}
\label{eq:semimart_HN}
\cro{ \cH_{ N, t}}{ g}= \cro{ \cH_{ N, 0}}{ g} + \int_{0}^{t} \cro{ \cH_{ N, s}}{ \mathscr{L}_{ s}g}\dd s + \int_{0}^{t} F_{ N, s}g \dd s + \int_{0}^{t} G_{ N, s}g \dd s +  \cM^{ (\cH)}_{ N, t}g,
\end{equation}
where
\begin{equation}
\mathscr{L}_{ s}g( \tau, \ttau):= \mathscr{L}_{ s}^{ (1)}g( \tau, \ttau) + \mathscr{L}^{ (2)}_{ s}g( \tau, \ttau)\label{eq:def_opL}
\end{equation}
for 
\begin{align}
 \mathscr{L}_{ s}^{ (1)}g( \tau, \ttau)&:=\frac{ 1}{ 2} \Delta_{ \theta, \ttheta} g(\tau, \ttau) +  \nabla_{ \theta}g( \tau, \ttau) \cdot\left\{c(\theta, \omega) + \left[ \Gamma \Psi, \nu_{s}\right](\tau)\right\} \nonumber\\ & \qquad \qquad \qquad +  \nabla_{ \ttheta}g( \tau, \ttau) \cdot\left\{c(\ttheta, \tomega) + \left[ \Gamma \Psi, \nu_{ s}\right](\ttau)\right\},\label{eq:def_opL1}\\
  \mathscr{L}_{ s}^{ (2)}g( \tau, \ttau)&:=\cro{ \nu_{s}}{ \Psi(\cdot, x)\nabla_{ \ttheta}g(\cdot, \tau)} \cdot\Gamma( \theta, \omega, \ttheta, \tomega),\label{eq:def_opL2}
\end{align}
where the remaining terms $F_{ N}$ and $G_{ N}$ are 
\begin{align}
F_{ N, s}g&:=\cro{ \cH_{ N, s}}{\nabla_{ \theta}g( \tau, \ttau) \cdot\left[ \Gamma \Psi, \nu_{ N, s}- \nu_{ s}\right](\tau)}\label{eq:def_UN}\\
G_{ N, s}g&:= \cro{ \cH_{ N,s}}{ \cro{\nu_{ N, s}- \nu_{ s}}{ \Psi(\cdot, x)\nabla_{ \ttheta}g(\cdot, \tau)} \cdot \Gamma( \theta, \omega, \ttheta, \tomega)}\label{eq:def_VN}
\end{align}
and the martingale part $\cM^{ (\cH)}_{ N, t}$ is given by
\begin{equation}
\label{eq:def_MNH}
\begin{split}
&\cM^{ (\cH)}_{ N, t}g:= \frac{ a_{ N}}{ \left\vert \Lambda_{ N} \right\vert^{ 2}} \sum_{ i, j\in \Lambda_{ N}} \int_{0}^{t}  \nabla_{  \ttheta} g(\tau_{ i, s}, \tau_{ j, s}) \Psi(x_{ i}, x_{ j})\cdot\dd B_{ j, s}+\\&\frac{ a_{ N}}{ \left\vert \Lambda_{ N} \right\vert} \sum_{ i\in \Lambda_{ N}} \int_{0}^{t} \left( \frac{ 1}{ \left\vert \Lambda_{ N} \right\vert} \sum_{ j\in \Lambda_{ N}} \nabla_{  \theta} g(\tau_{ i, s}, \tau_{ j, s})\Psi(x_{ i}, x_{ j}) - \int\nabla_{ \theta}g(\tau_{ i, s}, \ttau) \Psi(x_{ i}, \tx)\nu_{ s}(\dd\ttau)\right)\cdot\dd B_{ i, s}
\end{split}
\end{equation}
\end{proposition}
The proofs of Propositions~\ref{prop:semimart} and~\ref{prop:semimartI} are given in Section~\ref{sec:proof_semimart}. The whole point of the paper is to take the limit as $N\to\infty$ in the semi-martingale decompositions \eqref{eq:semimart} and \eqref{eq:semimart_HN}.
\subsection{Assumptions}
\label{sec:assumptions}
Define the following integer
\begin{equation}
P:= m+n +1=\dim (\bbX\times\bbY\times\bbS).\label{eq:P}
\end{equation}
The local dynamics term $(\theta, \omega)\mapsto c(\theta, \omega)$ in \eqref{eq:odegene} is supposed to be differentiable w.r.t. $( \theta, \omega)$ up to order $3P+9$ and satisfies a one-sided Lipschitz condition \wrt $\theta$, uniformly in $ \omega$:
\begin{equation}
\label{eq:cgrowthcond}
\sup_{ \omega\in \bbY} \left\lbrace \left(\theta-\bar\theta\right)\cdot \left(c(\theta, \omega)-c(\bar\theta, \omega)\right) \right\rbrace \leq L \left\vert \theta-\bar\theta \right\vert^{2},\ (\theta, \bar\theta)\in \bbX^{ 2}
\end{equation}
for some constant $L$. We also suppose that $ (\theta, \omega) \mapsto c(\theta, \omega)$ and its derivatives up to order $3P+9$ are bounded in $ \theta$ with polynomial bound in $ \omega$: there exist $\iota\geq1$ and $C>0$ such that for all differential operator $D$ in $( \theta, \omega)$ of order smaller than $3P+9$
\begin{equation}
\label{eq:polgrowthc}
\sup_{ \theta\in\bbX}\left\vert Dc(\theta, \omega) \right\vert\leq C \left(1+ \left\vert \omega \right\vert^{\iota}\right),\ \omega\in \bbY.
\end{equation}
The interaction term $\Gamma$ in \eqref{eq:odegene} is supposed to be bounded by $\Ninf{\Gamma}$, globally Lipschitz-continuous on $ \left(\bbX\times\bbY\right)^{2}$, with a Lipschitz constant $\N{\Gamma}_{Lip}$ and differentiable in all variables up to order $3P+9$ with bounded derivatives. We also require that 
\begin{equation}
\sum_{ \left\vert k \right\vert\leq 2P+3} \sup_{ \theta, \omega, \bar \omega} \left(\int_{ \bbX} \left\vert D_{ \theta, \omega}^{ k} \Gamma( \bar \theta, \bar \omega, \theta, \omega)\right\vert  \dd \bar\theta\right)<+\infty.
\end{equation}
Define the following exponents (where $ \iota$ is given by \eqref{eq:polgrowthc})
\begin{align}
\gamma&:= \max \left(P, \left\lfloor \frac{ n}{ 2}\right\rfloor + \iota\right)+1,\label{eq:def_beta}\\
\underline \kappa&:= m+1\text{ and }\bar \kappa:= \underline \kappa + 2 \gamma,\label{eq:kappas}\\
\underline \iota&:= n+2 \iota+1\text{ and }\bar \iota:= \underline \iota + 2 \gamma.\label{eq:iotas}
\end{align}
The law $ \zeta$ of the particles $( \theta_{ i, 0})_{ i\in \Lambda_{ N}}$ is assumed to be absolutely continuous with respect to the Lebesgue measure on $\bbX=\bbR^{ m}$ and its density (that we also denote by $ \theta \mapsto \zeta( \theta)$ with a slight abuse of notation) satisfies the integrability condition
\begin{equation}
\label{eq:integ_zeta}
\exists p>m,\ \int_{ \bbX} \zeta(\theta)^{ p} \dd \theta<+\infty.
\end{equation}
We also assume that $ \zeta$ and $ \mu$ satisfy the moment conditions
\begin{equation}
\label{eq:moments_mu_xi0}
\int_{ \bbX} \left\vert \theta \right\vert^{ 8\bar \kappa} \zeta(\dd \theta)< +\infty \text{ and } \int_{ \bbY} \left\vert \omega \right\vert^{ 8 \bar \iota} \mu(\dd \omega)< +\infty,
\end{equation}
where $(\bar \kappa, \bar \iota)$ are given in \eqref{eq:kappas} and \eqref{eq:iotas}. In what follows, we denote as $\bE(\cdot)$ the expectation with respect to the initial condition $ \theta_{ i, 0}$, the Brownian motions $(B_{ i})$ and the disorder $(\omega_{ i})$ (i.e. we work in the averaged model).

\subsection{Estimates on the McKean-Vlasov equation}
We first recall some results concerning the continuous limit \eqref{eq:nut}. The crucial object here is the nonlinear process associated to the McKean-Vlasov equation \eqref{eq:nut} \cite{SznitSflour,LucSta2014}. The following result is a direct consequence of \cite{LucSta2014}, Section~3.2:
\begin{proposition}
\label{prop:nonlin}
Under the assumptions made in Section~\ref{sec:assumptions}, there is pathwise existence and uniqueness of the solution $(\bar \theta_{ t}^{ x, \omega}, \omega, x)_{ t\in[0, T]}$ of the nonlinear system
\begin{equation}
\label{eq:nonlin}
\left\{\begin{split}
\bar\theta^{x, \omega}_{ t} &= \theta_{ 0} + \int_{0}^{t} \left(c(\bar\theta^{x, \omega}_{ s}, \omega) + \Big[\Gamma\Psi, \lambda_{s}\Big](\bar\theta^{x, \omega}_{ s}, \omega, x)\right)\dd s + B_{ t},\\
\lambda_{t}&=\lambda_{t}^{x, \omega}(\dd\theta) \mu(\dd\omega)\dd x,\\
\lambda_{t}^{x, \omega}&(\dd\theta) \mu(\dd\omega)\text{ is the law of }(\bar \theta_{ t}^{ x, \omega}, \omega),
\end{split}\right.\ t\in[0, T].
\end{equation}
satisfying $\sup_{ t\leq T} \int \left(\left\vert \theta \right\vert^{ 8\bar\kappa} \vee \left\vert \omega \right\vert^{ 8\bar\iota}\right) \lambda_{ t}({\rm d} \theta, {\rm d} \omega, \dd x)< +\infty$, where $ \bar \kappa$ and $ \bar \iota$ are given in \eqref{eq:kappas} and \eqref{eq:iotas}.
\end{proposition}
\begin{proposition}
\label{prop:decomp_nu}
Under the hypothesis made in Section~\ref{sec:assumptions}, there exists a unique weak solution $t \mapsto \nu_{ t}$ in $\mathcal{C}([0, T], \mathcal{M}_{ 1}(\mathbf{R}^{ m}\times\mathbf{R}^{ n}))$  to \eqref{eq:nut} satisfying \[\sup_{ t\leq T} \int \left(\left\vert \theta \right\vert^{ 8\bar \kappa} \vee \left\vert \omega \right\vert^{ 8\bar \iota}\right) \nu_{ t}({\rm d} \theta, {\rm d} \omega, \dd x)< +\infty.\]
Moreover, there exists a continuous measure-valued process $t\mapsto\xi_{ t}$ on $ \bbX\times\bbY$ such that
\begin{equation}
\label{eq:decomp_nu_xi}
\nu_{t}(\dd\theta, \dd\omega, \dd x)= \xi_{ t}(\dd \theta, \dd\omega)\dd x,\quad t\in[0, T].
\end{equation}
The process $ \xi$ admits a regular density $(t, \theta, \omega) \mapsto p_{ t}(\theta, \omega)$ with respect to $\dd \theta\otimes \mu(\dd \omega)$
\begin{equation}
\label{eq:xit_pt_mu}
\xi_{ t}(\dd \theta, \dd \omega)= p_{ t}(\theta, \omega)\dd \theta\mu(\dd \omega),\ t\in[0, T],
\end{equation}
and this density $p$ satisfies the a priori estimates:
\begin{align}
0\leq p_{ t}(\theta, \omega)&\leq \frac{ 1+ \left\vert \omega \right\vert^{ \iota}}{ t^{ \alpha_{ 0}}},\quad \omega\in\bbY,\ t\in(0, T]\label{eq:est_pt}\\
\left\vert \div_{ \theta}p_{ t}(\theta, \omega) \right\vert&\leq \frac{ 1+ \left\vert \omega \right\vert^{ \iota}}{ t^{ \alpha_{ 0}+ \frac{ 1}{ 2}}},\quad \omega\in\bbY,\ t\in(0, T],\label{eq:est_div_pt}
\end{align} for some $0< \alpha_{ 0}< \frac{ 1}{ 2}$.
\end{proposition}
The proof of Proposition~\ref{prop:decomp_nu} is postponed to Section~\ref{app:decomp_nu}.

\begin{proposition}
\label{prop:moment_particles}
The particle system \eqref{eq:odegene} and the non-linear process \eqref{eq:nonlin} satisfy the moment conditions
\begin{equation}
\sup_{ N\geq1}\sup_{ i\in \Lambda_{ N}}\bE \left( \sup_{ s\leq T}\left\vert \theta_{ i, s} \right\vert^{ 8 \bar \kappa}\right) <+\infty \text{ and } \bE \left(\sup_{ s\leq T} \sup_{ x\in \bbS}\left\vert \bar\theta_{ s}^{ x} \right\vert^{ 8 \bar \kappa}\right)<+\infty,
\end{equation}
where $ \bar \kappa$ is given in \eqref{eq:kappas}.
\end{proposition}
\begin{proof}[Proof of Proposition~\ref{prop:moment_particles}]
The estimate on the nonlinear process $ \bar \theta$ is a direct consequence of Proposition~\ref{prop:nonlin}. The same estimate for the particle system is standard and left to the reader.
\end{proof}
\subsection{Fluctuations results}
We prove the convergence of the fluctuations processes in appropriate Sobolev weighted spaces of distributions $\bV$ and $\bW$ that are defined in the next section, using Hilbertian techniques developed by Fernandez and M\'el\'eard in \cite{Fernandez1997}. We only state the result here and refer to Section~\ref{sec:Hilbert_spaces} for precise definitions of these spaces.
\subsubsection{Fluctuation result in the subcritical case $ \alpha< \frac{ 1}{ 2}$}
The first main result of the paper is the following (recall the definition of the propagator $L[ \nu_{ s}]$ in \eqref{eq:propagL}):
\begin{theorem}
\label{theo:conv_subcrit}
Suppose $ \alpha< \frac{ 1}{ 2}$. Under the hypothesis of Section~\ref{sec:assumptions}, the random process $(\eta_{ N})_{ N\geq1}$ defined by \eqref{eq:fluct} converges in law as $N\to\infty$ to $\eta\in\cC([0, T], \bV_{ -2(P+2)}^{ \underline\kappa + \gamma, \underline\iota + \gamma})$ solution in $\bV_{ -3(P+2)}^{ \underline\kappa, \underline\iota}$ to the linear stochastic partial differential equation
\begin{equation}
\label{eq:SPDE_subcrit_eta}
\eta_{ t} = \eta_{ 0} + \int_{0}^{t} \cL_{ s}^{ \ast} \eta_{ s} \dd s + \cM_{ t}^{ (\eta)}, \quad t\in[0, T]
\end{equation}
where 
\begin{equation}
\label{eq:scL}
\cL_{s}f(\theta, \omega, x):= L[\nu_{s}]f(\theta, \omega, x) + \cro{\nu_{s}}{\nabla_{\theta}f(\cdot) \cdot \Gamma(\cdot, \cdot, \theta, \omega)\Psi(\cdot, x)},
\end{equation}
$ \eta_{ 0}$ is a Gaussian process with explicit covariance given in Proposition~\ref{prop:initial_cond_HN} and $\cM^{ ( \eta)}$ is an explicit martingale given in Definition~\ref{def:covariance_WH}. $ \eta_{ 0}$ and $\cM^{ ( \eta)}$ are independent.
\end{theorem}

\subsubsection{Convergence result when $ \alpha>\frac{ 1}{ 2}$}
\label{sec:intro_supercrit}
The second main result concerns the fluctuations of \eqref{eq:odegene} in the supercritical case. For technical reasons (see Section~\ref{sec:ident_supercrit}), in addition to the assumptions of Section~\ref{sec:assumptions}, we restrict here to the case where $\bbX$ is no longer $\bbR^{ m}$ but a compact domain of $\bbR^{ m}$. The example we have particularly in mind here is the Kuramoto case where $\bbX= \bbR/2\pi\bbZ$. We also suppose that the support of the distribution $ \mu$ of the disorder is compact. These further assumptions are made in order to ensure uniqueness of a solution to \eqref{eq:SPDE_supercrit} below, due to the nonstandard nature of the operator $\mathscr{L}_{ s}$ (see Section~\ref{sec:ident_supercrit}).
\begin{theorem}
\label{theo:conv_supercrit}
Suppose $ \alpha> \frac{ 1}{ 2}$. Under the hypothesis of Sections~\ref{sec:assumptions} and~\ref{sec:intro_supercrit}, the random process $(\eta_{ N}, \cH_{ N})_{ N\geq1}$ converges in law as $N\to\infty$ to $(\eta, \cH)$, solution in $\cC([0, T], \bV^{ \underline\kappa + \gamma, \underline\iota + \gamma}_{ -3(P+2)}\oplus \bW^{ \underline\kappa + \gamma, \underline\iota + \gamma}_{ -3(P+2)})$ to the system of coupled deterministic equations
\begin{equation}
\label{eq:SPDE_supercrit}
\left\{
\begin{split}
\eta_{ t} &= \int_{0}^{t} L[ \nu_{ s}]^{ \ast} \eta_{ s}\dd s + \int_{0}^{t} \Phi^{ \ast} \cH_{ s}\dd s,\\
\cH_{ t}&= \cH_{ 0} + \int_{0}^{t} \mathscr{L}_{ s}^{ \ast} \cH_{ s} \dd s,
\end{split}\quad t\in[0, T],\right.
\end{equation}
where $\cH_{ 0}$ is a nontrivial initial condition defined in Proposition~\ref{prop:initial_cond_HN} and $L[ \nu_{ s}]^{ \ast}$ (respectively $ \Phi^{ \ast}$ and $\mathscr{L}_{ s}^{ \ast}$) is the dual of the propagator $L[ \nu_{ s}]$ defined in \eqref{eq:propagL} (respectively of the linear form $\Phi$ defined in \eqref{eq:Phif} and the linear operator $\mathscr{L}_{ s}$ defined in \eqref{eq:def_opL}).
\end{theorem}
\begin{remark}
The assumptions of Theorems~\ref{theo:conv_subcrit} and \ref{theo:conv_supercrit} do not cover the case of FitzHugh-Nagumo oscillators (that is when only a one-sided Lipchitz continuity on $c$ as in \eqref{eq:cgrowthcond} and polynomial bound on $c$ is required). A careful reading of the following shows that the tightness results (Theorems~\ref{theo:HN_tight} and~\ref{theo:etaN_tight}) are indeed true in the FitzHugh-Nagumo case. The restrictive conditions of the paper are only required for the uniqueness of the limits.
\end{remark}
\subsubsection{Comments on the critical fluctuations $ \alpha= \frac{ 1}{ 2}$}
\label{sec:fluct_crit}
In the critical case $ \alpha= \frac{ 1}{ 2}$, it is expected that the correct scaling is $a_{ N}= \sqrt{N}$ although the techniques used in this work do not seem to make this intuition rigorous. Nevertheless, a closer look at the proofs below shows the following partial result
\begin{proposition}[Critical fluctuations]
When $ \alpha= \frac{ 1}{ 2}$, under the assumptions of Section~\ref{sec:assumptions}, the following convergence holds in $\cC([0, T], \bV_{ -2(P+2)}^{ \underline\kappa + \gamma, \underline\iota + \gamma})$,
\begin{equation}
\frac{ \sqrt{N}}{ \ln N} \left( \nu_{ N} - \nu\right) \to 0,\ \text{as $N\to \infty$}.
\end{equation}
\end{proposition}
\section{Tightness results}
\label{sec:tightness_results}
This section is devoted to the tightness of $ \eta_{ N}$ and $\cH_{ N}$, based on their semi-martingale decomposition \eqref{eq:semimart} and \eqref{eq:semimart_HN}.
\subsection{First estimates on $ \Psi$}
\begin{lemma}
\label{lem:Psi}
There exists a constant $C$ that only depends on $ \alpha$ such for all $(x, y, z)$
\begin{equation}
\label{eq:regPsi}
\left\vert \Psi(x, y) - \Psi(x, z) \right\vert \leq C  d(y,z) \left( \frac{ 1}{ d( x, y)^{ \alpha+1}} + \frac{ 1}{d(x,z)^{ \alpha+1}}\right).
\end{equation}
\end{lemma}
\begin{proof}[Proof of Lemma~\ref{lem:Psi}]
Straightforward (see for example \cite{Godinho:2013aa}, Lemma~2.5 for a proof).
\end{proof}

\begin{lemma}
\label{lem:Riem}
For all $ \beta>0$, there exists a constant $C>0$ (that only depends on $\beta$), such that for all $N\geq1$, for all $i \in \Lambda_{ N}$,
\begin{equation}
\label{eq:Riem}
\sum_{j \in \Lambda_{ N},\ j \neq i} d\left(\frac{ j }{ 2N} , \frac{ i}{ 2N}\right)^{ - \beta}\leq C\cdot\begin{cases} N& \text{ if $ 0<\beta<1$},\\
N\ln N& \text{ if $\beta=1$},\\
N^{ \beta}& \text{ if $\beta>1$}.
\end{cases}
\end{equation}
\end{lemma}
\begin{proof}[Proof of Lemma~\ref{lem:Riem}]
This lemma has been proven in \cite{LucSta2014}, Lemma~6.1.
\end{proof}
\begin{remark}
\label{rem:Psi_bounded}
An easy consequence of Lemma ~\ref{lem:Riem} that will be continuously used in the following is that, since $ \alpha\in[0, 1)$,
\begin{equation}
\label{eq:Psi_bounded}
\sup_{ N\geq 1} \sup_{ i\in \Lambda_{ N}} \frac{ 1}{ \left\vert \Lambda_{ N} \right\vert}\sum_{ j\in \Lambda_{ N}} \Psi(x_{ i}, x_{ j})< +\infty.
\end{equation}
\end{remark}
The following (very simple) lemma is at the core of the difficulties of the paper: the rate of convergence, as $N\to \infty$, of the Riemann sum associated to the function $ x \mapsto \frac{ 1}{ x^{ \alpha}}$ to its integral is $N^{ 1- \alpha}$ which is in particular smaller than the Gaussian scaling $ \sqrt{N}$, when $ \alpha> \frac{ 1}{ 2}$.
\begin{lemma}
\label{lem:fluct_Psi}
For all $ \alpha \in [0, 1)$, $ N\geq1$ and $i\in \Lambda_{ N}$, the quantity
\begin{equation}
N^{ 1- \alpha}\bigg( \frac{1}{ \left\vert \Lambda_{ N} \right\vert} \sum_{ j\in \Lambda_{ N}} \Psi(x_{ i}, x_{ j})  - \int_{ \bbS} \Psi(x_{ i}, \tx) \dd \tx\bigg)
\end{equation}
is independent of $i\in \Lambda_{ N}$ and converges as $N\to \infty$ to a constant $ \chi( \alpha)\neq0$ (that depends only on $ \alpha$).
\end{lemma}
\begin{proof}[Proof of Lemma~\ref{lem:fluct_Psi}]
From the definition of $d(\cdot, \cdot)$ in \eqref{eq:dist_d}, a direct calculation shows that $\int_{ \bbS} \Psi(x_{ i}, \tx) \dd \tx$ is actually independent of $i\in \Lambda_{ N}$ and equal to $ \frac{ 2^{ \alpha}}{ 1- \alpha}$. Moreover, one directly sees from \eqref{eq:dist_d} that $ \frac{1}{ \left\vert \Lambda_{ N} \right\vert} \sum_{ j\in \Lambda_{ N}} \Psi(x_{ i}, x_{ j})= \frac{ 2^{ \alpha}}{ N^{ 1- \alpha}} \left(\sum_{ k=1}^{ N} \frac{ 1}{ k^{ \alpha}}\right) - \frac{ 2^{ \alpha-1}}{ N}$. By a usual comparison with integrals,
\begin{align*}
\sum_{ k=1}^{ N} \frac{ 1}{ k^{ \alpha}}&= \int_{1}^{N} \frac{ 1}{ t^{ \alpha}}\dd t + \frac{ 1}{ 2} + \frac{ 1}{ 2 N^{ \alpha}} - \int_{0}^{1} \left(u- \frac{ 1}{ 2}\right)\sum_{ k=1}^{ N-1} \frac{ \alpha}{ (u+k)^{ \alpha+1}}\dd u,\\
&= \frac{ N^{ 1- \alpha}}{ 1- \alpha} - \frac{ 1}{ 1- \alpha} + \frac{ 1}{ 2} + \frac{ 1}{ 2 N^{ \alpha}} - \int_{0}^{1} \left(u- \frac{ 1}{ 2}\right)\sum_{ k=1}^{ N-1} \frac{ \alpha}{ (u+k)^{ \alpha+1}}\dd u.
\end{align*}
Defining $ C(\alpha):= - \frac{ 1}{ 1- \alpha} + \frac{ 1}{ 2} - \int_{0}^{1} \left(u- \frac{ 1}{ 2}\right)\sum_{ k=1}^{+ \infty} \frac{ \alpha}{ (u+k)^{ \alpha+1}}\dd u\neq 0$, one obtains
\[\sum_{ k=1}^{ N} \frac{ 1}{ k^{ \alpha}}= \frac{ N^{ 1- \alpha}}{ 1- \alpha} + C( \alpha) + \frac{ 1}{ 2 N^{ \alpha}} + \int_{0}^{1} \left(u- \frac{ 1}{ 2}\right)\sum_{ k=N}^{ +\infty} \frac{ \alpha}{ (u+k)^{ \alpha+1}}\dd u.\] Consequently,
\begin{align*}
\frac{1}{ \left\vert \Lambda_{ N} \right\vert} \sum_{ j\in \Lambda_{ N}} \Psi(x_{ i}, x_{ j}) &= \frac{ 2^{ \alpha}}{ N^{ 1- \alpha}} \left(\frac{ N^{ 1- \alpha}}{ 1- \alpha} + C( \alpha) + r_{ N}\right) - \frac{ 2^{ \alpha-1}}{ N},\\
&=  \frac{ 2^{ \alpha}}{ 1- \alpha} + \frac{ 2^{ \alpha}C( \alpha)}{ N^{ 1- \alpha}}  + \frac{ 2^{ \alpha}}{ N^{ 1- \alpha}} r_{ N} - \frac{ 2^{ \alpha-1}}{ N},
\end{align*}where $r_{ N}\to0$ as $N\to\infty$. Lemma~\ref{lem:fluct_Psi} follows with $ \chi( \alpha):= 2^{ \alpha}C( \alpha)$.
\end{proof}

\subsection{Estimates on the nonlinear process}
\label{sec:nonlinear_process}
The tightness result is based on a coupling argument: for all $i\in \Lambda_{ N}$, define the nonlinear process $\btheta_{i}$ associated to the diffusion $\theta_{i}$, with the same Brownian motion $B_{i}$, initial condition $\theta_{i, 0}$, disorder $\omega_{i}$ and position $x_{i}$:
\begin{equation}
\label{eq:nonlinproc}
\btheta_{i,t} = \theta_{i, 0} + \int_{0}^{t} \left\{c(\btheta_{i,s}, \omega_{i}) + \Big[\Gamma\Psi, \nu_{s}\Big](\btheta_{i, s}, \omega_{i}, x_{i})\right\} \dd s + B_{i,t},\ t\in[0, T],\ i\in \Lambda_{ N}.
\end{equation}
\begin{remark}
\label{rem:Psix}
Note that, due to the rotational invariance of $\bbS$, the function $(\theta, \omega, x)\mapsto \left[ \Gamma \Psi, \nu_{ t}\right](\theta, \omega, x)$ is actually independent of $x$ (see \eqref{eq:GamPsix}), so that the definition of $ \btheta_{ i}$ (contrary to the the microscopic particle system $ \theta_{ i, N}$ in \eqref{eq:odegene_short}) does not depend on the position $x_{ i}$. But we keep this dependence for consistency of notations.
\end{remark}
The following proposition quantifies how good the approximation of $ \theta_{ i}$ by $ \bar\theta_{ i}$ is:
\begin{proposition}
\label{prop:ttaVSnonlin} Recall the definition of the scaling parameter $a_{ N}$ in \eqref{eq:aN}. There exists a constant $C>0$ only depending on $T$, $c$, $\Gamma$ and $\Psi$ such that
\begin{equation}
\label{eq:ttaVSnonlin}
\bE \left(\max_{i\in\Lambda_{N}}\sup_{t\leq T} \left\vert \theta_{i,t}-\btheta_{i, t} \right\vert^{8}\right)\leq \frac{C}{a_{N}^{8}},\ N\geq1.
\end{equation}
\end{proposition}
Proof of Proposition~\ref{prop:ttaVSnonlin} is given in Section~\ref{sec:proof_tta_vs_nonlin}.

\subsection{Weighted-Sobolev spaces}
\label{sec:Hilbert_spaces}
We introduce weighted spaces of distributions \cite{Fernandez1997,Lucon2011} for the study of processes $ \eta_{ N}$ and $\cH_{ N}$. The definitions are given in two steps, since $ \eta_{ N}$ acts on test functions of one variable $ \tau \mapsto f(\tau)$ whereas $ \cH_{ N}$ acts on test functions of two variables $ (\tau, \ttau) \mapsto g(\tau, \ttau)$. The only thing that differs from the usual Sobolev norm is the presence of a polynomial weight made necessary for the control in \eqref{eq:odegene} of the term $(\theta, \omega) \mapsto c(\theta, \omega)$ on $\bbX\times \bbY$. 
\subsubsection{Sobolev spaces on $\bbX\times\bbY\times\bbS$}
\label{sec:weighted_sobolev_spaces_one_var}
Recall the notation $\tau=(\theta, \omega, x)\in \bbX\times\bbY\times\bbS$ and the definition of $ \kappa$, $ \iota$ in \eqref{eq:polgrowthc} and of $P$ in \eqref{eq:P}. For all $l,l^{ \prime}\geq0$ and $p\geq 0$, define the norm $\N{\cdot}_{p, l, l^{ \prime}}$ over test functions $\tau\mapsto f(\tau)$ by
\begin{equation}
\label{eq:normj_one_var}
\N{f}_{p, l, l^{ \prime}}:= \left(\sum_{ \left\vert k \right\vert \leq p} \int  \frac{ \left\vert D_{\tau}^{k}f(\tau) \right\vert^{2}}{( 1+ \left\vert \theta \right\vert^{ l} + \left\vert \omega \right\vert^{ l^{ \prime}})^{ 2}}\dd\tau\right)^{ \frac{1}{2}},
\end{equation}
where, if $k=(k_{ 1}, \ldots, k_{ P})$ and $ \tau=( \theta, \omega, x)= ( \theta^{ (1)}, \ldots, \theta^{ (m)}, \omega^{ (1)}, \ldots, \omega^{ (n)}, x)$, we define $ \left\vert k \right\vert:= \sum_{ i=1}^{ P} k_{ i}$ and $D^{ k}_{ \tau}f( \tau):= \partial^{ k_{ 1}}_{ \theta^{ (1)}}\ldots \partial^{ k_{ m}}_{ \theta^{ (m)}} \partial^{ k_{ m+1}}_{ \omega^{ (1)}} \ldots \partial^{ k_{ m+n}}_{ \omega^{ (n)}}\partial^{ k_{ P}}_{ x}f( \tau)$. Denote by $\bV^{ l, l^{ \prime}}_{p}$ the completion of regular functions with compact support  under the norm $\N{\cdot}_{p, l, l^{ \prime}}$; $ \left(\bV^{l, l^{ \prime}}_{p}, \N{\cdot}_{p, l, l^{ \prime}}\right)$ is an Hilbert space. Denote by $\bV^{ l, l^{ \prime}}_{-p,}$ its dual, with norm $\N{\cdot}_{-p, l, l^{ \prime}}$. 

Denote as $\cC^{l, l^{ \prime}}_{ p}$ the Banach space of functions $f$ with continuous derivatives up to order $p$ such that $\lim_{ \left\vert \theta \right\vert + \left\vert \omega \right\vert\to\infty}\sup_{ x\in \bbS}\frac{ \left\vert D_{\tau}^{k}f(\theta, \omega, x)\right\vert}{1+ \left\vert \theta \right\vert^{ l} + \left\vert \omega \right\vert^{ l^{ \prime}}}=0$, for all $ \left\vert k \right\vert\leq p$. Endow this space with the norm $ \N{\cdot}_{C^{ l, l^{ \prime}}_{ p}}$ given by
\begin{equation}
\label{eq:normsupj_one_var}
\N{f}_{C^{ l, l^{ \prime}}_{ p}} :=  \sum_{\left\vert k \right\vert\leq p} \sup_{\tau} \frac{ \left\vert D_{\tau}^{k}f(\tau)\right\vert}{1+ \left\vert \theta \right\vert^{ l} + \left\vert \omega \right\vert^{ l^{ \prime}} }.
\end{equation} 
One has the continuous embedding (see \cite{Adams2003} or \cite{Fernandez1997} Section~2.1):
\begin{align}
\bV^{l, l^{ \prime}}_{q+r} &\hookrightarrow C^{l, l^{ \prime}}_{ r},\  r\geq0,\ q> P/2,\ l,l^{ \prime}\geq0,\label{eq:embed_one_var}\\
C^{l, l^{ \prime}}_{ r} &\hookrightarrow \bV^{k, k^{ \prime}}_{r}, \ r\geq0,\ k> \frac{ m}{ 2} + l,\ k^{ \prime}> \frac{ n}{ 2} + l^{ \prime},\label{eq:embed_cont_one_var}
\end{align}
and the Hilbert-Schmidt embedding
\begin{equation}
\label{eq:comp_embed_one_var}
\bV^{l, l^{ \prime}}_{q+r} \hookrightarrow \bV^{l+k, l^{ \prime}+k}_{r}, \quad r\geq0,\ q> P/2,\ k> P/2,\ l,l^{ \prime}\geq0.
\end{equation}
The corresponding Hilbert-Schmidt embedding holds for the dual spaces:
\begin{equation}
\label{eq:comp_embed_dual_one_var}
\bV^{l+k, l^{ \prime}+k}_{-r} \hookrightarrow \bV^{l, l^{ \prime}}_{-(q+r)}, \quad r\geq0,\ q> P/2,\ k> P/2,\ l,l^{ \prime}\geq0.
\end{equation}

\subsubsection{Sobolev spaces on $(\bbX\times\bbY\times\bbS)^{ 2}$}
\label{sec:weighted_sobolev_spaces_two_var}
We now define similar Sobolev spaces for test functions of two variables $(\tau, \ttau)\mapsto g(\tau, \ttau)$. For all $l,l^{ \prime}\geq0$ and $p\geq 0$, define:
\begin{equation}
\label{eq:normj}
\N{g}_{p, l, l^{ \prime}}:= \left(\sum_{ \left\vert k \right\vert + \left\vert \tk \right\vert\leq p} \int  \frac{ \left\vert D_{\tau}^{k}D_{ \ttau}^{ \tk}g(\tau, \ttau) \right\vert^{2}}{w( \tau, \ttau, l, l^{ \prime})^{ 2}}\dd\tau \dd \ttau\right)^{ \frac{1}{2}},
\end{equation}
and
\begin{equation}
\label{eq:normsupj}
\N{g}_{C^{ l, l^{ \prime}}_{ p}} :=  \sum_{\left\vert k \right\vert + \left\vert \tk \right\vert\leq p} \sup_{\tau, \ttau} \frac{ \left\vert D_{\tau}^{k}D_{ \ttau}^{ \tk}g(\tau, \ttau)\right\vert}{w(\tau, \ttau, l, l^{ \prime})},
\end{equation} 
where the weight $w$ is given by
\begin{equation}
\label{eq:def_w}
w(\tau, \ttau, l, l^{ \prime}):= \left(1+ \left\vert \theta \right\vert^{ l} + \vert \ttheta \vert^{ l}\right) \left(1+ \left\vert \omega \right\vert^{ l^{ \prime}} + \vert \tomega \vert^{ l^{ \prime}}\right).
\end{equation} 
Define similarly $ \left(\bW^{l, l^{ \prime}}_{p}, \N{\cdot}_{p, l, l^{ \prime}}\right)$, $(\bW^{ l, l^{ \prime}}_{-p,}, \N{\cdot}_{-p, l, l^{ \prime}})$ and $(\cC^{l, l^{ \prime}}_{ p}, \N{\cdot}_{C^{ l, l^{ \prime}}_{ p}})$ the corresponding spaces. We use for simplicity the same notations for the norms in Sections~\ref{sec:weighted_sobolev_spaces_one_var} and~\ref{sec:weighted_sobolev_spaces_two_var}, since the distinction between both cases will be often clear from the context. Similarly, one has the embeddings:
\begin{align}
\bW^{l, l^{ \prime}}_{q+r} &\hookrightarrow C^{l, l^{ \prime}}_{ r},\quad r\geq0,\ q> P,\ l,l^{ \prime}\geq0,\label{eq:embed}\\
C^{l, l^{ \prime}}_{ r} &\hookrightarrow \bW^{k, k^{ \prime}}_{r}, \ r\geq0,\ k>  m + l,\ k^{ \prime}>  n + l^{ \prime},\label{eq:embed_cont}\\
\bW^{l, l^{ \prime}}_{q+r} &\hookrightarrow \bW^{l+k, l^{ \prime}+k}_{r}, \quad r\geq0,\ q> P,\ k> P,\ l,l^{ \prime}\geq0,\label{eq:comp_embed}\\
\bW^{l+k, l^{ \prime}+k}_{-r} &\hookrightarrow \bW^{l, l^{ \prime}}_{-(q+r)}, \quad r\geq0,\ q> P,\ k> P\ l,l^{ \prime}\geq0.\label{eq:comp_embed_dual}
\end{align}
For the proof of tightness of $(\eta_{ N}, \cH_{ N})$ and the identification of its limit, we will need in Section~\ref{sec:identification_limits} several instances of the spaces $\bV$ and $\bW$, for different choices of the Sobolev parameters $(p, l, l^{ \prime})$. In the rest of Section~\ref{sec:tightness_results}, we work with general parameters and the precise values of the parameters will be specified in Section~\ref{sec:identification_limits}.
\begin{definition}
\label{def:kappas}
Fix two integers $ \kappa_{ 0}\geq 0$ and $ \iota_{ 0}\geq 0$ and define (recall \eqref{eq:def_beta}) 
\begin{equation}
\label{eq:kappas_01}
(\kappa_{ 1}, \iota_{ 1}):= (\kappa_{ 0}+ \gamma, \iota_{ 0} + \gamma).
\end{equation}
\end{definition}
Note that the definition of $( \kappa_{ 0}, \kappa_{ 1}, \iota_{ 0}, \iota_{ 1})$ ensures that the Sobolev embeddings \eqref{eq:comp_embed_one_var},  \eqref{eq:comp_embed_dual_one_var}, \eqref{eq:comp_embed} and \eqref{eq:comp_embed_dual} are true with $l= \kappa_{ 0}$, $l^{ \prime}= \iota_{ 0}$ and $k= \gamma$ (provided the regularity index $q$ is such that $q>P$). 

\subsubsection{Continuity of linear forms}
\label{sec:continuity_linear_forms}
We place ourselves in the context of test functions of two variables (Section~\ref{sec:weighted_sobolev_spaces_two_var}). Define the following linear forms describing the variations of the test functions with respect to either $(\theta, \omega)$ or to the space variable $x$. 
\begin{definition}[Continuity in $ \theta$]
\label{def:linearforms_theta}
For fixed $\tau= (\theta, \omega, x)$, $\ttau= (\ttheta, \tomega, \tx)\in\bbX\times\bbY\times\bbS$, and $ \pi, \tpi\in \bbX$, define the linear forms:
\begin{align}
\label{eq:def_cR}\cR_{ \tau, \ttau, \pi, \tpi}(g)&:= g(\theta, \omega, x, \ttheta, \tomega, \tx) - g(\pi, \omega, x, \tpi, \tomega, \tx),\\
\label{eq:def_cS}\cS_{\tau, \ttau}(g)&:= g(\tau, \ttau),\\
\label{eq:def_cT}\cT_{\tau, \ttau}(g)&:=\div_{\theta}g(\tau, \ttau) + \div_{\ttheta}g(\tau, \ttau),
\end{align}
\end{definition}
\begin{definition}[Continuity in $x$]
\label{def:linearforms_x}
For any $a, b\in \bbS$, $\Delta$ subinterval of $\bbS$, $ \tau\in \bbX\times\bbY\times\bbS$, $N\geq1$, $i\in \Lambda_{ N}$, $(\tau_{ 1}, \ldots, \tau_{ N})\in (\bbX\times\bbY\times\bbS)^{ N}$, $r=1, \ldots, m$, $ \xi\in\cM_{ 1}(\bbX\times\bbY)$, define the linear forms
\begin{align}
\cU(g)&=\cU_{ \tau, a, b, \Delta, \xi}(g):= \int_{\Delta} \Psi(b, \tx)\int  \left(g(\tau, \ttheta, \tomega, a) - g(\tau, \ttheta, \tomega, \tx)\right)\xi(\dd \ttheta, \dd\tomega) \dd \tx,\label{eq:def_cU}\\
\cV(g)&=\cV_{ N, i, r}(g):= \frac{1}{ \left\vert \Lambda_{ N} \right\vert} \sum_{ k\in \Lambda_{ N}}\partial_{ \theta^{ (r)}}g( \tau_{ k}, \tau_{ i}) \Psi(x_{ k}, x_{ i}),\label{eq:def_cV}\\
\cW(g)&=\cW_{ N, i, r, \xi}(g):= \int\partial_{ \theta^{ (r)}}g(\tau_{ i}, \ttau) \Psi(x_{ i}, \tx)\xi(\dd\ttheta, \dd \tomega) \dd \tx.\label{eq:def_cW}
\end{align}
\end{definition}
\begin{proposition}
\label{prop:linearforms}
For all $\tau, \ttau, \pi, \tpi$, for any $q\geq P+2$, the linear forms of Definition~\ref{def:linearforms_theta} are continuous on $\bW^{ \kappa_{ 1}, \iota_{ 1}}_{q}$, with norms:
\begin{align*}
\N{\cR_{ \tau, \ttau, \pi, \tpi}}_{- q, \kappa_{ 1}, \iota_{ 1}} &\leq C \left(\left\vert \theta- \pi \right\vert  + \left\vert \ttheta- \tpi \right\vert \right) \chi_{\kappa_{ 1}, \iota_{ 1}}( \tau, \ttau, \pi, \tpi),\\
\N{\cS_{ \tau, \ttau}}_{- q, \kappa_{ 1}, \iota_{ 1}}&\leq C w(\tau, \ttau, \kappa_{ 1}, \iota_{ 1}),\\
\N{\cT_{ \tau, \ttau}}_{- q, \kappa_{ 1}, \iota_{ 1}}&\leq C w(\tau, \ttau, \kappa_{ 1}, \iota_{ 1}),
\end{align*}
where $w$ is defined in \eqref{eq:def_w} and
\begin{equation}
\label{eq:chi}
\chi_{\kappa_{ 1}, \iota_{ 1}}( \tau, \ttau, \pi, \tpi):= 1+ \left\vert \theta \right\vert^{  \kappa_{ 1}} + \vert \ttheta \vert^{  \kappa_{ 1}} +\left\vert \pi \right\vert^{ \kappa_{ 1}} + \left\vert \tpi \right\vert^{ \kappa_{ 1}} + \left\vert \omega \right\vert^{  \iota_{ 1}} + \vert \tomega \vert^{  \iota_{ 1}}\end{equation} and where the constant $C$ is independent of $( \tau, \ttau, \pi, \tpi)$.
\end{proposition}
\begin{proof}[Proof of Proposition~\ref{prop:linearforms}]
We prove the first estimate of Proposition~\ref{prop:linearforms}. Fix $g$ regular, with compact support in $(\theta, \omega, \ttheta, \tomega)$. 
\begin{align*}
\left\vert \cR_{ \tau, \ttau, \pi, \tpi}(g) \right\vert &\leq \left(\left\vert \theta- \ttheta \right\vert  + \left\vert \pi- \tpi \right\vert \right) \sup_{u, v} \left\vert \sum_{r=1}^{m} \left(\partial_{u^{(r)}} g(u, \omega, x, v, \tomega, \tx) + \partial_{v^{(r)}} g(u, \omega, x, v, \tomega, \tx)\right)  \right\vert,\\
&\leq C\left(\left\vert \theta- \ttheta \right\vert + \left\vert \pi- \tpi \right\vert\right)\chi_{ \kappa_{ 1}, \iota_{ 1}}( \tau, \ttau, \pi, \tpi) \N{g}_{C^{ \kappa_{ 1}, \iota_{ 1}}_{ 1}},\\
&\leq C\left( \left\vert \theta- \ttheta \right\vert + \left\vert \pi- \tpi \right\vert\right)\chi_{ \kappa_{ 1}, \iota_{ 1}}( \tau, \ttau, \pi, \tpi) \N{g}_{q, \kappa_{ 1}, \iota_{ 1}},\ \text{(by \eqref{eq:embed})}
\end{align*}
where the supremum in the first inequality is taken over $ \left\vert u \right\vert \leq \left\vert \theta \right\vert + \left\vert \ttheta \right\vert$ and $ \left\vert v \right\vert\leq \left\vert \pi \right\vert + \left\vert \tpi \right\vert$ and where the integer $P$ is defined in \eqref{eq:P}. The result follows by a density argument. The proofs of the estimates on $\cS$ and $\cT$ are similar and left to the reader.
\end{proof}
\begin{proposition}
\label{prop:contU}
For any $q\geq P+2$, the linear forms of Definition~\ref{def:linearforms_x} are continuous on $\bW^{ \kappa_{ 1}, \iota_{ 1}}_{q}$, in the sense that
\begin{align}
\N{\cU}_{ -q, \kappa_{ 1}, \iota_{ 1}}&\leq C\left(1+ \left\vert \theta \right\vert^{ \kappa_{ 1}} + \left\vert \omega \right\vert^{\iota_{ 1}}\right) \left(\sup_{ x\in \Delta}\left\vert a-x \right\vert \right)\int_{ \Delta} \Psi(b, \tx)  \dd \tx,\label{eq:contU}\\
\N{\cV}_{ -q, \kappa_{ 1}, \iota_{ 1}}&\leq \frac{C}{ \left\vert \Lambda_{ N} \right\vert} \sum_{ k\in \Lambda_{ N}}w( \tau_{ k}, \tau_{ i}, \kappa_{ 1}, \iota_{ 1}) \Psi(x_{ k}, x_{ i}),\label{eq:contV}\\
\N{\cW}_{ -q, \kappa_{ 1}, \iota_{ 1}}&\leq C \int w( \ttau, \tau_{ i}, \kappa_{ 1}, \iota_{ 1}) \xi( \dd \ttheta, \dd \tomega).\label{eq:contW}
\end{align}
\end{proposition}
\begin{proof}[Proof of Proposition~ \ref{prop:contU}]
Using the moment estimates of Proposition~\ref{prop:moment_particles},
\begin{align*}
\left\vert \cU(g) \right\vert &\leq \int_{ \Delta} \Psi(b, \tx) \left\vert \int \left(g(\tau, \ttheta, \tomega, a) - g(\tau, \ttheta, \tomega, \tx)\right) \xi(\dd \ttheta, \dd\tomega) \right\vert \dd \tx\\
&\leq C \N{g}_{C^{ \kappa_{ 1}, \iota_{ 1}}_{ 1}} \int w( \tau, \ttau, \kappa_{ 1}, \iota_{ 1}) \xi(\dd \ttheta, \dd\tomega) \int_{ \Delta} \Psi(b, \tx) \left\vert a- \tx \right\vert   \dd \tx,\\
&\leq C \N{g}_{ q, \kappa_{ 1}, \iota_{ 1}}\left(1+ \left\vert \theta \right\vert^{ \kappa_{ 1}} + \left\vert \omega \right\vert^{ \iota_{ 1}}\right)\left(\sup_{ x\in \Delta} \left\vert x-a \right\vert\right)\int_{ \Delta} \Psi(b, \tx) \dd \tx.
\end{align*}
The proof for $\cV$ and $\cW$ are similar and left to the reader.
\end{proof}

\subsection{Tightness criterion}
\label{sec:tightness_crit}
We prove tightness for both fluctuation processes $ \eta_{ N}$ \eqref{eq:fluct} and $ \cH_{ N}$ \eqref{eq:secorderfluc} in suitable $\bV$ and $\bW$ defined in Section~\ref{sec:Hilbert_spaces}. We recall the tightness criterion used here (\cite{JoffeMetivier1986}, p.35): a sequence of $( \Omega_{ N}, \cF_{ N, t})$-adapted processes $(Y_{ N})_{ N\geq 1}$ with paths in $\cC([0, T], H)$ where $H$ is an Hilbert space is tight if both conditions hold:
\begin{enumerate}
\item There exists an Hilbert space $H_{ 0}$ such that $ H_{ 0}\hookrightarrow H$ (with Hilbert-Schmidt injection) and such that for all $t \leq T$, 
\begin{equation}
\label{eq:tight1}
\sup_{ N} \bE \left( \N{Y_{ N, t}}_{ H_{ 0}}^{ 2}\right)< +\infty,
\end{equation}
\item Aldous condition: for every $ \varepsilon_{ 1}, \varepsilon_{ 2}>0$, there exists $ s_{ 0}>0$ and an integer $N_{ 0}$ such that for every $(\cF_{ N, t})$-stopping time $ T_{ N}\leq T$,
\begin{equation}
\label{eq:tight2}
\sup_{ N\geq } \sup_{ s\leq s_{ 0}} \bP( \N{Y_{ N, \tau_{ N}} - Y_{ N, \tau_{ N} + s}}_{ H} \geq \varepsilon_{ 1})\leq \varepsilon_{ 2}.
\end{equation}
\end{enumerate}

\subsection{Tightness of the two-particle fluctuation process}
\label{sec:tightness_two_particle}
\subsubsection{Boundedness of the fluctuation process}
\label{sec:boundedness_two_particle}
\begin{proposition}
\label{prop:etaNbounded1}
Under the assumptions made in Section~\ref{sec:assumptions}, for any $q\geq P+2$,
\begin{equation}
\label{eq:etaNbounded1}
\sup_{1\leq N} \sup_{t\leq T} \bE \left(\N{\cH_{ N,t}}_{-q, \kappa_{ 1}, \iota_{ 1}}^{4}\right) <+ \infty.
\end{equation}
\end{proposition}
\begin{remark}
We deduce from \eqref{eq:etaNbounded1} that $\sup_{1\leq N, t\leq T} \bE \left(\N{\cH_{ N,t}}_{-q, \kappa_{ 1}, \iota_{ 1}}^{2}\right) <+ \infty$, which is the main estimate that we use in the following. A fourth moment estimate is only needed in Proposition~\ref{prop:UN_VN}.
\end{remark}
Proof of Proposition~\ref{prop:etaNbounded1} is given in Section~\ref{sec:proof_etaNbounded1}.

\begin{proposition}
\label{prop:continuityL}
Under the assumptions of Section~\ref{sec:assumptions}, the linear operator $\mathscr{L}_{s}$ defined in \eqref{eq:def_opL} is continuous from $\bW_{q+P+2}^{ \kappa_{ 0}, \iota_{ 0}}$ to $\bW_{ q}^{ \kappa_{ 1}, \iota_{ 1}}$: there exists some constant $C>0$, such that for all $g\in \bW_{ q+P+2}^{ \kappa_{ 0}, \iota_{ 0}}$,
\begin{equation}
\label{eq:continuityL}
\N{\mathscr{L}_{ s}g}_{q, \kappa_{ 1}, \iota_{ 1}}\leq C \N{g}_{ q+P+2, \kappa_{ 0}, \iota_{ 0}}.
\end{equation}
\end{proposition}
\begin{proof}[Proof of Proposition~ \ref{prop:continuityL}]
The first term $ \frac{ 1}{ 2} \Delta_{ \theta, \ttheta} g$ in $\mathscr{L}_{ s}g$ clearly satisfies \eqref{eq:continuityL}. We concentrate on the second term $(\tau, \ttau) \mapsto \nabla_{ \theta}g( \tau, \ttau)\cdot c( \theta, \omega)$: since $ \N{\nabla_{ \theta}g( \tau, \ttau) c( \theta, \omega)}_{ q,\kappa_{ 1}, \iota_{ 1}}^{ 2}=  \sum_{ \left\vert k \right\vert + \left\vert \tk \right\vert\leq q} \int  \frac{ \left\vert D_{\tau}^{k}D_{ \ttau}^{ \tk} \left(\nabla_{ \theta}g(\tau, \ttau) c(\theta, \omega)\right) \right\vert^{2}}{w( \tau, \ttau, \kappa_{ 1}, \iota_{ 1})^{ 2}}\dd\tau \dd \ttau$, it suffices to estimate the quantity \[\int  \frac{ \left\vert D_{ 1} \nabla_{ \theta}g(\tau, \ttau)\right\vert^{ 2} \left\vert D_{ 2}c(\theta, \omega) \right\vert^{2}}{ w( \tau, \ttau, \kappa_{ 1}, \iota_{ 1})^{ 2}}\dd\tau \dd \ttau,\] for all differential operators $D_{ 1}$ and $D_{ 2}$ of order smaller or equal than $q$. Namely,
\begin{align}
\int  \frac{ \left\vert D_{ 1} \nabla_{ \theta}g(\tau, \ttau)\right\vert^{ 2} \left\vert D_{ 2}c(\theta, \omega) \right\vert^{2}}{ w( \tau, \ttau, \kappa_{ 1}, \iota_{ 1})^{ 2}}\dd\tau \dd \ttau&= \int \frac{  \left\vert D_{ 1} \nabla_{ \theta}g(\tau, \ttau)\right\vert^{ 2} }{ w(\tau, \ttau, \kappa_{ 0}, \iota_{ 0})^{ 2}} \frac{\left\vert D_{ 2}c(\theta, \omega) \right\vert^{2}w(\tau, \ttau, \kappa_{ 0}, \iota_{ 0})^{ 2}}{ w(\tau, \ttau, \kappa_{ 1}, \iota_{ 1})^{ 2}}\dd\tau \dd \ttau,\nonumber\\
&\leq \N{g}_{C^{ \kappa_{ 0}, \iota_{ 0}}_{ q+1}}^{ 2} \int \frac{\left\vert D_{ 2}c(\theta, \omega) \right\vert^{2}w(\tau, \ttau, \kappa_{ 0}, \iota_{ 0})^{ 2}}{ w(\tau, \ttau, \kappa_{ 1}, \iota_{ 1})^{ 2}}\dd\tau \dd \ttau, \label{aux:D2c}\\
&\leq C \N{g}_{ q+P+2, \kappa_{ 0}, \iota_{ 0}}^{ 2},\nonumber
\end{align}
thanks to the embedding \eqref{eq:embed} and by assumption on $c$ (recall \eqref{eq:cgrowthcond}). Note that the definition of $ \gamma$ in \eqref{eq:def_beta} ensures that the integral in \eqref{aux:D2c} is indeed finite. The third term in \eqref{eq:def_opL}, $( \tau, \ttau) \mapsto \nabla_{  \theta} g( \tau, \ttau)\cdot  \left[ \Gamma \Psi, \nu_{ s}\right] (\tau)$ can be treated in the exact same way, observing that $\left[ \Gamma \Psi, \nu_{ s}\right] (\tau)$ is actually independent of the space variable $x$ (see~\eqref{eq:GamPsix}). It remains to treat the last term $( \tau, \ttau) \mapsto \cro{ \nu_{s}}{ \Psi(\cdot, x)\nabla_{ \ttheta}g(\cdot, \tau)} \cdot\Gamma( \theta, \omega, \ttheta, \tomega)$. For all $ \tau= (\theta, \omega, x)$, 
\begin{align*}
 \cro{ \nu_{s}}{ \Psi(\cdot, x)\nabla_{ \ttheta}g(\cdot, \tau)} &=\int_{ \bbS} \Psi(y, x) \int_{ \bbX\times\bbY} \nabla_{ \ttheta} g( \pi, \phi, y, \theta, \omega, x) \xi_{ s}( \dd \pi, \dd \phi)\dd y,\\
 &= \int_{ \bbS} \rho(z) \int_{ \bbX\times\bbY} \nabla_{ \ttheta} g( \pi, \phi, z+x, \theta, \omega, x) \xi_{ s}( \dd \pi, \dd \phi)\dd z,
\end{align*}
where (recall \eqref{eq:dist_d} and \eqref{eq:def_psi})
\begin{equation}
\label{eq:def_rho}
\rho(z):= \min( \left\vert z \right\vert, 1 - \left\vert z \right\vert)^{ -\alpha}
\end{equation} belongs to $L^{ 1}(\bbS)$. In particular, for any differential operator $D$ (acting on $(\theta, \omega, x)$) of order smaller or equal than $q$,
\begin{align*}
\left\vert D \cro{ \nu_{s}}{ \Psi(\cdot, x)\nabla_{ \ttheta}g(\cdot, \tau)} \right\vert&=  \int_{ \bbS} \rho(z) \int_{ \bbX\times\bbY} D \left[ \nabla_{ \ttheta} g( \pi, \phi, z+x, \theta, \omega, x)\right] \xi_{ s}( \dd \pi, \dd \phi)\dd z,\\
&\leq C\N{g}_{C^{ \kappa_{ 0}, \iota_{ 0}}_{ q+1}} \N{ \rho}_{ L^{ 1}(\bbS)}\int_{ \bbX\times\bbY} \left(1+ \left\vert \pi \right\vert^{ \kappa_{ 0}} + \vert \theta \vert^{ \kappa_{ 0}}\right) \left(1+ \left\vert \phi \right\vert^{ \iota_{ 0}} + \vert \omega \vert^{ \iota_{ 0}}\right) \xi_{ s}(\dd\pi, \dd\phi),\\
&\leq C \N{g}_{ q+P+2, \kappa_{ 0}, \iota_{ 0}} \left(1 + \vert \theta \vert^{ \kappa_{ 0}} +  \vert \omega \vert^{ \iota_{ 0}}\right),
\end{align*}
since $\xi_{ s}$ has finite moments of order $ \kappa_{ 0}$ in $ \theta$ and $ \iota_{ 0}$ in $ \omega$ (recall Proposition~\ref{prop:moment_particles}). Since $ \Gamma(\cdot, \cdot)$ is supposed to be regular and bounded as well as its derivatives (recall Section~\ref{sec:assumptions}), it suffices to estimate, 
\begin{align*}
\int \frac{ \left\vert D \cro{ \nu_{s}}{ \Psi(\cdot, x)\nabla_{ \ttheta}g(\cdot, \tau)}  \right\vert^{ 2}}{ w( \tau, \ttau, \kappa_{ 1}, \iota_{ 1})^{ 2}} \dd \tau\dd \ttau &\leq C \N{g}_{ q+P+2, \kappa_{ 0}, \iota_{ 0}}^{ 2}\int \frac{  \left(1 + \vert \theta \vert^{ \kappa_{ 0}} +  \vert \omega \vert^{ \iota_{ 0}}\right)^{ 2}}{ w( \tau, \ttau, \kappa_{ 1}, \iota_{ 1})^{ 2}} \dd \tau\dd \ttau,\\
&\leq C \N{g}_{ q+P+2, \kappa_{ 0}, \iota_{ 0}}^{ 2}.
\end{align*}
This concludes the proof of Proposition~\ref{prop:continuityL}.
\end{proof}
\begin{proposition}
\label{prop:Phi_regularity}
Under the assumptions of Section~\ref{sec:assumptions}, for any $q \geq P+2$, $ \Phi$ defined in \eqref{eq:Phif} is a continuous linear operator from $\bV_{ q+P+2}^{ \kappa_{ 0}, \iota_{ 0}}$ to $ \bW_{ q}^{ \kappa_{ 1}, \iota_{ 1}}$, i.e. there exists a constant $C>0$ such that for all $f\in \bV_{ q+P+2}^{ \kappa_{ 0}, \iota_{ 0}}$
\begin{equation}
\label{eq:Phi_regularity}
\N{ \Phi[f]}_{ q, \kappa_{ 1}, \iota_{ 1}} \leq C \N{f}_{ q+P+2, \kappa_{ 0}+ \iota_{ 0}}.
\end{equation}
\end{proposition}
\begin{proof}[Proof of Proposition~\ref{prop:Phi_regularity}]
Straightforward, since by assumption, $ \Gamma$ is bounded as well as its derivatives.
\end{proof}
\begin{proposition}
\label{prop:UN_VN}
For any $q\geq P+2$, the remaining terms $F_{ N}$ and $G_{ N}$ in the semimartingale decomposition \eqref{eq:semimart_HN} of $\cH_{ N}$  converge uniformly in time to $0$, as $N\to\infty$, in $\bW_{ -q}^{ \kappa_{ 1}, \iota_{ 1}}$: more precisely, there exists a constant $C>0$ such that
\begin{align}
\sup_{ t\leq T}\bE \left(\N{F_{ N, t}}^{ 2}_{ -q, \kappa_{ 1}, \iota_{ 1}}\right)&\leq \frac{ C}{ a_{ N}^{ 2}},\label{eq:UN_bounded}\\
\sup_{ t\leq T}\bE \left(\N{G_{ N, t}}_{ -q, \kappa_{ 1}, \iota_{ 1}}^{ 2}\right)&\leq \frac{ C}{ a_{ N}^{ 2}}.\label{eq:VN_bounded}
\end{align}
\end{proposition}
\begin{proof}[Proof of Proposition~\ref{prop:UN_VN}]
Writing explicitly $F_{ N}$ and $G_{ N}$ from their definitions \eqref{eq:def_UN} and \eqref{eq:def_VN} gives:
\begin{align}
\begin{split}
F_{ N, t}g&= \frac{ 1}{ \left\vert \Lambda_{ N} \right\vert} \sum_{ i\in \Lambda_{ N}} \left[ \Gamma \Psi, \nu_{ N, t} - \nu_{ t}\right]( \tau_{ i, t})\\ &\left\lbrace a_{ N} \left( \frac{ 1}{ \left\vert \Lambda_{ N} \right\vert} \sum_{ j\in\Lambda_{ N}} \Psi(x_{ i}, x_{ j}) \nabla_{ \theta} g( \tau_{ i, t}, \tau_{ j, t}) - \int \Psi(x_{ i}, \tx) \nabla_{ \theta}g(  \tau_{ i, t}, \ttau) \nu_{ t}(\dd \ttau)\right)\right\rbrace,\label{aux:UN_expand}
\end{split}\\
\begin{split}
G_{ N, t}g&= \frac{ 1}{ \left\vert \Lambda_{ N} \right\vert} \sum_{ i\in \Lambda_{ N}}  \cro{ \nu_{ N, t}- \nu_{ t}}{ \Psi(\cdot, x_{ i}) \nabla_{ \ttheta}g( \cdot, \tau_{ i, t})}\\ &\left\lbrace a_{ N} \left( \frac{ 1}{ \left\vert \Lambda_{ N} \right\vert} \sum_{ j\in\Lambda_{ N}} \Psi(x_{ i}, x_{ j}) \Gamma( \theta_{ i, t}, \omega_{ i}, \theta_{ j, t}, \omega_{ j}) - \int \Psi(x_{ i}, \tx) \Gamma( \theta_{ i, t}, \omega_{ i}, \ttheta, \tomega) \nu_{ t}(\dd \ttau)\right)\right\rbrace.
\end{split}
\end{align}
For all $\tau=(\theta, \omega, x)\in\bbX\times\bbY\times\bbS$, define the linear form $\cF_{ N, \tau, t}$ by
\begin{equation}
\cF_{ N, \tau, t}g:=a_{ N} \left( \frac{ 1}{ \left\vert \Lambda_{ N} \right\vert} \sum_{ j\in\Lambda_{ N}} \Psi(x, x_{ j}) \nabla_{ \theta} g( \tau, \tau_{ j, t}) - \int \Psi(x, \tx) \nabla_{ \theta}g(\tau, \ttau) \nu_{ t}(\dd \ttau)\right).
\end{equation}
Hence, 
\begin{align*}
\bE(\N{F_{ N, t}}_{ -q}^{ 2})&\leq \frac{ 1}{ \left\vert \Lambda_{ N} \right\vert}\sum_{ i\in \Lambda_{ N}}\bE( \left\vert \left[ \Gamma \Psi, \nu_{ N, t} - \nu_{ t}\right]( \tau_{ i, t}) \right\vert^{ 2}\N{\cF_{ N, \tau_{ i, t}, t}}_{ -q}^{ 2}),\\
&\leq \frac{ 1}{ \left\vert \Lambda_{ N} \right\vert}\sum_{ i\in \Lambda_{ N}} \bE(\left\vert \left[ \Gamma \Psi, \nu_{ N, t} - \nu_{ t}\right]( \tau_{ i, t}) \right\vert^{ 4})^{ \frac{ 1}{ 2}} \bE(\N{\cF_{ N, \tau_{ i, t}, t}}_{ -q}^{ 4})^{ \frac{ 1}{ 2}},
\end{align*}
by Cauchy-Schwartz inequality. By the same procedure as in Proposition~\ref{prop:etaNbounded1} in Section~\ref{sec:proof_etaNbounded1}
(note that the structure of $\cF_{ N, \tau_{ i, t}, t}$ is very similar to $\cH_{ N}$, recall \eqref{eq:HN_g}), it is easy to show that
\begin{align*}
\sup_{ t\leq T}\sup_{ i\in \Lambda_{ N}} \bE \left(\left\vert \left[ \Gamma \Psi, \nu_{ N, t} - \nu_{ t}\right]( \tau_{ i, t}) \right\vert^{ 4}\right)&\leq \frac{ C}{ a_{ N}^{ 4}},\\
\sup_{ N\geq1}\sup_{ t\leq T}\sup_{ i\in \Lambda_{ N}} \bE(\N{\cF_{ N, \tau_{ i, t}, t}}_{ -q}^{ 4})&<+\infty.
\end{align*}
Hence, \eqref{eq:UN_bounded} easily follows. The same argument holds for $G_{ N}$. Proposition~\ref{prop:UN_VN} is proven.
\end{proof}

\begin{proposition}
\label{prop:MNH_bounded}
For any $q\geq P+2$, the process $\cM_{ N}^{ (\cH)}$ defined in \eqref{eq:def_MNH} is a martingale with values in $\bW_{ -q}^{\kappa_{ 1}, \iota_{ 1}}$ and there exists a constant $C>0$ such that
\begin{equation}
\label{eq:MNH_bounded}
\sup_{ t\leq T}\bE \left(\N{\cM_{ N, t}^{ (\cH)}}_{ -q, \kappa_{ 1}, \iota_{ 1}}^{ 2}\right)\leq C \frac{ a_{ N}^{ 2}}{ \left\vert \Lambda_{ N} \right\vert}.
\end{equation}
\end{proposition}
\begin{remark}
\label{rem:norms}
Due to the embedding \eqref{eq:comp_embed_dual}, one has $ \N{\cdot}_{ -(q+P+2), \kappa_{ 0}, \iota_{ 0}}\leq C \N{\cdot}_{ -q, \kappa_{ 1}, \iota_{ 1}}$, so that every estimate (especially \eqref{eq:UN_bounded}, \eqref{eq:VN_bounded} and \eqref{eq:MNH_bounded}) involving the $\N{\cdot}_{ -q, \kappa_{ 1}, \iota_{ 1}}$-norm is also valid for the $\N{\cdot}_{ -(q+P+2), \kappa_{ 0}, \iota_{ 0}}$-norm.
\end{remark}

\begin{proof}[Proof of Proposition~\ref{prop:MNH_bounded}]
For any test function $f$, 
\begin{align*}
\left\vert \cM_{ N, t}^{ (\cH)}g \right\vert&\leq  \left\vert \frac{ a_{ N}}{ \left\vert \Lambda_{ N} \right\vert^{ 2}} \sum_{ i, j\in \Lambda_{ N}} \int_{0}^{t}  \nabla_{  \ttheta} g(\tau_{ i, s}, \tau_{ j, s}) \Psi(x_{ i}, x_{ j})\cdot\dd B_{ j, s} \right\vert\\&+ \left\vert \frac{ a_{ N}}{ \left\vert \Lambda_{ N} \right\vert^{ 2}} \sum_{ i, j\in \Lambda_{ N}} \int_{0}^{t}  \nabla_{  \theta} g(\tau_{ i, s}, \tau_{ j, s})\Psi(x_{ i}, x_{ j}) \cdot\dd B_{ i, s}\right\vert\\ &+ \left\vert \frac{ a_{ N}}{ \left\vert \Lambda_{ N} \right\vert}\sum_{ i\in\Lambda_{ N}} \int_{0}^{t}\int\nabla_{ \theta}g(\tau_{ i, s}, \ttau) \Psi(x_{ i}, \tx)\nu_{ s}(\dd\ttau)\cdot\dd B_{ i, s} \right\vert,\\
&:= \left\vert \cM_{ N, t}^{ (1)}g \right\vert + \left\vert \cM_{ N, t}^{ (2)}g \right\vert  + \left\vert \cM_{ N, t}^{ (3)}g \right\vert.
\end{align*}
We only treat the first term $\cM_{ N, t}^{ (1)}g$, the other two being similar: using a complete orthonormal system $(\varphi_{ p})_{ p\geq1}$ in $\bW_{ q}^{ \kappa_{ 1}, \iota_{ 1}}$ (we drop the dependence in $ \kappa_{ 1}, \iota_{ 1}$ below for simplicity), we have successively
\begin{align}
\bE \left(\sup_{ t\leq T}\N{\cM_{ N, t}^{ (\cH)}}^{ 2}_{ -q}\right)&=\bE \left(\sup_{ t\leq T}\sum_{ p\geq1} \left\vert \cM_{ N, t}^{ (\cH)}\varphi_{ p} \right\vert^{ 2}\right)\leq\sum_{ p\geq1}\bE \left(\sup_{ t\leq T} \left\vert \cM_{ N, t}^{ (\cH)}\varphi_{ p} \right\vert^{ 2}\right),\nonumber\\
&\leq\sum_{ p\geq1}\bE \left(\left\vert \cM_{ N, T}^{ (\cH)}\varphi_{ p} \right\vert^{ 2}\right),\quad \text{(by Doob's inequality),}\nonumber\\
&\leq \frac{ a_{ N}^{ 2}}{ \left\vert \Lambda_{ N} \right\vert^{ 4}} \sum_{ i, j, k\in \Lambda_{ N}} \Psi(x_{ i}, x_{ j}) \Psi(x_{ k}, x_{ j})\nonumber\\ &\int_{0}^{T} \bE \left( \sum_{ p\geq1}\left\vert \nabla_{ \ttheta}\varphi_{ p}( \tau_{ i, s}, \tau_{ j, s}) \right\vert^{ 2}\right)^{ \frac{ 1}{ 2}} \bE \left(\sum_{ p\geq1}\left\vert \nabla_{ \ttheta}\varphi_{ p}( \tau_{ k, s}, \tau_{ j, s}) \right\vert^{ 2}\right)^{ \frac{ 1}{ 2}}\dd s,\nonumber\\
&\leq \frac{ a_{ N}^{ 2}}{ \left\vert \Lambda_{ N} \right\vert^{ 4}} \sum_{ i, j, k\in \Lambda_{ N}} \Psi(x_{ i}, x_{ j}) \Psi(x_{ k}, x_{ j})\nonumber\\ &\int_{0}^{T} \bE \left( \N{\cT_{ \tau_{ i, s}, \tau_{ j, s}}}^{ 2}_{ -q}\right)^{ \frac{ 1}{ 2}} \bE \left(\N{\cT_{ \tau_{ k, s}, \tau_{ j,s}}}^{ 2}_{ -q}\right)^{ \frac{ 1}{ 2}}\dd s,\label{aux:cV1}\\
&\leq C\frac{a_{ N}^{ 2}}{ \left\vert \Lambda_{ N} \right\vert^{ 4}} \sum_{ i, j, k\in \Lambda_{ N}} \Psi(x_{ i}, x_{ j}) \Psi(x_{ k}, x_{ j})\int_{0}^{T} \bE \left(w(\tau_{ i, s}, \tau_{ j, s})\right)\bE \left(w(\tau_{ k, s}, \tau_{ j, s})\right) \dd s,\label{aux:cV2}\\
&\leq C\frac{a_{ N}^{ 2}}{ \left\vert \Lambda_{ N} \right\vert^{ 4}} \sum_{ i, j, k\in \Lambda_{ N}} \Psi(x_{ i}, x_{ j}) \Psi(x_{ k}, x_{ j})\leq C \frac{ a_{ N}^{ 2}}{ \left\vert \Lambda_{ N} \right\vert},\label{aux:cV3}
\end{align}
where we recall in \eqref{aux:cV1} the definition \eqref{eq:def_cT} of the linear form $\cT$, where we used Proposition~\ref{prop:linearforms} in \eqref{aux:cV2} and Lemma~\ref{lem:Riem} in \eqref{aux:cV3}. This concludes the proof of Proposition~ \ref{prop:MNH_bounded}.
\end{proof}
\begin{proposition}
\label{prop:HN_bounded_2}
Under the hypotheses of Section~\ref{sec:assumptions}, for any $q\geq P+2$, the two-particle fluctuation process $\cH_{N}$ belongs uniformly to $\bW_{-(q+P+2)}^{ \kappa_{ 0}, \iota_{ 0}}$:
\begin{equation}
\label{eq:HN_bounded_2}
\sup_{1 \leq N} \bE \left(\sup_{ t\leq T} \N{\cH_{ N, t}}_{ -(q+P+2), \kappa_{ 0}, \iota_{ 0}}^{ 2}\right)<+\infty.
\end{equation}
\end{proposition}
\begin{proof}[Proof of Proposition~\ref{prop:HN_bounded_2}]
Let $(\psi_{ p})_{ p\geq1}$ be a complete orthonormal system in $\bW_{ q+P+2}^{ \kappa_{ 0}, \iota_{ 0}}$. Then, for all $p\geq 1$, 
\begin{equation*}
\begin{split}
\cro{\cH_{ N, t}}{ \psi_{ p}}^{ 2}\leq C &\Bigg(\cro{\cH_{ N, 0}}{ \psi_{ p}}^{ 2} + T\int_{0}^{t} \cro{\cH_{ N, s}}{\mathscr{L}_{ s} \psi_{ p}}^{ 2} \dd s + T\int_{0}^{t} \left\vert F_{ N, s}\psi_{ p} \right\vert^{ 2} \dd s\\ &+ T\int_{0}^{t} \left\vert G_{ N, s}\psi_{ p} \right\vert^{ 2} \dd s + \left\vert \cM_{ N, t}^{ (\cH)}\psi_{ p}\right\vert^{ 2}\Bigg),
\end{split}
\end{equation*}
so that, by Doob's inequality,
\begin{equation*}
\begin{split}
\bE \left(\sum_{ p\geq1}\sup_{ t\leq T}\cro{\cH_{ N, t}}{ \psi_{ p}}^{ 2}\right)\leq C &\Bigg(\bE \left(\sum_{ p\geq1}\cro{\cH_{ N, 0}}{ \psi_{ p}}^{ 2}\right) + T\int_{0}^{T} \bE \left(\sum_{ p\geq1}\cro{\cH_{ N, s}}{\mathscr{L}_{ s} \psi_{ p}}^{ 2}\right) \dd s\\ &+ T\int_{0}^{T} \bE \left(\sum_{ p\geq 1}\left\vert F_{ N, s}\psi_{ p} \right\vert^{ 2}\right) \dd s + T\int_{0}^{T} \bE \left(\sum_{ p\geq1}\left\vert G_{ N, s}\psi_{ p} \right\vert^{ 2}\right) \dd s\\ &+ \bE \left(\sum_{ p\geq 1}\left\vert \cM_{ N, T}^{ (\cH)}\psi_{ p}\right\vert^{ 2}\right)\Bigg).
\end{split}
\end{equation*}
Since $\bE \left(\sup_{ t\leq T} \N{\cH_{ N, t}}_{ -(q+P+2), \kappa_{ 0}, \iota_{ 0}}^{ 2}\right)\leq \bE \left(\sum_{ p\geq1}\sup_{ t\leq T}\cro{\cH_{ N, t}}{ \psi_{ p}}^{ 2}\right)$, we obtain from the previous bound that
\begin{equation}
\label{aux:supHN}
\begin{split}
\bE \left(\sup_{ t\leq T} \N{\cH_{ N, t}}_{ -(q+P+2), \kappa_{ 0}, \iota_{ 0}}^{ 2}\right)\leq C &\Bigg(\bE \left(\N{\cH_{ N, 0}}_{ -(q+P+2), \kappa_{ 0}, \iota_{ 0}}^{ 2}\right) + T\int_{0}^{T} \bE \left(\sum_{ p\geq1}\cro{\cH_{ N, s}}{\mathscr{L}_{ s} \psi_{ p}}^{ 2}\right) \dd s\\ &+ T\int_{0}^{T} \bE \left(\N{F_{ N, s}}_{ -(q+P+2), \kappa_{ 0}, \iota_{ 0}}^{ 2}\right) \dd s\\ &+ T\int_{0}^{T} \bE \left(\N{G_{ N, s}}_{ -(q+P+2), \kappa_{ 0}, \iota_{ 0}}^{ 2}\right) \dd s\\ &+ \bE \left(\N{\cM_{ N, T}^{ (\cH)}}_{ -(q+P+2), \kappa_{ 0}, \iota_{ 0}}^{ 2}\right)\Bigg).
\end{split}
\end{equation}
Using \eqref{eq:MNH_bounded} and Remark~\ref{rem:norms}, one obtains that
\begin{equation}
\label{eq:MNH_bounded2}
\sup_{ 1\leq N}\bE \left(\N{\cM_{ N, T}^{ (\cH)}}_{ -(q+P+2), \kappa_{ 0}, \iota_{ 0}}^{ 2}\right)<+\infty,
\end{equation}
and using \eqref{eq:etaNbounded1}, one obtains a similar bound for the initial fluctuations
\begin{equation}
\label{eq:HN0_bounded}
\sup_{ 1\leq N}\bE \left(\N{\cH_{ N, 0}}_{ -(q+P+2), \kappa_{ 0}, \iota_{ 0}}^{ 2}\right)<+\infty.
\end{equation}
Moreover, if one introduces the linear form $\ell_{ N, s}: \psi \mapsto \cro{\cH_{ N, s}}{\mathscr{L}_{ s} \psi}$, we see from Proposition~\ref{prop:continuityL} that $\ell_{ N, s}$ is continuous on $\bW_{ (q+P+2), \kappa_{ 0}, \iota_{ 0}}$: indeed, for all $ \psi\in \bW_{ (q+P+2), \kappa_{ 0}, \iota_{ 0}}$
\begin{align*}
\left\vert \ell_{ N, s}\psi \right\vert&\leq \N{\cH_{ N, s}}_{ -q, \kappa_{ 1}, \iota_{ 1}} \N{\mathscr{L}_{ s}\psi}_{q, \kappa_{ 1}, \iota_{ 1}},\\
&\leq C \N{\cH_{ N, s}}_{ -q, \kappa_{ 1}, \iota_{ 1}} \N{ \psi}_{ q+P+2, \kappa_{ 0}, \iota_{ 0}}, \ \text{(thanks to \eqref{eq:continuityL}).}
\end{align*}
In particular
\begin{align*}
\bE \left(\sum_{ p\geq1}\cro{\cH_{ N, s}}{\mathscr{L}_{ s} \psi_{ p}}^{ 2}\right)&= \bE \left( \N{\ell_{ N, s}}^{ 2}_{ -(q+P+2), \kappa_{ 0}, \iota_{ 0}}\right),\\
&\leq C \bE \left(\N{\cH_{ N, s}}^{ 2}_{ -q, \kappa_{ 1}, \iota_{ 1}}\right),\\
&\leq C \sup_{ 1\leq N} \sup_{ s\leq T}\bE \left(\N{\cH_{ N, s}}^{ 2}_{ -q, \kappa_{ 1}, \iota_{ 1}}\right)<+\infty,\ \text{(by \eqref{eq:etaNbounded1}).}
\end{align*}
Putting the previous estimate as well as \eqref{eq:UN_bounded}, \eqref{eq:VN_bounded}, \eqref{eq:MNH_bounded2} and \eqref{eq:HN0_bounded} into \eqref{aux:supHN}, one obtains the result \eqref{eq:HN_bounded_2}. Proposition~\ref{prop:HN_bounded_2} is proven.
\end{proof}

The rest of this section is devoted to prove the tightness of the two-particle fluctuation process $\cH_{ N}$. We follow here the lines of \cite{Fernandez1997,Jourdain1998,Lucon2011}.
\begin{proposition}
\label{prop:contHN_MNH}
For all $q\geq P+2$, for every $N\geq1$, the trajectories of the two-particle fluctuation process $\cH_{ N}$ and the martingale $\cM_{ N}^{ (\cH)}$ are almost-surely continuous in $\bW_{ q+P+2}^{ \kappa_{ 0}, \iota_{ 0}}$.
\end{proposition}
\begin{proof}[Proof of Proposition~\ref{prop:contHN_MNH}]
We only prove the result for $\cH_{ N}$, the proof for $\cM_{ N}^{ (\cH)}$ being similar. Using again a complete orthonormal system $ (\psi_{ p})_{ p\geq1}$ in $\bW_{ q+P+2}^{ \kappa_{ 0}, \iota_{ 0}}$, we know that, thanks to the proof of Proposition~\ref{prop:HN_bounded_2}, with probability $1$, for all $ \varepsilon>0$, there exists $p_{ 0}\geq1$ such that $\sum_{ p> p_{ 0}}\sup_{ t\leq T} \cro{ \cH_{ N, t}}{ \psi_{ p}}^{ 2}< \frac{ \varepsilon}{ 6}$. Let $0\leq t\leq T$ and $t_{ m}$ a sequence converging to $t$, as $m\to \infty$. Then,
\begin{align*}
\N{\cH_{ N,t_{ m}} - \cH_{ N, t}}_{ -(q+P+2), \kappa_{ 0}, \iota_{ 0}}^{ 2}&= \sum_{ p\geq1} \cro{ \cH_{ N,t_{ m}} - \cH_{ N, t}}{ \psi_{ p}}^{ 2},\\
&\leq \sum_{ p=1}^{ p_{ 0}} \cro{ \cH_{ N,t_{ m}} - \cH_{ N, t}}{ \psi_{ p}}^{ 2} + 2 \sum_{ p>p_{ 0}} \left(\cro{ \cH_{ N,t_{ m}}}{ \psi_{ p}}^{ 2} + \cro{\cH_{ N, t}}{ \psi_{ p}}^{ 2}\right),\\
&\leq \sum_{ p=1}^{ p_{ 0}} \cro{ \cH_{ N,t_{ m}} - \cH_{ N, t}}{ \psi_{ p}}^{ 2} + \frac{ 2 \varepsilon}{ 3}\leq \varepsilon,
\end{align*}
for $t_{ m}$ sufficiently close to $t$ (by continuity of the map $t\mapsto \cro{ \cH_{ N, t}}{ \psi_{ p}}$ for all $p\leq p_{ 0}$).
\end{proof}
\begin{proposition}
\label{prop:HN_semimart}
The process $\cH_{ N}$ is a semimartingale in $\bW_{ -(q+P+2)}^{ \kappa_{ 0}, \iota_{ 0}}$ and in this space
\begin{equation}
\label{eq:HN_semimart}
\cH_{ N,t} = \cH_{ N, 0} + \int_{0}^{t} \mathscr{L}_{ s}^{ \ast} \cH_{ N, s}\dd s + \int_{0}^{t} F_{ N, s} \dd s + \int_{0}^{t} G_{ N, s}\dd s + \cM_{ N, t}^{ (\cH)},
\end{equation}
where $\mathscr{L}_{ s}^{ \ast}$ is the adjoint of $\mathscr{L}_{ s}$ and $\cM_{ N}^{ (\cH)}$ is a martingale in $\bW_{ -(q+P+2)}^{\kappa_{ 0}, \iota_{ 0}}$ with Doob-Meyer process $\llangle \cM_{ N}^{ (\cH)}\rrangle$ with values in $L(\bW_{ q+P+2}^{ \kappa_{ 0}, \iota_{ 0}}, \bW_{ -(q+P+2)}^{ \kappa_{ 0}, \iota_{ 0}})$ is given for every $ \varphi$, $ \psi\in \bW_{ q+P+2}^{ \kappa_{ 0}, \iota_{ 0}}$ by
\begin{align}
\llangle \cM_{ N}^{ (\cH)}\rrangle_{ t} ( \varphi) ( \psi) &= \cro{ \cM_{ N}^{ (\cH)}\varphi}{ \cM_{ N}^{ (\cH)} \psi}_{ t}= \frac{ a_{ N}^{ 2}}{ \left\vert \Lambda_{ N} \right\vert} \int_{0}^{t} b_{ N, s}( \nu_{ N,s})(\varphi, \psi) \dd s,\label{eq:cro_Hilb_MNH}
\end{align}
where for any measure $ \lambda$ on $\bbX\times\bbY\times\bbS$ (we recall \eqref{eq:cro_2} and write $\theta= (\theta^{ (1)}, \ldots, \theta^{ (m)})\in\bbX$), 
\begin{equation}
\begin{split}
b_{ N, s}(\lambda)(\varphi, \psi)&:=\sum_{ r=1}^{ m} \int \left[(\partial_{ \ttheta^{ (r)}}\varphi)\Psi, \lambda\right]_{ 2}(\ttau)\left[(\partial_{ \ttheta^{ (r)}}\psi)\Psi, \lambda\right]_{ 2}(\ttau)\lambda(\dd \ttau)\\
&+ \sum_{ r=1}^{ m}\int \partial_{ \ttheta^{ (r)}} \varphi(\tau, \ttau) \Psi(x, \tx) \epsilon_{ N, s}^{ (r)}( \ttau, \psi)\lambda(\dd \tau)\lambda(\dd \ttau)\\ &+ \sum_{ r=1}^{ m}\int \partial_{ \ttheta^{ (r)}} \psi(\tau, \ttau) \Psi(x, \tx) \epsilon_{ N, s}^{ (r)}( \ttau, \varphi)\lambda(\dd \tau)\lambda(\dd \ttau)\\
&+ \sum_{ r=1}^{ m}\int \epsilon_{ N, s}^{ (r)}(\tau,\varphi)\epsilon_{ N, s}^{ (r)}(\tau,\psi)\lambda(\dd \tau)
\end{split}
\end{equation}
and
\begin{equation}
\label{eq:def_bN}
\epsilon_{ N,s}^{ (r)}(\tau,\varphi):=\frac{ 1}{ \left\vert \Lambda_{ N} \right\vert} \sum_{ j\in \Lambda_{ N}} \partial_{  \theta^{ (r)}} \varphi(\tau, \tau_{ j, s})\Psi(x, x_{ j}) - \int\partial_{ \theta^{ (r)}}\varphi(\tau, \ttau) \Psi(x, \tx)\nu_{ s}(\dd\ttau).
\end{equation}
\end{proposition}
\begin{proof}[Proof of Proposition~\ref{prop:HN_semimart}]
It is a simple adaptation to our case of \cite{Fernandez1997}, Proposition~4.5. We refer to it for details. Note that, using the independence of the Brownian motions $(B_{ i})_{ i\in \Lambda_{ N}}$, one has
\begin{align*}
\left\langle \cM_{ N}^{ (\cH)} \varphi\, ,\, \cM_{ N}^{ (\cH)} \psi\right\rangle_{ t}&= \frac{ a_{ N}^{ 2}}{ \left\vert \Lambda_{ N} \right\vert^{ 4}} \sum_{ i, j, k\in \Lambda_{ N}} \sum_{ r=1}^{ m} \int_{0}^{t} \partial_{  \ttheta^{ (r)}} \varphi(\tau_{ i, s}, \tau_{ j, s}) \Psi(x_{ i}, x_{ j})\partial_{  \ttheta^{ (r)}} \psi(\tau_{ k, s}, \tau_{ j, s}) \Psi(x_{ k}, x_{ j}) \dd s\\
&+ \frac{ a_{ N}^{ 2}}{ \left\vert \Lambda_{ N} \right\vert^{ 3}}\sum_{ i, j\in \Lambda_{ N}} \sum_{ r=1}^{ m} \int_{0}^{t} \partial_{  \ttheta^{ (r)}} \varphi(\tau_{ i, s}, \tau_{ j, s}) \Psi(x_{ i}, x_{ j})\epsilon_{ N, s}^{ (r)}(\tau_{ j, s}, \psi)\dd s\\
&+ \frac{ a_{ N}^{ 2}}{ \left\vert \Lambda_{ N} \right\vert^{ 3}}\sum_{ i, j\in \Lambda_{ N}} \sum_{ r=1}^{ m} \int_{0}^{t} \partial_{  \ttheta^{ (r)}} \psi(\tau_{ i, s}, \tau_{ j, s}) \Psi(x_{ i}, x_{ j})\epsilon_{ N, s}^{ (r)}(\tau_{ j, s}, \varphi)\dd s\\
&+ \frac{ a_{ N}^{ 2}}{ \left\vert \Lambda_{ N} \right\vert^{ 2}} \sum_{ i\in \Lambda_{ N}} \sum_{ r=1}^{ m} \int_{0}^{t}   \epsilon_{ N, s}^{ (r)}(\tau_{ i, s}, \varphi) \epsilon_{ N, s}^{ (r)}(\tau_{ i, s}, \psi) \dd s,
\end{align*}
which is precisely \eqref{eq:cro_Hilb_MNH}.
\end{proof}

\subsubsection{Tightness results}
\begin{proposition}
\label{prop:MNH_tight}
The sequence of the laws of $(\cM_{ N}^{ (\cH)})_{ N\geq 1}$ is tight in $\cC \left([0, T], \bW_{ -(q+P+2)}^{ \kappa_{ 0}, \iota_{ 0}}\right)$.
\end{proposition}
\begin{proof}[Proof of Proposition~\ref{prop:MNH_tight}]
We prove that $\cM_{ N}^{ (H)}$ satisfies the tightness criterion of Section~\ref{sec:tightness_crit}: by Proposition~\ref{prop:MNH_bounded}, the first part of the tightness criterion above is verified for $H_{ 0}=\bW_{ -q}^{ \kappa_{ 1}, \iota_{ 1}}$ and $H= \bW_{ -(q+P+2)}^{ \kappa_{ 0}, \iota_{ 0}}$, since the embedding $\bW_{ -q}^{ \kappa_{ 1}, \iota_{ 1}}\hookrightarrow \bW_{ -(q+P+2)}^{ \kappa_{ 0}, \iota_{ 0}}$ is Hilbert-Schmidt. The second part of the criterion is verified if (by Rebolledo's Theorem, see \cite{JoffeMetivier1986}, p.40) it is satisfied by the trace ${\rm tr}_{ \bW_{ -(q+P+2)}^{ \kappa_{ 0}, \iota_{ 0}}}\llangle \cM_{ N}^{ (H)}\rrangle$. Let $ \varepsilon_{ 1}>0$ and $ T_{ N}\leq T$ be a stopping time. For a complete orthonormal system $(\psi_{ p})_{ p\geq1}$ in $\bW_{ q+P+2}^{ \kappa_{ 0}, \iota_{ 0}}$, we have
\begin{align*}
\bP &\left(\left\vert {\rm tr}_{ \bW_{ -(q+P+2)}^{ \kappa_{ 0}, \iota_{ 0}}}\llangle \cM_{ N}^{ (H)}\rrangle_{ T_{ N}+ s} - {\rm tr}_{ \bW_{ -(q+P+2)}^{ \kappa_{ 0}, \iota_{ 0}}}\llangle \cM_{ N}^{ (H)}\rrangle_{ T_{ N}} \right\vert > \varepsilon_{ 1}\right),\\
&\leq \frac{ 1}{ \varepsilon_{ 1}}\bE \left(\int_{ T_{ N}}^{ T_{ N}+s} {\rm tr}_{ \bW_{ -(q+P+2)}^{ \kappa_{ 0}, \iota_{ 0}}}b_{ N, u}( \nu_{ N,u}) \dd u\right),\\
&\leq \frac{ 1}{ \varepsilon_{ 1}} \sum_{ p\geq1}\sum_{ r=1}^{ m}\bE \left( \int_{ T_{ N}}^{ T_{ N}+s} \int \left(\int\partial_{ \ttheta^{ (r)}}\psi_{ p}(\tau_{ 1}, \tau_{ 2}) \Psi(x_{ 1}, x_{ 2}) \nu_{ N, u}(\dd \tau_{ 1})\right)^{ 2} \nu_{ N, u}(\dd \tau_{ 2})\dd u\right)\\
&+ \frac{ 1}{ \varepsilon_{ 1}} \sum_{ p\geq1}\sum_{ r=1}^{ m}\bE \left(\int_{T_{ N}}^{T_{ N}+s} \int \epsilon^{ (r)}_{ N, u}(\tau, \psi_{ p})^{ 2} \nu_{ N, u}(\dd \tau) \dd u\right)\\
&+ \frac{ 2}{ \varepsilon_{ 1}} \sum_{ p\geq1}\sum_{ r=1}^{ m}\bE \left( \int_{T_{ N}}^{T_{ N}+s} \int \partial_{ \ttheta^{ (r)}} \psi_{ p}(\tau_{ 1}, \tau_{ 2})\Psi(x_{ 1}, x_{ 2})\epsilon^{ (r)}_{ N, u}(\tau_{ 2}, \psi_{ p}) \nu_{ N, u}(\dd \tau_{ 1}) \nu_{ N, u}(\dd \tau_{ 2})\dd u\right),\\
&=A_{N, 1} + A_{ N, 2} + A_{ N, 3}.
\end{align*}
We only treat the first term above and leave the remaining two to the reader.
\begin{align*}
A_{ N, 1}&=\frac{ 1}{ \varepsilon_{ 1}} \sum_{ p\geq1}\sum_{ r=1}^{ m}\bE \left( \int_{ T_{ N}}^{ T_{ N}+s} \frac{ 1}{ \left\vert \Lambda_{ N} \right\vert}\sum_{ i\in \Lambda_{ N}} \left( \frac{ 1}{ \left\vert \Lambda_{ N} \right\vert}\sum_{ j\in \Lambda_{ N}}\partial_{ \ttheta^{ (r)}}\psi_{ p}(\tau_{ j, u}, \tau_{ i, u}) \Psi(x_{ j}, x_{ i}) \right)^{ 2}\dd u\right),\\
&=\frac{ 1}{ \varepsilon_{ 1}} \sum_{ r=1}^{ m}\bE \left( \int_{ T_{ N}}^{ T_{ N}+s} \frac{ 1}{ \left\vert \Lambda_{ N} \right\vert}\sum_{ i\in \Lambda_{ N}} \N{ \cV_{ N, i, r}}_{ -(q+P+2), \kappa_{ 0}, \iota_{ 0}}^{ 2}\dd u\right),
\end{align*}
where $\cV_{ N, i, r}$ is the linear form introduced in \eqref{eq:def_cV} for the choice of $(\tau_{ 1, u}, \ldots, \tau_{ N, u})$. By \eqref{eq:contV}, one obtains (writing $w( \tau_{ k, u}, \tau_{ i, u})$ instead of $w( \tau_{ k, u}, \tau_{ i, u}, \kappa_{ 1}, \iota_{ 1})$),
\begin{align*}
A_{ N, 1}&\leq\frac{ Cm}{ \varepsilon_{ 1}\left\vert \Lambda_{ N} \right\vert} \sum_{ i\in \Lambda_{ N}}\bE \left( \int_{ T_{ N}}^{ T_{ N}+s}  \left(\frac{1}{ \left\vert \Lambda_{ N} \right\vert} \sum_{ k\in \Lambda_{ N}}w( \tau_{ k, u}, \tau_{ i, u}) \Psi(x_{ k}, x_{ i})\right)^{ 2}\dd u\right),\\
&=\frac{ Cs}{ \varepsilon_{ 1}\left\vert \Lambda_{ N} \right\vert^{ 3}} \sum_{ i, k, l\in \Lambda_{ N}}\Psi(x_{ k}, x_{ i}) \Psi(x_{ l}, x_{ i})\bE \left(\sup_{ u\leq T}w( \tau_{ k, u}, \tau_{ i, u}) w( \tau_{ l, u}, \tau_{ i, u})\right)\leq \frac{ Cs}{ \varepsilon_{ 1}},
\end{align*}
using Proposition~\ref{prop:moment_particles} and Remark~\ref{rem:Psi_bounded}. This term is lower or equal than any $ \varepsilon_{ 2}>0$ provided $s$ is sufficiently small. The treatment of the remaining terms can be done in the same way, using the continuity of the linear form $\cW_{ N}$ defined in \eqref{eq:def_cW}.
\end{proof}
\begin{theorem}
\label{theo:HN_tight}
The sequence of the laws of $(\cH_{ N})_{ N\geq1}$ is tight in $\cC([0, T], \bW_{ -(q+P+2)}^{ \kappa_{ 0}, \iota_{ 0}})$.
\end{theorem}
\begin{proof}[Proof of Theorem~\ref{theo:HN_tight}]
Proposition~\ref{prop:etaNbounded1} implies, as in the proof of Proposition ~\ref{prop:MNH_tight}, that the first item of the tightness criterion of Section~\ref{sec:tightness_crit} is satisfied for the process $\cH_{ N}$. We verify now the second part of the tightness criterion. Since it is verified for the process $\cM_{ N}^{ (\cH)}$ (see Proposition ~\ref{prop:MNH_tight}) it suffices to prove it for the three terms (recall \eqref{eq:HN_semimart}) $\int_{0}^{t} \mathscr{L}_{ s}^{ \ast} \cH_{ N, s}\dd s$, $\int_{0}^{t} F_{ N, s} \dd s$ and $\int_{0}^{t} G_{ N, s}\dd s$. Concerning $\int_{0}^{t} \mathscr{L}_{ s}^{ \ast} \cH_{ N, s}\dd s$, if $( \psi_{ p})_{ p\geq1}$ is a complete orthonormal system in $\bW_{ q+P+2}^{ \kappa_{ 0}, \iota_{ 0}}$, we have successively, for all $ \varepsilon_{ 1}>0$ and stopping time $T_{ N}\leq T$,
\begin{align*}
\bP &\left(\N{ \int_{0}^{T_{ N}+ s}\mathscr{L}_{ u}^{ \ast} \cH_{ N, u}\dd u - \int_{0}^{T_{ N}}\mathscr{L}_{ u}^{ \ast} \cH_{ N, u}\dd u }_{ -(q+P+2), \kappa_{ 0}, \iota_{ 0}}\geq \varepsilon_{ 1}\right)\\
&\leq \frac{ 1}{ \varepsilon_{ 1}^{ 2}}\bE \left(\N{ \int_{T_{ N}}^{T_{ N}+s} \mathscr{L}_{ u}^{ \ast} \cH_{ N, u}\dd u}_{ -(q+P+2), \kappa_{ 0}, \iota_{ 0}}^{ 2}\right),\\
&\leq \frac{ s}{ \varepsilon_{ 1}^{ 2}}\bE \left( \int_{T_{ N}}^{T_{ N}+s} \N{\mathscr{L}_{ u}^{ \ast} \cH_{ N, u}}_{ -(q+P+2), \kappa_{ 0}, \iota_{ 0}}^{ 2}\dd u\right),\\
&\leq \frac{ s}{ \varepsilon_{ 1}^{ 2}}\bE \left( \int_{T_{ N}}^{T_{ N}+s} \sum_{ p\geq1} \cro{ \cH_{ N, u}}{ \mathscr{L}_{ u} \psi_{ p}}^{ 2}\dd u\right),\\
&\leq \frac{ Cs}{ \varepsilon_{ 1}^{ 2}}\bE \left( \int_{T_{ N}}^{T_{ N}+s} \N{\cH_{ N, u}}_{ -q, \kappa_{ 1}, \iota_{ 1}}^{ 2}\dd u\right).
\end{align*}
By Proposition ~\ref{prop:etaNbounded1}, this expectation is finite. The same kind of estimates can be proven concerning the remaining terms using Proposition ~\ref{prop:UN_VN} and Remark~\ref{rem:norms}. The process $\cH_{ N}$ satisfies the tightness criterion, and hence, Theorem ~\ref{theo:HN_tight} is proven.
\end{proof}
\subsection{Tightness of the fluctuation process $ \eta_{ N}$}
Note that it is possible to retrieve the fluctuation process $ \eta_{ N}$ from the two-particle process $\cH_{ N}$. Indeed, we easily see from \eqref{eq:rel_etaN_HN} that, for all function $ \tau \mapsto f(\tau)$, 
\begin{equation}
\label{eq:HN_to_etaN}
\left\langle \eta_{ N, t}\, ,\, f\right\rangle = \left\langle \cH_{ N, t}\, ,\, F\right\rangle,
\end{equation} for the choice of 
\begin{equation}
\label{eq:F}
F(\theta, \omega, x, \ttheta, \tomega, \tx):= d(x, \tx)^{ \alpha} f(\ttheta, \tomega, \tx)
\end{equation}
Unfortunately, deriving the tightness of $\eta_{ N}$ from the tightness of $\cH_{ N}$ is not easy: the main difficulty is that, whatever the regularity of $f$ may be, the function $F$ in \eqref{eq:F} will always be at most $ \alpha-$H\"older in the space variables $(x, \tx)$, which is not a sufficient regularity for the test functions and the Sobolev embeddings used in the previous section. 

The strategy we follow is simply to adapt the proof of the tightness of $\cH_{ N}$ to the set-up of $ \eta_{ N}$. Note that the proof becomes considerably easier since the definition of the process $ \eta_{ N}$ (contrary to $\cH_{ N}$) does not incorporate any spatial singularities: in the semimartingale decomposition \eqref{eq:semimart} of $ \eta_{ N}$, all the singularities in $x$ are inside the term involving $\cH_{ N}$. Due to the similarity between the proof of Theorem~\ref{theo:etaN_tight} and the arguments of Section~\ref{sec:tightness_two_particle}, we will only sketch the main lines of proof in this paragraph and leave the details to the reader. Recall the definitions of the weighted Sobolev spaces $\bV$ in Section~\ref{sec:weighted_sobolev_spaces_one_var}.
\begin{theorem}
\label{theo:etaN_tight}
For any $q\geq P+2$, the sequence of the laws of $(\eta_{ N})_{ N\geq1}$ is tight in $\cC([0, T], \bV_{ -(q+P+2)}^{ \kappa_{ 0}, \iota_{ 0}})$.
\end{theorem}
\begin{proof}[Proof of Theorem~\ref{theo:etaN_tight}]
Following the same procedure as in Proposition~\ref{prop:etaNbounded1}, it is easy to prove that $ \sup_{t\leq T} \bE \left(\N{ \eta_{ N,t}}_{-q, \kappa_{ 1}, \iota_{ 1}}^{2}\right) \leq \frac{ a_{ N}^{ 2}}{ \left\vert \Lambda_{ N} \right\vert}$. Here $\N{\cdot}_{ -q, \kappa_{ 1}, \iota_{ 1}}$ stands for the norm given by \eqref{eq:normj_one_var}. A uniform bound about $ \eta_{ N}$ similar to \eqref{eq:HN_bounded_2}, i.e. 
\begin{equation}
\label{eq:second_bound_etaN}
\sup_{1 \leq N} \bE \left(\sup_{ t\leq T} \N{ \eta_{ N, t}}_{ -(q+P+2), \kappa_{ 0}, \iota_{ 0}}^{ 2}\right)<+\infty,
\end{equation}
can be proven along the same lines as in the proof of Proposition~\ref{prop:HN_bounded_2}, using the semimartingale decomposition \eqref{eq:semimart} of $ \eta_{ N}$. The continuity of the underlying linear operator $L[ \nu_{ s}]$ from $\bV_{q+P+2}^{ \kappa_{ 0}, \iota_{ 0}}$ to $\bV_{ q}^{ \kappa_{ 1}, \iota_{ 1}}$ 
\begin{equation}
\label{eq:continuityL_eta}
\N{L[ \nu_{ s}]f}_{q, \kappa_{ 1}, \iota_{ 1}}\leq C \N{f}_{ q+P+2, \kappa_{ 0}, \iota_{ 0}}.
\end{equation}can be proven as in Proposition~\ref{prop:continuityL}. It is also immediate to derive bounds similar to Propositions~\ref{prop:MNH_bounded} and~\ref{prop:MNH_tight} about the martingale process $\cM_{ N}^{ ( \eta)}$ defined in \eqref{eq:def_MN_eta} (recall in particular Remark~\ref{rem:norms}), that is
\begin{equation}
\label{eq:MN_eta_bounded}
\sup_{ t\leq T}\bE \left(\N{\cM_{ N, t}^{ (\eta)}}_{ -q, \kappa_{ 1}, \iota_{ 1}}^{ 2}\right)\leq C \frac{ a_{ N}^{ 2}}{ \left\vert \Lambda_{ N} \right\vert}.
\end{equation}
 The only unusual term in \eqref{eq:semimart} is $\int_{0}^{t} \cro{\cH_{ N, s}}{ \Phi[f]} \dd s$ (recall the definition of $ \Phi[f]$ in \eqref{eq:Phif}). Using in particular the existence of a constant $C>0$ such that $ \left\vert \Phi[f] \right\vert \leq C \N{f}_{ q+P+2}$ for any test function $f\in \bV^{ \kappa_{ 0}, \iota_{ 0}}_{ q+P+2}$, this gives that $ \left\vert \cro{\cH_{ N, s}}{ \Phi[f]} \right\vert \leq C \N{\cH_{ N, s}}_{ -(q+P+2), \kappa_{ 0}, \iota_{ 0}} \N{f}_{ q+P+2, \kappa_{ 0}, \iota_{ 0}}$ and we can use the boundedness of the process $\cH_{ N}$ in $\bW_{ -(q+P+2)}^{ \kappa_{ 0}, \iota_{ 0}}$ proven in Proposition~\ref{prop:HN_bounded_2} to conclude. The rest of the proof of Theorem~\ref{theo:etaN_tight} follows the same lines as in the proof of Theorem~\ref{theo:HN_tight}.
\end{proof}
\section{Convergence of the fluctuations processes}
\label{sec:identification_limits}

The rest of the paper is devoted to the identification of the limits as $N\to\infty$ of the couple $( \eta_{ N}, \cH_{ N})$ in $\cC([0, T], \bV\oplus \bW)$ for some Sobolev spaces $\bV$ and $\bW$ defined in Section~\ref{sec:Hilbert_spaces}. Obviously, since each component of this couple is tight, the couple is also tight. 
\subsection{Identification of the initial value and the martingale part}
\label{sec:identification_ini_mart}
\subsubsection{Identification of the initial value}
\label{sec:ident_initvalue}
Note that at $t=0$, the marginal $ \xi_{ 0}(\dd \theta, \dd \omega)$ on $\bbX\times\bbY$ of the McKean-Vlasov process $ \nu_{ 0}$ \eqref{eq:cond_ini_meanfield} is equal to $ \zeta(\dd \theta)\otimes \mu(\dd \omega)$ (see~\eqref{eq:xi0} below).
\begin{proposition}
\label{prop:initial_cond_HN}
For any $q\geq P+2$, under the assumptions made in Section~\ref{sec:assumptions}, the sequence $(\eta_{ N, 0}, \cH_{ N, 0})_{ N\geq1}$ of the fluctuation processes at $t=0$ converges in law, as $N\to \infty$, in $\bV_{ -(q+P+2)}^{ \kappa_{ 0}, \iota_{ 0}} \oplus \bW_{ -(q+P+2)}^{ \kappa_{ 0}, \iota_{ 0}}$ to the process $(\eta_{ 0}, \cH_{ 0})$ defined by
\begin{itemize}
\item if $ \alpha< \frac{ 1}{ 2}$, $(\eta_{ 0}, \cH_{ 0})$ is a Gaussian process on $\bV_{ -(q+P+2)}^{ \kappa_{ 0}, \iota_{ 0}} \oplus \bW_{ -(q+P+2)}^{ \kappa_{ 0}, \iota_{ 0}}$ with covariance,
\begin{equation}
\label{eq:cov_H0_small_alpha}
C \left( \left(\begin{smallmatrix} f_{ 1}\\ g_{ 1}\end{smallmatrix}\right), \left(\begin{smallmatrix} f_{ 2}\\ g_{ 2}\end{smallmatrix}\right)\right):= C_{ \eta}( f_{ 1}, f_{ 2}) + C_{ \eta, \cH}( f_{ 1}, g_{ 2}) + C_{ \eta, \cH}( f_{ 2}, g_{ 1}) + C_{ \cH}( g_{ 1}, g_{ 2}),
\end{equation}
where
\begin{equation}
\label{eq:covariances_0}
\begin{split}
C_{ \eta}( f_{ 1}, f_{ 2})&:=  \frac{ 1}{ 2} \int_{ \bbS} \Big\{ \left\langle f_{ 1}(\cdot, \cdot, x) f_{ 2}(\cdot, \cdot, x)\, ,\, \xi_{ 0}\right\rangle - \left\langle f_{ 1}(\cdot, \cdot, x)\, ,\, \xi_{ 0}\right\rangle\left\langle f_{ 2}(\cdot, \cdot, x)\, ,\, \xi_{ 0}\right\rangle\Big\}\dd x\\
C_{ \eta, \cH}( f, g)&:= \frac{ 1}{ 2} \int_{ \bbS} \Big\{ \left\langle f(\cdot, \cdot, x)\left[ g\Psi, \xi_{ 0}\right](\cdot, \cdot, x) \, ,\, \xi_{ 0}\right\rangle - \left\langle f(\cdot,\cdot, x)\, ,\, \xi_{ 0}\right\rangle \left\langle \left[ g\Psi, \xi_{ 0}\right](\cdot, \cdot, x)\, ,\, \xi_{ 0}\right\rangle\Big\} \dd x\\
C_{ \cH}( g_{ 1}, g_{ 2})&:= \frac{ 1}{ 2}\int_{ \bbS} \Big\{ \left\langle \left[ g_{ 1} \Psi, \nu_{ 0}\right](\cdot, \cdot, x)\left[ g_{ 2} \Psi, \nu_{ 0}\right](\cdot, \cdot, x)\, ,\, \xi_{ 0}\right\rangle\\ &\qquad\qquad -\left\langle \left[ g_{ 1} \Psi, \nu_{ 0}\right](\cdot, \cdot, x)\, ,\, \xi_{ 0}\right\rangle \left\langle \left[ g_{ 2} \Psi, \nu_{ 0}\right](\cdot, \cdot, x)\, ,\, \xi_{ 0}\right\rangle \Big\}\dd x\\
\end{split}
\end{equation}
where $ f_{ 1}$, $ f_{ 2}$ and $ f$ (respectively $ g_{ 1}$, $ g_{ 2}$ and $ g$) are test functions on $\bbX\times\bbY\times\bbS$ (respectively $(\bbX\times\bbY\times\bbS)^{ 2}$), $ \left\langle f(\cdot, \cdot, x)\, ,\, \xi_{ 0}\right\rangle= \int_{ \bbX\times\bbY} f( \theta, \omega, x) \xi_{ 0}(\dd \theta, \dd \omega)$ (recall \eqref{eq:cro_f_lambda}) and where $\left[ g \Psi, \nu_{ 0}\right](\ttheta, \tomega, \tx)=\left[ g \Psi, \nu_{ 0}\right]_{ 2}(\ttheta, \tomega, \tx)$ (recall \eqref{eq:cro_2}).

\item
if $ \alpha> \frac{ 1}{ 2}$, $(\eta_{ 0}, \cH_{ 0})$ is a deterministic process given by
\begin{equation}
\label{eq:H0_large_alpha}
\left\{\begin{split}
\eta_{ 0}&\equiv 0,\\
\left\langle \cH_{ 0}\, ,\, g\right\rangle&:= \chi(\alpha) \int_{ (\bbX\times\bbY)^{ 2}\times\bbS} g(\theta, \omega, x, \ttheta, \tomega, x) \xi_{ 0}(\dd \theta, \dd \omega)\xi_{ 0}(\dd \ttheta, \dd \tomega)\dd x,
\end{split}\right.
\end{equation}
for any bounded $\cC^{ 2}$ bounded function $ g$ with bounded derivatives and where $ \chi(\alpha)$ is the constant defined in Lemma~\ref{lem:fluct_Psi}.
\end{itemize}
\end{proposition}
Proof of Proposition~\ref{prop:initial_cond_HN} is given in Section~\ref{sec:proof_initial_cond_HN}.

\subsubsection{Identification of the martingale part}
\label{sec:ident_martpart}
We identify in this paragraph the limit of the martingale part $\overline{\cM}_{ N}:=(\cM_{ N}^{ (\eta)}, \cM_{ N}^{ (\cH)})$ in the semi-martingale decomposition of the process $(\eta_{ N}, \cH_{ N})$ \eqref{eq:semimart} and \eqref{eq:semimart_HN}. Here $\overline{\cM}_{ N}$ is seen as a square integrable martingale in $\bV_{ -(q+P+2)}^{ \kappa_{ 0}, \iota_{ 0}}\oplus \bW_{ -(q+P+2)}^{ \kappa_{ 0}, \iota_{ 0}}$, for any $q\geq P+2$.
\begin{definition}
\label{def:covariance_WH}
Define the process $\overline{\cM}= (\cM^{ (\eta)}, \cM^{ (\cH)})\in \bV_{ -(q+P+2)}^{ \kappa_{ 0}, \iota_{ 0}}\oplus \bW_{ -(q+P+2)}^{ \kappa_{ 0}, \iota_{ 0}}$ as follows:
\begin{enumerate}
\item If $ \alpha< \frac{ 1}{ 2}$, $\overline{\cM}$ is the Gaussian process with covariance,
\begin{equation}
\label{eq:covariance_WH}
\begin{split}
\cK_{ s, t}\left(\left(\begin{smallmatrix} f_{ 1}\\ g_{ 1}\end{smallmatrix}\right), \left(\begin{smallmatrix} f_{ 2}\\ g_{ 2}\end{smallmatrix}\right)\right)&:=\bE \left( \left(\cM_{ s}^{ (\eta)}(f_{ 1}) + \cM_{ s}^{ (\cH)}(g_{ 1})\right) \left(\cM_{ t}^{ (\eta)}(f_{ 2}) + \cM_{ t}^{ (\cH)}(g_{ 2})\right)\right),\\&= \cK^{ (\eta)}_{ s, t}(f_{ 1}, f_{ 2}) + \cK^{ (\eta, \cH)}_{ s, t}(f_{ 1}, g_{ 2}) + \cK^{ (\eta, \cH)}_{ s, t}(f_{ 2}, g_{ 1}) + \cK^{ (\cH)}_{ s, t}(g_{ 1}, g_{ 2}),
\end{split}
\end{equation}
where $\left(\begin{smallmatrix} f_{ 1}\\ g_{ 1}\end{smallmatrix}\right)$, $\left(\begin{smallmatrix} f_{ 2}\\ g_{ 2}\end{smallmatrix}\right)\in \bV_{ q+P+2}^{ \kappa_{ 0}, \iota_{ 0}}\oplus \bW_{ q+P+2}^{ \kappa_{ 0}, \iota_{ 0}}$, $s,t\in[0, T]$ and
\begin{equation}
\label{eq:def_covariances_mart}
\begin{split}
\cK^{ (\eta)}_{ s, t}(f_{ 1}, f_{ 2})&:=\frac{ 1}{ 2} \int_{0}^{s\wedge t}\int \sum_{ r=1}^{ m} \partial_{ \theta^{ (r)}} f_{ 1}(\tau)\partial_{ \theta^{ (r)}} f_{ 2}(\tau) \nu_{ u}(\dd \tau)\dd u,\\
\cK^{ (\eta, \cH)}_{ s, t}(f, g)&:= \frac{ 1}{ 2}\int_{0}^{s\wedge t} \int\sum_{ r=1}^{ m} \partial_{ \theta^{ (r)}}f(\ttau) \left[(\partial_{ \ttheta^{ (r)}}g)\Psi, \nu_{ u}\right]_{ 2}(\ttau) \nu_{ u}(\dd \ttau)\dd u\\
\cK^{ (\cH)}_{ s, t}(g_{ 1}, g_{ 2})&:=\frac{ 1}{ 2} \int_{0}^{t} \int\sum_{ r=1}^{ m}\left[(\partial_{ \ttheta^{ (r)}}g_{ 1})\Psi, \nu_{ u}\right]_{ 2}(\ttau)\left[(\partial_{ \ttheta^{ (r)}}g_{ 2})\Psi, \nu_{ u}\right]_{ 2}(\ttau) \nu_{ u}(\dd \ttau)\dd u.
\end{split}
\end{equation}
\item
If $ \alpha> \frac{ 1}{ 2}$, $\overline{\cM}=(\cM^{ (\eta)}, \cM^{ (\cH)})$ is equally $0$.
\end{enumerate}
\end{definition}
\begin{proposition}
\label{prop:martingale_part_H}
For any $q\geq P+2$, the process $\overline{\cM}_{ N}:=(\cM_{ N}^{ (\eta)}, \cM_{ N}^{ (\cH)})_{ N\geq1}$ is a martingale in $\bV_{ -(q+P+2)}^{ \kappa_{ 0}, \iota_{ 0}}\oplus \bW_{ -(q+P+2)}^{ \kappa_{ 0}, \iota_{ 0}}$ with Doob-Meyer process $\llangle \overline{\cM}_{ N}\rrangle$ with values in $L(\bV_{ q+P+2}^{ \kappa_{ 0}, \iota_{ 0}}\oplus \bW_{ q+P+2}^{ \kappa_{ 0}, \iota_{ 0}}, \bV_{ -(q+P+2)}^{ \kappa_{ 0}, \iota_{ 0}}\oplus \bW_{ -(q+P+2)}^{ \kappa_{ 0}, \iota_{ 0}})$ given by
\begin{equation}
\label{eq:DM_MN_bar}
\begin{split}
\llangle \overline{\cM}_{ N}\rrangle_{ t}  \left(\begin{smallmatrix} f_{ 1}\\ g_{ 1}\end{smallmatrix}\right)\left(\begin{smallmatrix} f_{ 2}\\ g_{ 2}\end{smallmatrix}\right) &= \llangle \cM_{ N}^{ (\eta)}\rrangle_{ t}( f_{ 1})(f_{ 2}) + \llangle \cM_{ N}^{ (\eta, \cH)}\rrangle_{ t}( f_{ 2})(g_{ 1}) + \llangle \cM_{ N}^{ (\eta, \cH)}\rrangle_{ t}( f_{ 1})(g_{ 2})\\ &+ \llangle \cM_{ N}^{ (\cH)}\rrangle_{ t}( g_{ 1})(g_{ 2}),
\end{split}
\end{equation}
In \eqref{eq:DM_MN_bar}, $\left(\begin{smallmatrix} f_{ 1}\\ g_{ 1}\end{smallmatrix}\right)$, $\left(\begin{smallmatrix} f_{ 2}\\ g_{ 2}\end{smallmatrix}\right)\in \bV_{ q+P+2}^{ \kappa_{ 0}, \iota_{ 0}}\oplus \bW_{ q+P+2}^{ \kappa_{ 0}, \iota_{ 0}}$ and \begin{equation}
\label{eq:DM_MN_eta}
\begin{split}
\llangle \cM_{ N}^{ (\eta)}\rrangle_{ t}( f_{ 1})(f_{ 2})&=\cro{ \cM_{ N}^{ (\eta)}f_{ 1}}{ \cM_{ N}^{ (\eta)} f_{ 2}}_{ t},\\
&= \frac{ a_{ N}^{ 2}}{ \left\vert \Lambda_{ N} \right\vert^{ 2}}\sum_{ i\in \Lambda_{ N}} \int_{0}^{t} \sum_{ r=1}^{ m}\partial_{ \theta^{ (r)}} f_{ 1}( \theta_{ i, s}, \omega_{ i}, x_{ i}) \partial_{ \theta^{ (r)}} f_{ 2}( \theta_{ i, s}, \omega_{ i}, x_{ i}) \dd s
\end{split}
\end{equation}
is the Doob-Meyer in $L(\bV_{ q+P+2}^{ \kappa_{ 0}, \iota_{ 0}}, \bV_{ -(q+P+2)}^{ \kappa_{ 0}, \iota_{ 0}})$ associated to $ \cM_{ N}^{ (\eta)}$, $\llangle \cM_{ N}^{ (\cH)}\rrangle_{ t}$ is the Doob-Meyer process associated to $ \cM_{ N}^{ (\cH)}$ already calculated in \eqref{eq:cro_Hilb_MNH} and 
\begin{equation}
\label{eq:DM_MN_eta_H}
\begin{split}
\llangle \cM_{ N}^{ (\eta, \cH)}\rrangle_{ t}( f)(g)&:=\cro{ \cM_{ N}^{ (\eta)}f}{ \cM_{ N}^{ (\cH)} g}_{ t},\\
&= \frac{ a_{ N}^{ 2}}{ \left\vert \Lambda_{ N} \right\vert^{ 3}} \sum_{ i, j\in \Lambda_{ N}} \int_{0}^{t} \sum_{ r=1}^{ m} \partial_{ \theta^{ (r)}} f(\tau_{ j, s}) \partial_{ \ttheta^{ (r)}}g(\tau_{ i, s}, \tau_{ j, s}) \Psi(x_{ i}, x_{ j})\dd s\\
&+ \frac{ a_{ N}^{ 2}}{ \left\vert \Lambda_{ N} \right\vert^{ 2}} \sum_{ i\in \Lambda_{ N}}\int_{0}^{t} \sum_{ r=1}^{ m} \partial_{ \theta^{ (r)}}f(\tau_{ i, s}) \epsilon_{ N, s}^{ (r)}(\tau_{ i, s}, g)\dd s,
\end{split}
\end{equation}
where $ \epsilon_{ N, s}^{ (r)}$ is given in \eqref{eq:def_bN}. Moreover, under the assumptions made in Section~\ref{sec:assumptions}, the process $(\overline{\cM}_{ N})_{ N\geq1}$ converges in law in $\bV_{ -(q+P+2)}^{ \kappa_{ 0}, \iota_{ 0}}\oplus \bW_{ -(q+P+2)}^{ \kappa_{ 0}, \iota_{ 0}}$, as $N\to\infty$, to the process $\overline{\cM}$ given in Definition~\ref{def:covariance_WH}.
\end{proposition}
\begin{proof}[Proof of Proposition~\ref{prop:martingale_part_H}]
Estimates \eqref{eq:DM_MN_bar}, \eqref{eq:DM_MN_eta} and \eqref{eq:DM_MN_eta_H} are a straightforward consequence of the definitions of the martingale terms $\cM_{ N}^{ (\eta)}$ in \eqref{eq:def_MN_eta} and $\cM_{ N}^{ (\cH)}$ in \eqref{eq:def_MNH} and of the independence of the Brownian motions $B_{ i}= (B_{ i}^{ (1)}, \ldots, B_{ i}^{ (m)})$,  $i\in \Lambda_{ N}$. It remains to prove the convergence of $\overline{\cM}_{ N}$ to $ \overline{\cM}$. First remark that the case $ \alpha> \frac{ 1}{ 2}$ is a trivial consequence of the estimates \eqref{eq:MNH_bounded2}, \eqref{eq:MN_eta_bounded} and the definition of $a_{ N}$ in \eqref{eq:aN}. 

In the case of $ \alpha< \frac{ 1}{ 2}$, $a_{ N}= \sqrt{N}$ and $(\overline{\cM}_{ N})_{ N\geq1}$ is a sequence of uniformly square-integrable continuous martingales which is tight in $\bV_{ -2(P+2)}^{ \kappa_{ 0}, \iota_{ 0}}\oplus \bW_{ -2(P+2)}^{ \kappa_{ 0}, \iota_{ 0}}$. Let $\overline{\cM}$ be an accumulation point. By arguments analogous to the ones continuously used in Section~\ref{sec:tightness_results} (see in particular Proposition~\ref{prop:etaNbounded1}), it is easy to show that, first, $ \epsilon_{ N}^{ (r)}$ defined in \eqref{eq:def_bN} vanishes to $0$ as $N\to\infty$ and, second, that for all regular test function $\left(\begin{smallmatrix} f\\ g\end{smallmatrix}\right)$,  one can make $N\to\infty$ in the expression of $\llangle \overline{\cM}_{ N}\rrangle_{ t}  \left(\begin{smallmatrix} f\\ g\end{smallmatrix}\right)$ in \eqref{eq:DM_MN_bar} and obtain for all $t\in[0, T]$
\begin{align*}
\llangle \overline{\cM}\rrangle_{ t}  \left(\begin{smallmatrix} f\\ g\end{smallmatrix}\right)&= \frac{ 1}{ 2} \int_{0}^{t} \int\sum_{ r=1}^{ m} \left(\partial_{ \theta^{ (r)}} f(\tau)\right)^{ 2} \nu_{ s}(\dd \tau)\dd s\\ &+ \int_{0}^{t} \int\sum_{ r=1}^{ m} \partial_{ \theta^{ (r)}}f(\ttau) \left[(\partial_{ \ttheta^{ (r)}}g)\Psi, \nu_{ s}\right]_{ 2}(\ttau) \nu_{ s}(\dd \ttau)\dd s\\ &+ \frac{ 1}{ 2} \int_{0}^{t} \int\sum_{ r=1}^{ m}\left[(\partial_{ \ttheta^{ (r)}}g)\Psi, \nu_{ s}\right]_{ 2}(\ttau)^{ 2} \nu_{ s}(\dd \ttau)\dd s.
\end{align*}
So, $\overline{\cM}$ is a square integrable martingale whose Doob-Meyer process is deterministic, given as above. So $\overline{\cM}$ is characterized by the Gaussian process with covariance given in \eqref{eq:covariance_WH}. The convergence follows.
\end{proof}

\begin{proposition}
\label{prop:independence}
In the case of $ \alpha< \frac{ 1}{ 2}$, the processes $(\eta_{ 0}, \cH_{ 0})$ introduced in Proposition~\ref{prop:initial_cond_HN} and $\overline{\cM}=(\cM^{ (\eta)}, \cM^{ (\cH)})$ defined in Definition~\ref{def:covariance_WH} are independent.
\end{proposition}
\begin{proof}[Proof of Proposition~\ref{prop:independence}]
It is a straightforward consequence of the mutual independence of $(\theta_{ i, 0}, \omega_{ i}, B_{ i})$, $i\in \Lambda_{ N}$.
\end{proof}
\subsection{Identification of the limit}
\label{sec:identification_limit}
So far, all the analysis has been made for general Sobolev indices $(q, \kappa_{ 0}, \iota_{ 0})$.  We now specify the adequate parameters: recall the definition of $(\underline{ \kappa}, \underline{ \iota})$ in \eqref{eq:kappas} and \eqref{eq:iotas} and define 
\begin{equation}
\label{eq:q_kappas_1}
q:=P+2 \text{ and }  (\kappa_{ 0}, \iota_{ 0}):=(\underline{ \kappa} + \gamma, \underline{ \iota}+ \gamma),
\end{equation}
so that by definition of $( \kappa_{ 1}, \iota_{ 1})$, $ \kappa_{ 1}= \bar \kappa$ and $ \iota_{ 1}= \bar \iota$ (recall \eqref{eq:kappas} and \eqref{eq:kappas_01}).
The conclusion of Section~\ref{sec:tightness_results} for this choice of parameters is that the process $(\eta_{ N}, \cH_{ N})_{ N\geq1}$ is tight in $\cC([0, T], \bV_{ -2(P+2)}^{ \underline{\kappa} + \gamma, \underline{\iota} + \gamma}\oplus \bW_{ -2(P+2)}^{ \underline{\kappa} + \gamma, \underline{\iota} + \gamma})$ and that any of its accumulation point $(\eta, \cH)$ belongs to the same space. In order to have a closed formula for $(\eta, \cH)$, we need to increase the regularity of the test functions, and hence, enlarge the space of distributions: we will give a characterization of $(\eta, \cH)$ in the larger space $\cC([0, T], \bV_{ -3(P+2)}^{ \underline{\kappa}, \underline{\iota} }\oplus \bW_{ -3(P+2)}^{ \underline{\kappa}, \underline{\iota}})$.
\begin{proposition}
\label{prop:ident_limit_of_subsequences}
Under the assumptions of Section~\ref{sec:assumptions}, the process $(\eta_{ N}, \cH_{ N})_{ N\geq1}$ has convergent subsequences in $\cC([0, T], \bV_{ -2(P+2)}^{ \underline{\kappa} + \gamma, \underline{\iota} + \gamma}\oplus \bW_{ -2(P+2)}^{ \underline{\kappa} + \gamma, \underline{\iota} + \gamma})$ and any accumulation point $(\eta, \cH)$ is a solution in $\bV_{ -3(P+2)}^{ \underline{ \kappa}, \underline{ \iota}}\oplus \bW_{ -3(P+2)}^{  \underline{ \kappa}, \underline{ \iota}}$ to the following system of coupled equations
\begin{equation}
\label{eq:SPDE_general}
\left\{
\begin{split}
\eta_{ t} &= \eta_{ 0} + \int_{0}^{t} L[ \nu_{ s}]^{ \ast} \eta_{ s}\dd s + \int_{0}^{t} \Phi^{ \ast} \cH_{ s}\dd s + \cM^{ (\eta)}_{ t},\\
\cH_{ t}&= \cH_{ 0} + \int_{0}^{t} \mathscr{L}_{ s}^{ \ast} \cH_{ s} \dd s + \cM^{ (\cH)}_{ t},
\end{split}\quad t\in[0, T],\right.
\end{equation}
where $(\eta_{ 0}, \cH_{ 0})$ is defined in Proposition~\ref{prop:initial_cond_HN}, $(\cM^{ (\eta)}, \cM^{ (\cH)})$ is given in Definition~\ref{def:covariance_WH} and $L[ \nu_{ s}]^{ \ast}$ (respectively $ \Phi^{ \ast}$ and $\mathscr{L}_{ s}^{ \ast}$) is the dual of $L[ \nu_{ s}]$ defined in \eqref{eq:propagL} (respectively of $\Phi$ defined in \eqref{eq:Phif} and $\mathscr{L}_{ s}$ defined in \eqref{eq:def_opL}).
\end{proposition}
\begin{proof}[Proof of Proposition~\ref{prop:ident_limit_of_subsequences}]
Consider a subsequence (that we rename by $(\eta_{ N}, \cH_{ N})$ for simplicity) converging in $\cC([0, T], \bV_{ -2(P+2)}^{ \underline{\kappa} + \gamma, \underline{\iota} + \gamma}\oplus \bW_{ -2(P+2)}^{ \underline{\kappa} + \gamma, \underline{\iota} + \gamma})$ to $(\eta, \cH)$. We easily deduce from \eqref{eq:HN_bounded_2} and \eqref{eq:second_bound_etaN} that 
\begin{equation}
\label{aux:bE_H_eta}
\bE \left(\sup_{ t\leq T} \left(\N{\cH_{ t}}_{ -2(P+2), \underline{ \kappa} + \gamma, \underline{ \iota}+ \gamma}^{ 2} + \N{ \eta_{ t}}_{ -2(P+2), \underline{ \kappa} + \gamma, \underline{ \iota}+ \gamma}^{ 2}\right)\right)<+\infty.
\end{equation}
An application of \eqref{eq:continuityL}, \eqref{eq:Phi_regularity} and \eqref{eq:continuityL_eta} in the case of $q=2(P+2)$, $ \kappa_{ 0}= \underline{ \kappa}$ and $ \iota_{ 0}= \underline{ \iota}$ leads to
\begin{equation}
\label{aux:continuity_H_eta}
\begin{cases}
\N{L[ \nu_{ s}]f}_{2(P+2), \underline{ \kappa} + \gamma, \underline{ \iota} + \gamma}&\leq C \N{f}_{ 3(P+2), \underline{ \kappa}, \underline{ \iota}},\\
\N{ \Phi[f]}_{2(P+2), \underline{ \kappa} + \gamma, \underline{ \iota} + \gamma}&\leq C \N{f}_{ 3(P+2), \underline{ \kappa}, \underline{ \iota}},\\
\N{\mathscr{L}_{ s}g}_{2(P+2), \underline{ \kappa} + \gamma, \underline{ \iota} + \gamma}&\leq C \N{g}_{ 3(P+2), \underline{ \kappa}, \underline{ \iota}}.
\end{cases}
\end{equation}
Estimate \eqref{aux:bE_H_eta} together with \eqref{aux:continuity_H_eta} shows that $ \int_{0}^{t} L[ \nu_{ s}]^{ \ast} \eta_{ s}\dd s$ and $ \int_{0}^{t} \Phi^{ \ast} \cH_{ s} \dd s$ (respectively $ \int_{0}^{t} \mathscr{L}_{ s}^{ \ast} \cH_{ s}\dd s$) make sense as a Bochner integral in $\bV_{ -3(P+2)}^{ \underline{ \kappa}, \underline{ \iota}}$ (respectively in $\bW_{ -3(P+2)}^{ \underline{ \kappa}, \underline{ \iota}}$). Furthermore, Proposition~\ref{prop:initial_cond_HN} and~\ref{prop:martingale_part_H} concerning the convergence of the initial value $( \eta_{ N, 0}, \cH_{ N, 0})$ and the martingale part $(\cM_{ N}^{ (\eta)}, \cM_{ N}^{ (\cH)})$ and Proposition~\ref{prop:UN_VN} are also valid for the choice of the parameters $q=2(P+2)$, $ \kappa_{ 0}= \underline{ \kappa}$ and $ \iota_{ 0}= \underline{ \iota}$. Moreover, for every $\left(\begin{smallmatrix} f\\ g\end{smallmatrix}\right)\in \bV_{ 3(P+2)}^{ \underline{\kappa}, \underline{\iota} }\oplus \bW_{ 3(P+2)}^{ \underline{\kappa}, \underline{ \iota}}$,
\begin{equation}
\left\{
\begin{split}
\cro{\etaN{t}}{f} &= \cro{\etaN{0}}{f} + \int_{0}^{t} \cro{\etaN{s}}{L[\nu_{ s}]f}\dd s+ \int_{0}^{t} \cro{\cH_{ N, s}}{ \Phi[f]} \dd s+ \cM_{ N, t}^{( \eta)}f,\\
\cro{ \cH_{ N, t}}{ g}&= \cro{ \cH_{ N, 0}}{ g} + \int_{0}^{t} \cro{ \cH_{ N, s}}{ \mathscr{L}_{ s}g}\dd s + \int_{0}^{t} F_{ N, s}g \dd s + \int_{0}^{t} G_{ N, s}g \dd s +  \cM^{ (\cH)}_{ N, t}g.
\end{split}\quad t\in[0, T],\right.
\end{equation}
For fixed $\left(\begin{smallmatrix} f\\ g\end{smallmatrix}\right)$ in $\bV_{ 3(P+2)}^{ \underline{\kappa}, \underline{\iota} }\oplus \bW_{ 3(P+2)}^{ \underline{\kappa}, \underline{ \iota}}$, define the continuous functional from $\cC([0, T], \bV_{ -3(P+2)}^{ \underline{\kappa}, \underline{\iota} }\oplus \bW_{ -3(P+2)}^{ \underline{\kappa}, \underline{ \iota}})$ to $\bbR^{ 2}$ by \[\Pi_{ f, g}(u, v):= \begin{pmatrix} \cro{u_{ t}}{f} - \cro{u_{ 0}}{f} - \int_{0}^{t} \cro{u_{ s}}{L[\nu_{ s}]f}\dd s- \int_{0}^{t} \cro{v_{ s}}{ \Phi[f]} \dd s\\ \cro{ v_{ t}}{ g}- \cro{ v_{ 0}}{ g} - \int_{0}^{t} \cro{ v_{ s}}{ \mathscr{L}_{ s}g}\dd s\end{pmatrix}.\] Since $\left(\begin{smallmatrix} f\\ g\end{smallmatrix}\right)$ in $\bV_{ 3(P+2)}^{ \underline{\kappa}, \underline{\iota} }\oplus \bW_{ 3(P+2)}^{ \underline{\kappa}, \underline{ \iota}} \hookrightarrow \bV_{ 2(P+2)}^{ \underline{\kappa} + \gamma, \underline{\iota} + \gamma}\oplus \bW_{ 2(P+2)}^{ \underline{\kappa} + \gamma, \underline{\iota} + \gamma}$ and since $(\eta, \cH)$ is an accumulation point of $(\eta_{ N}, \cH_{ N})$ in $\cC([0, T], \bV_{ -2(P+2)}^{ \underline{\kappa} + \gamma, \underline{\iota} + \gamma}\oplus \bW_{ -2(P+2)}^{ \underline{\kappa} + \gamma, \underline{\iota} + \gamma})$, one has that $\Pi_{ f, g}( \eta_{ N}, \cH_{ N})\to\Pi_{ f, g}( \eta, \cH)$, as $N\to\infty$, which ends the proof of Proposition~\ref{prop:ident_limit_of_subsequences}.
\end{proof}
\subsection{Some uniqueness results}
\label{sec:uniqueness_result}
We derive in this section two uniqueness results that will be useful in the following. The first one concerns functional equations driven by the linear operator $ \mathscr{L}_{ s}^{ (1)}$ defined in \eqref{eq:def_opL1}. The second one addresses the same problem about the linear operator $L[ \nu_{ s}]$ given by \eqref{eq:propagL}. We treat in details the case of $\mathscr{L}_{ s}^{ (1)}$ in Section~\ref{sec:uniqueness_L1}. The treatment of the operator $L[ \nu_{ s}]$ being strictly identical, we only state the result in Section~\ref{sec:uniqueness_Lnu} and leave the proof to the reader.
\subsubsection{Uniqueness in linear functional equations driven by $\mathscr{L}_{ s}^{ (1)}$}
\label{sec:uniqueness_L1}
Recall the definition of the operator $\mathscr{L}_{ s}^{ (1)}$ in \eqref{eq:def_opL1}. For all $0\leq s \leq t\leq T $ and $( \tau, \ttau)\in (\bbX\times\bbY\times\bbS)^{ 2}$, define as $ X_{ s, t}( \tau, \ttau):= ( \Theta_{ s, t}(\tau), \tTheta_{ s, t}(\ttau))$ the unique solution of the following stochastic differential equation in $(\bbR^{ m})^{ 2}$
\begin{equation}
\left\{\begin{split}
\Theta_{ s, t}( \tau) &= \theta + \int_{s}^{t} v(t, \Theta_{ s, r}(\tau), \omega) \dd r + B_{ t} - B_{ s},\\
\tTheta_{ s, t}( \ttau) &= \ttheta + \int_{s}^{t} v(t, \tTheta_{ s, r}(\tau), \omega) \dd r + \tB_{ t} - \tB_{ s},
\end{split}\right.
\end{equation}
for two independent Brownian motions $B$ and $\tB$ in $\bbR^{ m}$, $(\tau, \ttau)=( \theta, \omega, x, \ttheta, \tomega, \tx)$ and where
\begin{equation}
\label{eq:def_v_sde}
v(t, \Theta, \omega):= c( \Theta, \omega) + \left[ \Gamma \Psi, \nu_{ t}\right](\Theta, \omega).
\end{equation}
We recall here that the quantity $ \left[ \Gamma \Psi, \nu_{ t}\right]$ is indeed independent of $x$ (Remark~\ref{rem:Psix}). Define for any test function $g$
\begin{equation}
\label{eq:def_Uts}
U(t, s)g(\tau, \ttau)= \bE_{ B, \tB}(g(X_{ st}(\tau, \ttau))).
\end{equation}
\begin{proposition}
\label{prop:unique_L1}
Under the assumptions of Section~\ref{sec:assumptions}, for any functional $R$ in $\cC([0, T], \bW_{ -3(P+2)}^{ \underline{ \kappa}, \underline{ \iota}})$, there is at most one solution in $\cC([0, T], \bW_{ -2(P+2)}^{ \underline{ \kappa}+ \gamma, \underline{ \iota}+ \gamma})$ to the equation
\begin{equation}
\label{eq:cE_linear_eq}
\cE_{ t} = \int_{0}^{t}  \left(\mathscr{L}^{ (1)}_{ s}\right)^{ \ast} \cE_{ s} \dd s + \int_{0}^{t} R_{ s} \dd s,\quad t\in[0, T], \text{ in $ \bW_{ -3(P+2)}^{ \underline{ \kappa}, \underline{ \iota}}$}
\end{equation}
Moreover, one has the representation
\begin{equation}
\label{eq:repr_cE}
\cE_{ t}= \int_{0}^{t} U(t, s)^{ \ast} R_{ s}\dd s,\text{ in $C_{ -2(P+2)}^{ 0, 0}$},
\end{equation}
where $U$ is given in \eqref{eq:def_Uts}.
\end{proposition}
Proof of Proposition~\ref{prop:unique_L1} is given in Section~\ref{sec:proof_unique_L1}.

\subsubsection{Uniqueness for linear functional equations driven by $L[ \nu_{ s}]$}
\label{sec:uniqueness_Lnu}
We state here a result similar to Proposition~\ref{prop:unique_L1} for the operator $L[ \nu_{ s}]$ defined in \eqref{eq:propagL}. Since the proof is identical to Proposition~\ref{prop:unique_L1}, we only state the result. 

For all $0\leq s \leq t\leq T $ and $\tau=( \theta, \omega, x)\in \bbX\times\bbY\times\bbS$, define as $ Y_{ s, t}( \tau)$ the unique solution of the following stochastic differential equation in $\bbX$
\begin{equation}
Y_{ s, t}( \tau) = \theta + \int_{s}^{t} v(t, Y_{ s, r}(\tau), \omega) \dd r + B_{ t} - B_{ s},
\end{equation}
for $B$ Brownian motion in $\bbX$ and $v$ defined as in \eqref{eq:def_v_sde} and introduce the flow (for any test function $f$)
\begin{equation}
\label{eq:def_Vts}
V(t, s)f(\tau)= \bE_{ B}(f(Y_{ st}(\tau))).
\end{equation}
\begin{lemma}
\label{lem:regularity_V}
Under the assumptions of Section~\ref{sec:assumptions}, the operator $L[ \nu_{ t}]$ is continuous from $C_{ 3P+8}^{ 0, \iota}$ to $C_{ 3(P+2)}^{ 0, 2 \iota}$
\begin{align*}
\N{L[ \nu_{ t}]f}_{ C_{ 3(P+2)}^{ 0, 2 \iota}}&\leq C \N{f}_{ C_{ 3P+8}^{ 0, \iota}},\ t\in[0, t]\\
\N{L[ \nu_{ s}]f - L[ \nu_{ t}]f}_{ C_{ 3(P+2)}^{ 0, 2 \iota}}&\leq C \N{f}_{ C_{ 3P+8}^{ 0, \iota}} \left\vert t-s \right\vert,\ s,t\in[0, T].
\end{align*}
For any $j\leq 3P+8$, the operator $V(t, s)$ is a linear operator from $C_{ j}^{ 0, 0}$ to $C_{ j}^{ 0, \iota}$ such that
\begin{align}
\N{V(t, s) f}_{ C_{ j}^{ 0, \iota}}&\leq C \N{f}_{ C_{ j}^{ 0, 0}},\ 0\leq s\leq t\leq T,\label{eq:bound_Vts_Cj}\\
\N{V(t, s) f- V(t, s^{ \prime}) f}_{ C_{ j}^{ 0, \iota}}&\leq C \N{f}_{ C_{ j+1}^{ 0, 0}} \sqrt{s^{ \prime}-s},\ 0\leq s\leq s^{ \prime}\leq t\leq T.\label{eq:bound_Vts_diff_Cj}
\end{align}
Moreover, the following backward Kolmogorov equation holds: for every $f\in C_{ 3P+9}^{ 0, 0}$
\begin{equation}
\label{eq:back_Kolm_V_2}
V(t, s)f - f = \int_{s}^{t} L[ \nu_{ t}]V(t, r)f\dd r,\ \text{in $C_{ 3(P+2)}^{ 0, 2 \iota}$}.
\end{equation}
\end{lemma}
One can deduce from Lemma~\ref{lem:regularity_V} the proposition
\begin{proposition}
\label{prop:unique_Lnu}
Under the assumptions of Section~\ref{sec:assumptions}, for any functional $R$ in $\cC([0, T], \bV_{ -3(P+2)}^{ \underline{ \kappa}, \underline{ \iota}})$, there is at most one solution in $\cC([0, T], \bV_{ -2(P+2)}^{ \underline{ \kappa}+ \gamma, \underline{ \iota}+ \gamma})$ to the equation
\begin{equation}
\label{eq:cE_linear_eq_V}
\cE_{ t} = \int_{0}^{t} L[ \nu_{ t}]^{ \ast} \cE_{ s} \dd s + \int_{0}^{t} R_{ s} \dd s,\quad t\in[0, T], \text{ in $ \bV_{ -3(P+2)}^{ \underline{ \kappa}, \underline{ \iota}}$}
\end{equation}
Moreover, one has the representation
\begin{equation}
\label{eq:repr_cE_V}
\cE_{ t}= \int_{0}^{t} V(t, s)^{ \ast} R_{ s}\dd s,\text{ in $C_{ -2(P+2)}^{ 0, 0}$},
\end{equation}
where $V$ is given in \eqref{eq:def_Vts}.

\end{proposition}
\subsection{Convergence when $ \alpha< \frac{ 1}{ 2}$}
\subsubsection{Representation of $\cH$ in terms of $\eta$}
\label{sec:repr_H_eta}
The purpose of this paragraph is to prove that the representation result \eqref{eq:rel_eta_H} of the two-particle process $\cH$ in terms of the fluctuation process $ \eta$ is indeed true when $ \alpha< \frac{ 1}{ 2}$ (Section~\ref{sec:relation_HN_etaN}). More precisely, we prove
\begin{proposition}
\label{prop:repres_H_eta}
Suppose $ \alpha< \frac{ 1}{ 2}$. Under the assumptions of Section~\ref{sec:assumptions}, the process
\begin{equation}
\label{eq:cEN}
\cE_{ N, t}g:= \left\langle \cH_{ N, t}\, ,\, g\right\rangle - \left\langle \eta_{ N, t}\, ,\, \left\langle \nu_{ t}\, ,\, \Psi g\right\rangle\right\rangle,\ g\in \bW_{ -2(P+2)}^{ \underline{ \kappa} + \gamma, \underline{ \iota}+ \gamma},\ t\in[0, T]
\end{equation}
converges in law to $0$, in $\cC([0, T], \bW_{ -2(P+2)}^{ \underline{ \kappa} + \gamma, \underline{ \iota}+ \gamma})$, as $N\to\infty$.
\end{proposition}
\begin{remark}
\label{rem:uncoupling_eta_H}
Let us admit for a moment Proposition~\ref{prop:repres_H_eta}. This result (applied to the function $g= \Phi[f]$ for any $ \tau \mapsto f(\tau)$ where $ \Phi$ is defined in \eqref{eq:Phif}) combined with the analysis made in Section~\ref{sec:identification_limit} shows that each accumulation point $ \eta$ of $ \eta_{ N}$ is actually a solution in $\bV_{ -3(P+2)}^{ \underline{\kappa}, \underline{\iota} }$ to the uncoupled equation \eqref{eq:SPDE_subcrit_eta}.
\end{remark}

\begin{proof}[Proof of Proposition~\ref{prop:repres_H_eta}]
The strategy is the following: first, write the semimartingale decomposition the process $\cE_{ N}$ (based on the corresponding decompositions of $ \eta_{ N}$ and $\cH_{ N}$), second, prove that any accumulation point is a solution to a linear PDE and third, prove that the unique solution to this equation is the trivial solution $0$.

First observe that the tightness of the processes $\cH_{ N}$ and $ \eta_{ N}$ implies the tightness of the process $\cE_{ N}$ itself in $\cC([0, T], \bW_{ -2(P+2)}^{ \underline{ \kappa} + \gamma, \underline{ \iota}+ \gamma})$. Fix now a test function $(\tau, \ttau) \mapsto g( \tau, \ttau)$. An application of the semimartingale decomposition of $ \eta_{ N}$ \eqref{eq:semimart} to a regular time-dependent test function $(t, \tau) \mapsto f_{ t}(\tau)$ gives
\begin{align}
\cro{\etaN{t}}{f_{ t}} &= \cro{\etaN{0}}{f_{ 0}} + \int_{0}^{t} \cro{\etaN{s}}{L[\nu_{ s}]f_{ s}}\dd s+ \int_{0}^{t} \cro{\cH_{ N, s}}{ \Phi[f_{ s}]} \dd s+ \cM_{ N, t}^{( \eta)}f_{ t}\nonumber\\
& +  \int_{0}^{t} \left\langle \eta_{ N, s}\, ,\, \partial_{ s} f_{ s}\right\rangle \dd s.\label{aux:etaN_semi}
\end{align}
In the particular case of 
\begin{equation}
\label{eq:ft}
f_{ t}(\ttau)= \left\langle \nu_{ t}\, ,\, \Psi(\cdot, \tx)g(\cdot, \ttau)\right\rangle,
\end{equation}
we have, using \eqref{eq:nut},
\begin{align*}
\partial_{ t} f_{ t}(\ttau)&= \left\langle \nu_{ t}(\dd \tau)\, ,\, \Psi(x, \tx) \left(\frac{ 1}{ 2} \Delta_{ \theta}g(\tau, \ttau)  + \nabla_{\theta}g(\tau, \ttau) \cdot \left\lbrace c(\theta, \omega) + \Big[ \Gamma \Psi\, ,\, \nu_{ t}\Big](\tau) \right\rbrace\right) \right\rangle
\end{align*}
and
\begin{align*}
L[\nu_{ t}]f_{ t}(\ttau)&= \left\langle \nu_{ t}(\dd \tau)\, ,\, \Psi(x, \tx) \left(\frac{ 1}{ 2} \Delta_{ \ttheta}g(\tau, \ttau)  + \nabla_{\ttheta}g(\tau, \ttau) \cdot \left\lbrace c(\ttheta, \tomega) + \Big[ \Gamma \Psi\, ,\, \nu_{ t}\Big](\ttau) \right\rbrace\right)\right\rangle.
\end{align*}
By definition of $ \Phi[\cdot]$ in \eqref{eq:Phif}, we have also that
\begin{align*}
\Phi[f_{ t}](\tau, \ttau)&=\nabla_{\theta}f_{ t}(\theta, \omega, x)\cdot \Gamma( \theta, \omega, \ttheta, \tomega),\\
&= \left\langle \nu_{ t}\, ,\, \Psi(\cdot, x) \nabla_{\ttheta}g(\cdot, \tau)\right\rangle \cdot \Gamma( \theta, \omega, \ttheta, \tomega)= \mathscr{L}_{ t}^{ (2)}g( \tau, \ttau),
\end{align*}where the definition of $\mathscr{L}^{ (2)}_{ t}$ is given in \eqref{eq:def_opL2}. Consequently, one can rewrite \eqref{aux:etaN_semi} as 
\begin{align}
\cro{\etaN{t}}{ \left\langle \nu_{ t}\, ,\, \Psi g\right\rangle¥} &= \cro{\etaN{0}}{ \left\langle \nu_{ 0}\, ,\, \Psi g\right\rangle} + \int_{0}^{t} \cro{\etaN{s}}{ \left\langle \nu_{ s}\, ,\, \Psi \mathscr{L}_{ s}^{ (1)}g\right\rangle}\dd s \nonumber\\&+ \int_{0}^{t} \cro{\cH_{ N, s}}{ \mathscr{L}_{ s}^{ (2)}g} \dd s+ \cM_{ N, t}^{( \eta)}f_{ t}\label{aux:etaN_semi2}
\end{align}
Subtracting \eqref{aux:etaN_semi2} to \eqref{eq:semimart_HN} gives, for $t\in[0, T]$,
\begin{equation}
\label{eq:EN_semimart}
\cE_{ N, t}g = \cE_{ N, 0}g + \int_{0}^{t} \cE_{ N, s}(\mathscr{L}_{ s}^{ (1)}g) \dd s + \int_{0}^{t} F_{ N, s}g \dd s + \int_{0}^{t} G_{ N, s}g \dd s + \cM_{ N, t}^{ (\cH)}g - \cM_{ N, t}^{ (\eta)}f_{ t}.
\end{equation}
We already know from Proposition~\ref{prop:UN_VN} that $ \int_{0}^{t} F_{ N, s}g \dd s + \int_{0}^{t} G_{ N, s}g \dd s$ converges to $0$ as $N\to\infty$. We prove that it is also the case for the initial condition and the martingale part in \eqref{eq:EN_semimart}, when $ \alpha< \frac{ 1}{ 2}$. First consider $\cE_{ N, 0}g= \left\langle \cH_{ N, 0}\, ,\, g\right\rangle - \left\langle \eta_{ N, 0}\, ,\, \left\langle \nu_{ 0}\, ,\, \Psi g\right\rangle\right\rangle$. Using Proposition~\ref{prop:initial_cond_HN}, we know that $\cE_{ N, 0}g$ converges as $N\to\infty$ to $\left\langle \cH_{0}\, ,\, g\right\rangle - \left\langle \eta_{0}\, ,\, \left\langle \nu_{ 0}\, ,\, \Psi g\right\rangle\right\rangle$. Using the form of the covariance of $( \eta_{ 0}, \cH_{ 0})$ in \eqref{eq:covariances_0}, it is immediate to see that $\lim_{ N\to\infty}\cE_{N, 0}\equiv 0$. 
 
We prove now a similar result for the martingale part in \eqref{eq:EN_semimart}: we deduce from Proposition~\ref{prop:martingale_part_H} that the martingale part in \eqref{eq:EN_semimart} converges as $N\to \infty$ to $\cM_{t}^{ (\cH)}g - \cM_{t}^{ (\eta)}f_{ t}$, where $(\cM^{ (\eta)}, \cM^{ (\cH)})$ is given in Definition~\ref{def:covariance_WH}. It is then easy from the form of the covariance $\cK_{ s, t}$ in \eqref{eq:def_covariances_mart} and the definition of $f_{ t}$ in \eqref{eq:ft} to see that this process is equally $0$. 
 
Following now the same procedure as in Section~\ref{sec:identification_limit}, we deduce from \eqref{eq:EN_semimart} that any accumulation point $\cE$ of $\cE_{ N}$ in $\cC([0, T], \bW_{ -2(P+2)}^{ \underline{ \kappa} + \gamma, \underline{ \iota}+ \gamma})$ necessarily solves \[\cE_{ t} = \int_{0}^{t} (\mathscr{L}_{ s}^{ (1)})^{ \ast}\cE_{ s} \dd s,\quad t\in[0, T].\] 
Then Proposition~\ref{prop:repres_H_eta} follows directly from Proposition~\ref{prop:unique_L1}.
\end{proof}
\subsubsection{Convergence of $(\eta_{ N})_{ N\geq1}$}
\label{sec:conv_subcrit_etaN}
Theorem~\ref{theo:conv_subcrit} is an easy consequence of the following uniqueness result.
\begin{proposition}[Uniqueness in law of the limit]
\label{prop:uniqueness_subcrit}
Under the assumptions of Section~\ref{sec:assumptions}, there is uniqueness in law (as well as pathwise uniqueness) in the limiting SPDE \eqref{eq:SPDE_subcrit_eta}. 
\end{proposition}
\begin{proof}[Proof of Proposition~\ref{prop:uniqueness_subcrit}]
We follow here the strategy used by Mitoma (see \cite{MR820620}, Step~3, p.~352) who proved a similar uniqueness result in a slightly different context of weighted tempered distributions. Although the functional context is different, the strategy is similar. Let $ t\mapsto \eta(t) \in \cC([0, T], \bV_{ -2(P+2)}^{ \underline\kappa + \gamma, \underline\iota + \gamma})$ be a solution in $\bV_{ -3(P+2)}^{ \underline\kappa, \underline\iota}$ to \eqref{eq:SPDE_subcrit_eta}. For simplicity, we write $\cM$ in \eqref{eq:SPDE_subcrit_eta} instead of $\cM^{ (\eta)}$ in the following. Setting 
\begin{equation}
\label{eq:def_h}
h(t):= \int_{0}^{t} \cL_{ s}^{ \ast} \eta_{ s} \dd s= \eta_{ t} - \eta_{ 0} - \cM_{ t}
\end{equation} and differentiating this quantity with respect to $t$, one obtains that (almost surely w.r.t. the randomness) 
\begin{equation}
\frac{ \dd}{ \dd t} h(t) = \cL^{ \ast}_{ t} h(t) + \eta_{ 0} + \cM_{ t}
\end{equation}
Define for all $0\leq t\leq T$, all test function $f$ and $(\theta, \omega, x)$
\begin{equation}
\label{eq:Knu}
K[ \nu_{ t}]f(\theta, \omega, x):= \cro{\nu_{t}}{\nabla_{\theta}f(\cdot) \cdot \Gamma(\cdot, \cdot, \theta, \omega)\Psi(\cdot, x)},
\end{equation}
so that the linear operator $\cL_{ t}$ in \eqref{eq:scL}
\begin{equation}
\label{eq:decomp_scL}
\cL_{ t}= L[ \nu_{ t}] + K[ \nu_{ t}].
\end{equation}
Focus for a moment on the regularity of the linear operator $K$ in \eqref{eq:Knu}: for all $l\geq1$, there exists a constant $C>0$ such that for all $f\in\cC_{ l}^{ 0, \iota}$, for all $u\in(0, T]$
\begin{equation}
\label{eq:bound_Knu}
\N{K[ \nu_{ u}]f}_{ \cC_{ l}^{ 0, 0}}\leq C\N{f}_{ \cC_{ l}^{ 0, \iota}} \left(\frac{ 1}{ u^{ \alpha_{ 0}}} + \frac{ 1}{ u^{ \alpha_{ 0}+ \frac{ 1}{ 2}}}\right),
\end{equation}
where $ \alpha_{ 0}\in[0, \frac{ 1}{ 2})$ is defined in Proposition~\ref{prop:decomp_nu}. Indeed, using the decomposition $ \nu_{ t}(\dd \theta, \dd \omega, \dd x)= p_{ t}(\theta, \omega) \dd \theta\mu(\dd \omega)\dd x$ (recall Proposition~\ref{prop:decomp_nu}), one obtains by integration by parts
\begin{align*}
K[ \nu_{ u}]f(\theta, \omega, x)&= - \int_{ \bbX\times\bbY\times\bbS} \div_{ \bar\theta} \left(p_{ t}( \bar\theta, \bar \omega) \Gamma(\bar \theta, \bar \omega, \theta, \omega)\right) f(\bar \theta, \bar \omega, \bar x) \Psi(x, \bar x) \dd \bar\theta \mu(\dd \bar\omega)\dd \bar x,\\
&= - \int_{ \bbX\times\bbY\times\bbS} \div_{ \bar\theta} \left(p_{ t}( \bar\theta, \bar \omega) \Gamma(\bar \theta, \bar \omega, \theta, \omega)\right) f(\bar \theta, \bar \omega, x-z) \rho(z) \dd \bar\theta \mu(\dd \bar\omega)\dd z,
\end{align*}where $ \rho$ is defined in \eqref{eq:def_rho}. Hence, for all $l\geq 1$,
\begin{align*}
\N{K[ \nu_{ u}]f}_{ \cC_{ l}^{ 0, 0}}&=\\ \sum_{ \left\vert k \right\vert\leq l} \sup_{ \theta, \omega, x} &\left\vert \int_{ \bbX\times\bbY\times\bbS} \div_{ \bar \theta} \left(p_{ t}(\bar \theta, \bar \omega)D_{ \theta}^{ k_{ 1}}D_{ \omega}^{ k_{ 2}} \Gamma( \bar \theta, \bar \omega, \theta, \omega)\right)D_{ x}^{ k_{ 3}}f(\bar \theta, \bar \omega, x-z) \rho(z) \dd \bar\theta \mu(\dd \bar\omega)\dd z\right\vert,\\
\leq C\N{f}_{ \cC_{ l}^{ 0, \iota}} \sum_{ \left\vert k \right\vert\leq l} &\sup_{ \theta, \omega} \int_{ \bbX\times\bbY} \left\vert \div_{ \bar \theta} \left(p_{ t}(\bar \theta, \bar \omega)D_{ \theta}^{ k_{ 1}}D_{ \omega}^{ k_{ 2}} \Gamma( \bar \theta, \bar \omega, \theta, \omega)\right)\right\vert (1+ \left\vert \bar \omega \right\vert^{ \iota}) \dd \bar\theta \mu(\dd \bar\omega),
\end{align*}
where the sum above is taken over all multi-indices $k=(k_{ 1}, k_{ 2}, k_{ 3})$ such that $ \left\vert k \right\vert= \left\vert k_{ 1} \right\vert + \left\vert k_{ 2} \right\vert + \left\vert k_{ 3} \right\vert\leq l$. Consequently, using \eqref{eq:est_pt} and \eqref{eq:est_div_pt},
\begin{align*}
\N{K[ \nu_{ u}]f}_{ \cC_{ l}^{ 0, 0}}\leq&\\ C\N{f}_{ \cC_{ l}^{ 0, \iota}} &\left(\frac{ 1}{ u^{ \alpha_{ 0}}} + \frac{ 1}{ u^{ \alpha_{ 0}+ \frac{ 1}{ 2}}}\right) \int_{ \bbY}(1+ \left\vert \bar \omega \right\vert^{ \iota})^{ 2} \mu(\dd \bar\omega) \sum_{ \left\vert k \right\vert\leq l+1} \sup_{ \theta, \omega, \bar \omega} \left(\int_{ \bbX} \left\vert D_{ \theta, \omega}^{ k} \Gamma( \bar \theta, \bar \omega, \theta, \omega)\right\vert  \dd \bar\theta\right),\\
&\leq C\N{f}_{ \cC_{ l}^{ 0, \iota}} \left(\frac{ 1}{ u^{ \alpha_{ 0}}} + \frac{ 1}{ u^{ \alpha_{ 0}+ \frac{ 1}{ 2}}}\right),
\end{align*}
by the assumptions made on $ \mu$ and $ \Gamma$ (recall Section~\ref{sec:assumptions}). Using the semigroup $V(s, t)$ associated to $L[ \nu_{ t}]$, one has
\begin{align*}
\frac{ \dd}{ \dd u} V^{ \ast}(t, u) h(u)&= - V^{ \ast}(t, u) L[ \nu_{ u}]^{ \ast} h(u) + V^{ \ast}(t,u)(\cL^{ \ast}_{ u} h(u) + \eta_{ 0} + \cM_{ u}),\\
&= V^{ \ast}(t, u)(K^{ \ast}[ \nu_{ u}] h(u) + \eta_{ 0} + \cM_{ u}).
\end{align*}
Hence, $h$ is also solution to
\begin{equation}
\label{eq:eq_ht}
h(t)=  \int_{0}^{t}V^{ \ast}(t, u)(K^{ \ast}[ \nu_{ u}] h(u) + \eta_{ 0} + \cM_{ u}) \dd u.
\end{equation}
The main point of the proof is to see that $h$ solution of \eqref{eq:eq_ht} can be approximated by the converging sequence $(h_{ n})_{ n\geq1}$ defined recursively as follows
\begin{equation}
\label{eq:recurs_hn}
\begin{cases}
h_{ 1}(t)&= \int_{0}^{t} V^{ \ast}(t, u) ( \eta_{ 0} + \cM_{ u})\dd u,\\
h_{ n}(t)&= \int_{0}^{ t}V^{ \ast}(t, u)( K^{ \ast}[ \nu_{ u}] h_{ n-1}(u) + \eta_{ 0} + \cM_{ u})\dd u,\ n\geq 2.
\end{cases} 
\end{equation}
Indeed, by the boundedness of the semigroup $V(t, u)$ \eqref{eq:bound_Vts_Cj} and by the boundedness of $K[ \nu_{ u}]$ \eqref{eq:bound_Knu}, we obtain that for all $0<u<t<T$, for all $f\in \cC_{ 2(P+2)}^{ 0, 0}$, for all $h\in \cC_{ -2(P+2)}^{ 0, 0}$
\begin{align*}
\left\vert h K[ \nu_{ u}] V(t, u) f \right\vert &\leq \N{h}_{  \cC_{ -2(P+2)}^{ 0, 0}} \N{K[ \nu_{ u}] V(t, u) f}_{  \cC_{ 2(P+2)}^{ 0, 0}},\\
&\leq C\N{h}_{  \cC_{ -2(P+2)}^{ 0, 0}} \left(\frac{ 1}{ u^{ \alpha_{ 0}}} + \frac{ 1}{ u^{ \alpha_{ 0} + \frac{ 1}{ 2}}}\right) \N{V(t, u) f}_{  \cC_{ 2(P+2)}^{ 0, \iota}},\\
&\leq C\N{h}_{  \cC_{ -2(P+2)}^{ 0, 0}} \left(\frac{ 1}{ u^{ \alpha_{ 0}}} + \frac{ 1}{ u^{ \alpha_{ 0} + \frac{ 1}{ 2}}}\right) \N{f}_{  \cC_{ 2(P+2)}^{ 0, 0}}.
\end{align*}
Thus, the sequence $(h_{ n})_{ n\geq1}$ defined in \eqref{eq:recurs_hn} satisfies, for all $n\geq 2$
\begin{align*}
\N{h_{ n+1}(t) - h_{ n}(t)}_{ \cC_{ -2(P+2)}^{ 0, 0}}&\leq C \int_{0}^{t} \left( \frac{ 1}{ u^{ \alpha_{ 0}}} + \frac{ 1}{ u^{ \alpha_{ 0} + \frac{ 1}{ 2}}}\right) \N{h_{ n}(u)- h_{ n-1}(u)}_{ \cC_{ -2(P+2)}^{ 0, 0}}\dd u.
\end{align*}
Choose now $r>1$ such that $1< r < \frac{ 1}{ \alpha_{ 0} + \frac{ 1}{ 2}}$ and $r^{ \ast}$ such that $ \frac{ 1}{ r} + \frac{ 1}{ r^{ \ast}}=1$. Then, by H¨\"older inequality, one obtains that $\N{h_{ n+1}(t) - h_{ n}(t)}^{ r^{ \ast}}_{ \cC_{ -2(P+2)}^{ 0, 0}}\leq C \int_{0}^{t} \N{h_{ n}(u)- h_{ n-1}(u)}^{ r^{ \ast}}_{ \cC_{ -2(P+2)}^{ 0, 0}}\dd u$. By an immediate recursion, for all $k\geq 1$,  $\N{h_{ n+1+k}(t) - h_{ n+k}(t)}^{ r^{ \ast}}_{ \cC_{ -2(P+2)}^{ 0, 0}}\leq C^{ k} \frac{ T^{ k}}{ k!}$, so that $(h_{ n})_{ n\geq1}$ is a Cauchy sequence in $\cC([0, T], \cC_{ -2(P+2)}^{ 0, 0})$ and thus converges to $h$, solution of \eqref{eq:eq_ht}. Turning back to $ \eta$ (recall \eqref{eq:def_h}) and writing $R(t):= \eta_{ 0} + \cM_{ t}$, we obtain that $ \eta$ is uniquely written as
\begin{equation}
\label{eq:repr_eta}
\begin{split}
\eta_{ t} = \lim_{ n\to \infty} \Bigg\lbrace \eta_{ 0} + \cM_{ t} &+ \int_{0}^{t} V^{ \ast}(t, t_{ 1})\cL^{ \ast}_{ t_{ 1}} R(t_{ 1}) \dd t_{ 1}\\ &+ \int_{0}^{t} \int_{0}^{t_{ 1}} V^{ \ast}(t, t_{ 1}) K[ \nu_{ t_{ 1}}]^{ \ast} V^{ \ast}(t, t_{ 2}) \cL^{ \ast}_{ t_{ 2}} R(t_{ 2}) \dd t_{ 2} \dd t_{ 1} + \ldots \\
&+\int_{0}^{t} \int_{0}^{t_{ 1}} \cdots \int_{0}^{t_{ n-1}} V^{ \ast}(t, t_{ 1}) K[ \nu_{ t_{ 1}}]^{ \ast}\cdots V^{ \ast}(t, t_{ n}) \cL^{ \ast}_{ t_{ n}} R(t_{ n})\dd t_{ n}\ldots\dd t_{ 2}\dd t_{ 1}\Bigg\rbrace.
\end{split}
\end{equation}
This proves pathwise uniqueness. But if one chooses another solution $ \tilde \eta$ defined on another probability space, with initial condition $\tilde \eta_{ 0}$ and noise $\tilde\cM$ with the same law as $(\eta_{ 0}, \cM)$, we obtain the same expression as above with $R$ replaced by $\tilde R(t) = \tilde\eta_{ 0} + \tilde \cM_{ t}$. Since $R$ and $\tilde R$ have then the same law, uniqueness in law in \eqref{eq:SPDE_subcrit_eta} follows from \eqref{eq:repr_eta}. Proposition~\ref{prop:uniqueness_subcrit} is proven.
\end{proof}
\subsection{Convergence when $ \alpha>\frac{ 1}{ 2}$}
\label{sec:ident_supercrit}
Theorem~\ref{theo:conv_supercrit} will be proven once we have established uniqueness of a solution to \eqref{eq:SPDE_supercrit}. The main issue here is that the linear operator $\mathscr{L}_{ s}$ in \eqref{eq:def_opL} is not the generator of a diffusion, due to the nonstandard integro-differential term $\mathscr{L}^{ (2)}_{ s}$ in \eqref{eq:def_opL}. In particular, the existence of a semi-group solving a backward Kolmogorov equation similar to \eqref{eq:back_Kolm_U_2} is unclear in the case of $\mathscr{L}_{ s}$. Thus, we restrict here to the case where the state space $\bbX$ is compact and $c$ bounded (see Section~\ref{sec:intro_supercrit}).
\begin{proposition}
\label{prop:uniqueness_supercrit}
Under the assumptions of Sections~\ref{sec:assumptions} and~\ref{sec:intro_supercrit}, there exists at most one solution $( \eta, \cH)$ in $\cC([0, T], \bV_{ -2(P+2)}^{ \underline{\kappa} + \gamma, \underline{\iota} + \gamma}\oplus \bW_{ -2(P+2)}^{ \underline{\kappa} + \gamma, \underline{\iota} + \gamma})$ to \eqref{eq:SPDE_supercrit}.
\end{proposition}
\begin{proof}[Proof of Proposition~\ref{prop:uniqueness_supercrit}]
The supplementary assumptions made in the beginning at the paragraph ensure that the weighted norms introduced in Section~\ref{sec:Hilbert_spaces} are now equivalent to the usual Sobolev norms without weights. The purpose here is to use the regularizing properties of the heat kernel in those spaces.

Let $(\eta^{ (1)}, \cH^{ (1)})$ and $(\eta^{ (2)}, \cH^{ (2)})$ be two solutions of \eqref{eq:SPDE_supercrit} in $\cC([0, T], \bV_{ -2(P+2)}^{ \underline{\kappa} + \gamma, \underline{\iota} + \gamma}\oplus \bW_{ -2(P+2)}^{ \underline{\kappa} + \gamma, \underline{\iota} + \gamma})$. Then the difference $\tcH= \cH^{ (1)} - \cH^{ (2)}$ is in $\cC([0, T], \bW_{ -2(P+2)}^{ \underline{\kappa} + \gamma, \underline{\iota} + \gamma})$ and solves in $\bW_{ -3(P+2)}^{ \underline{\kappa}, \underline{\iota}}$
 \begin{equation}
 \label{eq:SPDE_diff_cH}
 \tcH_{ t}=  \int_{0}^{t} (\mathscr{L}_{ s})^{ \ast} \tcH_{ s} \dd s.
 \end{equation}
 In particular, it is a solution to the following equation (with unknown $\cE\in\cC([0, T], \bW_{ -2(P+2)}^{ \underline{\kappa} + \gamma, \underline{\iota} + \gamma})$)
 \begin{equation}
 \label{eq:SPDE_diff_cH2}
 \cE_{ t}=  \int_{0}^{t} \left( \frac{ 1}{ 2} \Delta_{ \theta, \ttheta}\right)^{ \ast} \cE_{ s} \dd s + \int_{0}^{t} R_{ s}\dd s,
 \end{equation}
 where
 \begin{equation}
 \label{eq:def_Rs}
 R_{ s}:= \left(\mathscr{L}_{ s} - \frac{ 1}{ 2} \Delta_{ \theta, \ttheta}\right)^{ \ast} \tcH_{ s},\ s\in[0, T],
 \end{equation}
Denote by $ (\theta, \bar \theta) \mapsto g_{ t}( \theta, \bar \theta)$ the heat kernel on $\bbX^{ 2}$ and by
\begin{equation}
G(s, t) \varphi(\tau, \ttau):= g_{ t-s} \ast \varphi(\tau, \ttau),
\end{equation}
the corresponding semi-group (note that the convolution only concerns the variables $(\theta, \ttheta)$). It is standard to prove a backward Kolmogorov equation similar to \eqref{eq:back_Kolm_U_2} concerning $G$:
\begin{equation}
G(t, s)g - g = \int_{s}^{t} \frac{ 1}{ 2} \Delta_{ \theta, \ttheta}(G(t, r)g)\dd r,\ \text{in $\cC_{ 3(P+2)}$}.
\end{equation}
By the same procedure as in Proposition~\ref{prop:unique_L1}, one obtains that \eqref{eq:SPDE_diff_cH2} has a unique solution (which is necessarily $\tcH$) that satisfies,
 \begin{equation}
 \label{eq:repr_tcH}
 \tcH_{ t}= \int_{0}^{t} G(t, s)^{ \ast}R_{ s}\dd s,\text{ in $\cC_{ -3(P+2)}$},
 \end{equation}
that is, for all $g$ in $\cC_{ 3(P+2)}$,
 \begin{equation}
 \label{eq:repr_tcH_g}
\left\langle \tcH_{ t}\, ,\, g\right\rangle = \int_{0}^{t} \left\langle \tcH_{ s}\, ,\, \left(\mathscr{L}_{ s} - \frac{ 1}{ 2} \Delta_{ \theta, \ttheta}\right) G(t, s) g\right\rangle\dd s.
 \end{equation}
Note also that the coefficients in $\mathscr{L}_{ s}$ do not depend on the space variables $(x, \tx)$ (by rotational invariance), and that $ \Phi[f]$ in \eqref{eq:Phif} only depends on $x$, not on $\tx$. Hence, with no loss of generality, we can consider test functions $g$ depending only on $(\theta, \omega, x, \ttheta, \tomega)$. Using the regularizing properties of the heat kernel, we deduce from \eqref{eq:repr_tcH} that, for some constant $C>0$ and some $ \beta\in(0, \frac{ 1}{ 2})$,
\begin{equation}
\label{eq:NtcH}
\N{\tcH_{ t}}_{ -3(P+2)} \leq C \int_{0}^{t} \frac{ 1}{ s^{ \beta} \sqrt{t-s}} \N{\tcH_{ s}}_{ -3(P+2)} \dd s.
\end{equation}
Let us admit for a moment \eqref{eq:NtcH}. Then, following the same procedure as in \cite{daiPra96}, Step~3, p.~763, we obtain directly the uniqueness for $\cH$. Consequently, $\tilde \eta:= \eta^{ (2)} - \eta^{ (1)}$ solves
\begin{equation}
\label{eq:tceta}
\tilde\eta_{ t} = \int_{0}^{t} L[ \nu_{ s}]^{ \ast} \tilde\eta_{ s}\dd s,\ t\in[0, T].
\end{equation}
Applying Proposition~\ref{prop:unique_Lnu}, we obtain that $\tilde \eta\equiv 0$ and Proposition~\ref{prop:uniqueness_supercrit} follows. Hence, it remains to prove \eqref{eq:NtcH}. One has for any differential operators $D_{1}$ and $D_{ 2}$ of order $q_{ 1}$ and $q_{ 2}$ such that $3(P+2)$,
\begin{align*}
\N{D_{1}\left(\nabla_{ \theta}G(t, s)g\right) \cdot \left(D_{ 2}c\right)}_{ 2}^{ 2}&= \int_{ (\bbX\times\bbY\times\bbZ)^{ 2}} \left\vert D_{1} \left(\nabla_{ \theta}G(t, s)g\right) \cdot \left(D_{ 2}c\right) \right\vert^{ 2} \dd \tau \dd \ttau,\\
&\leq C\int_{ (\bbX\times\bbY\times\bbZ)^{ 2}} \left\vert \nabla_{ \theta}G(t, s)D_{1}g \right\vert^{ 2} \dd \tau \dd \ttau\leq \frac{ C}{ t-s} \N{D_{ 1}g}_{ 2}^{ 2}.
\end{align*}
The term $\nabla_{ \theta}g( \tau, \ttau) \cdot\left[ \Gamma \Psi, \nu_{s}\right](\tau)$ in \eqref{eq:def_opL} can be treated in the same way, using the boundedness of $ \Gamma$ and its derivatives. For the last term in \eqref{eq:def_opL}, for the same reasons, it suffices to estimate, for any differential operator $D$ of order smaller than $3(P+2)$,
 \begin{align*}
\int \left\vert \cro{ \nu_{s}}{ \Psi(\cdot, x)\nabla_{ \ttheta}g(\cdot, \tau)} \right\vert^{ 2} \dd \theta\dd \omega \dd x &= \int\left\vert \int_{ \bbS}  \Psi(z, x) \int_{ \bbX\times\bbY}\nabla_{ \ttheta}g(\bar\theta, \bar\omega, z, \theta, \omega) p_{ s}(\bar\theta, \bar\omega)\dd \bar\theta \mu(\dd \bar\omega)\dd z \right\vert^{ 2} \dd \theta\dd \omega \dd x,\\
&=  \int\left\vert \int_{ \bbS}  \rho(x-z) \left\langle \nabla_{ \ttheta}g(\cdot, \cdot, z, \theta, \omega)\, ,\, \xi_{ s}\right\rangle\dd z \right\vert^{ 2} \dd x \dd \theta\dd \omega,\\
&=\int_{ \bbX\times\bbY} \left(\int_{ \bbS}\left\vert \rho \ast \left\langle \nabla_{ \ttheta}g(\cdot, \cdot, \cdot, \theta, \omega)\, ,\, \xi_{ s}\right\rangle \right\vert^{ 2}(x) \dd x\right) \dd \theta\dd \omega,
\end{align*}
where we recall the definition of $ \rho$ in \eqref{eq:def_rho}. By Young's inequality,
\begin{align*}
\int \left\vert \cro{ \nu_{s}}{ \Psi(\cdot, x)\nabla_{ \ttheta}g(\cdot, \tau)} \right\vert^{ 2} \dd \theta\dd \omega \dd x &\leq \N{ \rho}_{ 1}^{ 2}\int_{ \bbX\times\bbY} \left(\int_{ \bbS}\left\vert \left\langle \nabla_{ \ttheta}g(\cdot, \cdot, \cdot, \theta, \omega)\, ,\, \xi_{ s}\right\rangle \right\vert^{ 2}(x) \dd x\right) \dd \theta\dd \omega,\\
&=\N{ \rho}_{ 1}^{ 2}\int_{ \bbX\times\bbY} \left(\int_{ \bbS}\left\vert \int_{ \bbX\times\bbY} \nabla_{ \ttheta}g(\bar \theta, \bar \omega, x, \theta, \omega) p_{ s}(\bar\theta, \bar\omega) \dd \bar\theta \mu(\dd \bar\omega)\right\vert^{ 2} \dd x\right) \dd \theta\dd \omega,\\
&\leq \N{ \rho}_{ 1}^{ 2}\int_{ (\bbX\times\bbY)^{ 2}\times\bbS}\left\vert \nabla_{ \ttheta}g(\bar \theta, \bar \omega, x, \theta, \omega) \right\vert^{ 2} p_{ s}(\bar\theta, \bar\omega) \dd \bar\theta \mu(\dd \bar\omega)\dd \theta\dd \omega \dd x,\\
&\leq \frac{ \N{ \rho}_{ 1}^{ 2}}{ s^{ \alpha_{ 0}}}\int_{ (\bbX\times\bbY)^{ 2}\times\bbS}\left\vert \nabla_{ \ttheta}g(\bar \theta, \bar \omega, x, \theta, \omega) \right\vert^{ 2}  (1+ \left\vert \bar\omega \right\vert^{ \iota}) \dd \bar\theta \dd \theta  \dd \omega \mu(\dd \bar\omega) \dd x,
\end{align*}
where we used the \emph{a priori} estimate \eqref{eq:est_pt} on the density $p$. As in \eqref{eq:repr_tcH}, injecting $D_{ \tau}G(t,s)g= G(t,s) D_{ \tau}g$ into the previous estimate, we obtain
\begin{align*}
\int \left\vert \cro{ \nu_{s}}{ \Psi(\cdot, x)\nabla_{ \ttheta}G(t, s)D_{ \tau}g(\cdot, \tau)} \right\vert^{ 2}& \dd \theta\dd \omega \dd x\\ &\leq   \frac{ \N{ \rho}_{ 1}^{ 2}}{ s^{ \alpha_{ 0}}}\int\left\vert \nabla_{ \ttheta}G(t, s)D_{ \tau}g(\bar \theta, \bar \omega, x, \theta, \omega) \right\vert^{ 2}  (1+ \left\vert \bar\omega \right\vert^{ \iota}) \dd \bar\theta \dd \theta  \dd \omega \mu(\dd\bar\omega) \dd x.
\end{align*}
Another application of Young's inequality and usual estimate on the heat kernel lead to
\begin{align*}
\int \left\vert \cro{ \nu_{s}}{ \Psi(\cdot, x)\nabla_{ \ttheta}G(t, s)D_{ \tau}g(\cdot, \tau)} \right\vert^{ 2} \dd \theta\dd \omega \dd x &\leq   \frac{C}{ s^{ \alpha_{ 0}} (t-s)} \N{D_{ \tau}g}_{ 2}^{ 2}.
\end{align*}
The inequality \eqref{eq:NtcH} follows by density.
\end{proof}
\section{Proofs}
\label{sec:proofs}

\subsection{Proof of Propositions~\ref{prop:semimart} and~\ref{prop:semimartI}}
\label{sec:proof_semimart}
We prove in this paragraph the semimartingale decompositions of processes $ \eta_{ N}$ and $\cH_{ N}$ of Section~\ref{sec:semimart_repr_etaN_HN}.
\begin{proof}[Proof of Proposition~\ref{prop:semimart}]
Writing Ito's formula for \eqref{eq:odegene} gives
\begin{multline}
\label{aux:itonuN}
\cro{\nuN{t}}{f}= \cro{\nuN{0}}{f} + \int_{0}^{t} \cro{\nuN{s}}{ \frac{ 1}{ 2}\Delta_{\theta}f}\dd s \\+\int_{0}^{t} \cro{\nuN{s}}{ \nabla_{\theta}f\cdot \left\{c + \Big[\Gamma\Psi, \nuN{s}\Big]\right\}} \dd s+ \frac{1}{a_{N}} \cM_{ N, t}^{( \eta)}f
\end{multline}
Combining the McKean-Vlasov equation \eqref{eq:nut} and \eqref{aux:itonuN}, one obtains
\begin{align*}
\cro{ \eta_{ N, t}}{ f}&= \cro{ \eta_{ N, 0}}{ f} + \int_{0}^{t}  \cro{ \eta_{ N, s}}{ \frac{ 1}{ 2}\Delta_{\theta}f + \nabla_{\theta}f\cdot c}\dd s \\&+a_{ N}\int_{0}^{t} \left(\cro{\nuN{s}}{ \nabla_{\theta}f\cdot\Big[\Gamma\Psi, \nuN{s}\Big]} - \cro{\nu_{ s}}{ \nabla_{\theta}f\cdot\Big[\Gamma\Psi, \nu_{ s}\Big]}\right) \dd s+ \cM_{ N, t}^{( \eta)}f,\\
&= \cro{ \eta_{ N, 0}}{ f} + \int_{0}^{t}  \cro{ \eta_{ N, s}}{ \frac{ 1}{ 2}\Delta_{\theta}f + \nabla_{\theta}f\cdot c}\dd s + \cM_{ N, t}^{( \eta)}f \\&+\int_{0}^{t} \left(a_{ N} \cro{ \nu_{ N, s}- \nu_{ s}}{ \nabla_{\theta}f\cdot\Big[\Gamma\Psi, \nu_{ s}\Big]}+ a_{ N}\cro{\nuN{s}}{ \nabla_{\theta}f\cdot\Big[\Gamma\Psi, \nuN{s}- \nu_{ s}\Big]}\right) \dd s,\\
&= \cro{ \eta_{ N, 0}}{ f} + \int_{0}^{t}  \cro{ \eta_{ N, s}}{ \frac{ 1}{ 2}\Delta_{\theta}f + \nabla_{\theta}f\cdot \left\{c + \Big[\Gamma\Psi, \nu_{ s}\Big]\right\}}\dd s 
\\&+\underbrace{\int_{0}^{t} a_{ N}\cro{\nuN{s}}{ \nabla_{\theta}f\cdot\Big[\Gamma\Psi, \nuN{s}- \nu_{ s}\Big]}  \dd s}_{:=u_{ N, t}}+ \cM_{ N, t}^{( \eta)}f.
\end{align*}
The remaining term $u_{ N}$ can be computed as:
\begin{align*}
  \begin{split}
  u_{ N, t} &= \int_{0}^{t} a_{ N} \Bigg(\frac{1}{ \left\vert \Lambda_{ N} \right\vert^{ 2}}\sum_{ i, j\in \Lambda_{ N}}\nabla_{ \theta} f( \tau_{ i, s}) \cdot \Gamma( \theta_{ i, s}, \omega_{ i}, \theta_{ j, s}, \omega_{ j})\Psi(x_{ i}, x_{ j})\\& \qquad\qquad\qquad- \frac{1}{ \left\vert \Lambda_{ N} \right\vert}\sum_{ i\in \Lambda_{ N}} \int \nabla_{ \theta} f( \tau_{ i, s}) \cdot\Gamma(\theta_{ i, s}, \omega_{ i}, \ttheta, \tomega) \Psi(x_{ i}, \tx)\nu_{ s}(\dd \ttau)\Bigg)\dd s
  \end{split}\\
&= \int_{0}^{t}\cro{ \cH_{ N, s}}{ \Phi[f]} \dd s.
\end{align*}
Proposition~\ref{prop:semimart} follows.
\end{proof}

\begin{proof}[Proof of Proposition~\ref{prop:semimartI}]
Applying Ito's Formula to \eqref{eq:odegene_short} (for any regular two-variable function $(\tau, \ttau)\mapsto g( \tau, \ttau)= g( \theta, \omega, x, \ttheta, \tomega, \tx)$), one obtains, for all $t\in[0, T]$
\begin{equation}
\label{eq:ito2tau}
\begin{split}
g( \tau_{i,t}, \tau_{j,t})&= g( \tau_{i, 0}, \tau_{j, 0}) + \frac{1}{2} \int_{0}^{t} \Delta_{ \theta, \ttheta} g( \tau_{i,s}, \tau_{j,s })\dd s\\ &+ \int_{0}^{t} \nabla_{ \theta} g( \tau_{i, s}, \tau_{j, s}) \cdot\left\{ c( \theta_{i, s}, \omega_{i}) + [ \Gamma \Psi, \nu_{ N, s}](\tau_{ i, s})\right\}\dd s\\
&+ \int_{0}^{t} \nabla_{ \ttheta} g( \tau_{i,s}, \tau_{j,s})\cdot \left\{c( \theta_{j, s}, \omega_{j}) + [ \Gamma \Psi, \nu_{ N, s}](\tau_{ j, s})\right\} \dd s\\
&+ \int_{0}^{t} \nabla_{ \theta} g(\tau_{i,s}, \tau_{j, s}) \cdot\dd B_{i, s} + \int_{0}^{t} \nabla_{ \ttheta} g(\tau_{i, s}, \tau_{j, s}) \cdot\dd B_{j,s}
\end{split}
\end{equation}

Applying Ito's Formula to the process $t\mapsto \int \Psi(x_{ i}, \tx) g( \tau_{ i, t}, \ttau) \nu_{ t}( \dd \ttau)$ (recall \eqref{eq:nut}):
\begin{equation}
\label{eq:itoint}
\begin{split}
\int& \Psi(x_{ i}, \tx)g( \tau_{ i, t}, \ttau) \nu_{ t}(\dd\ttau)=\int \Psi(x_{ i}, \tx)g( \tau_{ i, 0}, \ttau) \nu_{ 0}(\dd\ttau)\\&+ \frac{ 1}{ 2}\int_{0}^{t} \int \Psi(x_{ i}, \tx) \Delta_{ \theta, \ttheta} g( \tau_{ i, s}, \ttau) \nu_{ s}(\dd\ttau) \dd s\\ +& \int_{0}^{t} \int \Psi(x_{ i}, \tx)\nabla_{ \theta}g( \tau_{ i, s}, \ttau) \nu_{ s}(\dd\ttau) \cdot\Big\{c( \theta_{ i, s}, \omega_{ i}) + \left[ \Gamma \Psi, \nu_{ N, s}\right]( \tau_{ i, s})\Big\}\dd s\\ 
+& \int_{0}^{t} \int \Psi(x_{ i}, \tx)\nabla_{  \theta}g( \tau_{ i, s}, \ttau) \nu_{ s}(\dd\ttau) \cdot\dd B_{ i, s}\\
+& \int_{0}^{t} \int \Psi(x_{ i}, \tx)\nabla_{ \ttheta}g( \tau_{ i, s}, \ttau) \cdot\left\{c( \ttheta, \tomega) + \left[ \Gamma \Psi, \nu_{s}\right]( \ttau)\right\}\nu_{ s}(\dd\ttau) \dd s
\end{split}
\end{equation}
Combining \eqref{eq:ito2tau} and \eqref{eq:itoint}, we have
\begin{equation}
\label{aux:HNtg}
\begin{split}
\cro{ \cH_{ N, t}}{ g}= \cro{ \cH_{ N, 0}}{ g} + \int_{0}^{t} \cro{ \cH_{ N, s}}{ \frac{ 1}{ 2} \Delta_{ \theta, \ttheta} g( \tau, \ttau) + \nabla_{ \theta}g( \tau, \ttau)\cdot c(\theta, \omega)  + \nabla_{ \ttheta}g( \tau, \ttau)\cdot c(\ttheta, \tomega)} \dd s \\ + \int_{0}^{t}\cro{ \cH_{ N, s}}{ \nabla_{ \theta}g( \tau, \ttau)\cdot \left[ \Gamma \Psi, \nu_{ N, s}\right](\tau)}\dd s + \cM^{ (\cH)}_{ N, t}g\\
+ \frac{ a_{ N}}{ \left\vert \Lambda_{ N} \right\vert} \sum_{ i\in \Lambda_{ N}} \int_{0}^{t} \Bigg\lbrace \int \Psi(x_{ i}, \tx) \nabla_{  \ttheta} g( \tau_{ i, s}, \ttau)  \cdot\left[\Gamma \Psi, \nu_{N, s}\right](\ttau)\nu_{ N, s}(\dd\ttau)\\
- \int \Psi(x_{ i}, \tx) \nabla_{  \ttheta} g( \tau_{ i, s}, \ttau) \cdot \left[ \Gamma \Psi, \nu_{s}\right](\ttau)\nu_{s}(\dd\ttau) \Bigg\rbrace.
\end{split}
\end{equation}
We treat the last term of \eqref{aux:HNtg} apart: it can be written as the sum of $A_{ N}$ and $B_{ N}$ where
\begin{align*}
A_{ N, t}&:= \frac{ a_{ N}}{ \left\vert \Lambda_{ N} \right\vert^{ 2}}\sum_{ i, j\in \Lambda_{ N}} \int_{0}^{t} \Psi(x_{ i}, x_{ j}) \nabla_{ \ttheta} g( \tau_{ i,s}, \tau_{ j,s}) \cdot\Big\{ \left[ \Gamma \Psi, \nu_{N,s}\right]( \tau_{j,s}) - \left[ \Gamma \Psi, \nu_{s}\right]( \tau_{j, s})\Big\},\\
&= \int_{0}^{t} \frac{ a_{ N}}{ \left\vert \Lambda_{ N} \right\vert} \sum_{ j\in \Lambda_{ N}} \Big\{ \frac{ 1}{ \left\vert \Lambda_{ N} \right\vert} \sum_{ i\in \Lambda_{ N}} \nabla_{ \ttheta} g( \tau_{ i,s}, \tau_{ j, s})\Psi(x_{ i}, x_{ j})\Big\}\cdot\Big\{ \frac{ 1}{ \left\vert \Lambda_{ N} \right\vert} \sum_{ l\in \Lambda_{ N}} \Gamma( \theta_{ j,s}, \omega_{ j}, \theta_{ l, s}, \omega_{ l}) \Psi(x_{ j}, x_{ l})\\
&\qquad - \int \Gamma( \theta_{ j,s}, \omega_{ j}, \ttheta, \tomega) \Psi(x_{ j}, \tx) \nu_{ s}(\dd \ttau)\Big\}\dd s,\\
&= \int_{0}^{t} \cro{ \cH_{ N,s}}{ \cro{ \nu_{ N, s}}{ \nabla_{ \ttheta}g(\cdot, \tau)\Psi(\cdot, x)} \cdot\Gamma( \theta, \omega, \ttheta, \tomega)}\dd s,
\end{align*}
and $B_{ N}$ is given by
\begin{align*}
B_{ N, t}&= \frac{ a_{ N}}{ \left\vert \Lambda_{ N} \right\vert} \sum_{ i\in \Lambda_{ N}} \int_{0}^{t} \Bigg( \frac{ 1}{ \left\vert \Lambda_{ N} \right\vert}\sum_{ j\in \Lambda_{ N}} \Psi(x_{ i}, x_{ j}) \nabla_{ \ttheta}g(\tau_{ i,s}, \tau_{ j,s}) \cdot\left[ \Gamma \Psi, \nu_{ s}\right](\tau_{ j,s})\\ 
&- \int  \Psi(x_{ i}, \tx) \nabla_{ \ttheta} g(\tau_{ i,s}, \ttau) \cdot\left[ \Gamma \Psi, \nu_{ s}\right](\ttau)\nu_{ s}(\dd \ttau)\Bigg)\dd s,\\
&= \int_{0}^{t} \cro{ \cH_{ N,s}}{ \nabla_{ \ttheta}g(\tau, \ttau) \cdot \left[ \Gamma \Psi, \nu_{ s}\right](\ttau)}\dd s.
\end{align*}
Rewriting \eqref{aux:HNtg} in term of $A_{ N}$, $B_{ N}$, $F_{ N}$ and $G_{ N}$ gives Proposition~ \ref{prop:semimartI}.
\end{proof}

\subsection{Proof of Proposition~\ref{prop:decomp_nu}}
\label{app:decomp_nu}
A solution to \eqref{eq:nut} is provided by the nonlinear measure $(\lambda_{ t})_{ t\in[0, T]}$ introduced in Proposition~\ref{prop:nonlin}. Indeed, applying Ito's formula to any function $(\theta, \omega, x)\mapsto f(\theta, \omega, x)$ that is $\cC^{2}$ \wrt $\theta$, one obtains:
\begin{multline}
\label{eq:Itofixedpoint}
f(\bar\theta^{ x, \omega}_{t}, \omega, x)= f(\theta_{0}, \omega, x) + \frac{1}{2} \int_{0}^{t} \Delta_{\theta}f(\bar\theta^{ x, \omega}_{s}, \omega, x)\dd s +\int_{0}^{t} \nabla_{\theta}f \cdot c(\bar\theta^{ x, \omega}_{s}, \omega)\dd s\\ + \int_{0}^{t}\nabla_{\theta}f \cdot  \Big[\Gamma\Psi, \lambda_{t}\Big](\theta^{ x, \omega}_{s}, \omega, x) \dd s + \int_{0}^{t} \nabla_{\theta} f(\bar\theta^{ x, \omega}_{s}, \omega, x) \cdot\dd B_{s}.
\end{multline}
Taking the expectation in \eqref{eq:Itofixedpoint} gives that $t\mapsto \lambda_{t}$ is a weak solution to \eqref{eq:nut}. The uniqueness in \eqref{eq:nut} is a consequence of \cite{LucSta2014}, Proposition~2.19. 

We now turn to the proof of the decomposition \eqref{eq:decomp_nu_xi}: assume for a moment that a decomposition like \eqref{eq:decomp_nu_xi} exists. Necessarily $ \xi$ must satisfy 
\begin{equation}
\label{eq:xi0}
\xi_{ 0}(\dd \theta, \dd\omega)= \zeta(\dd \theta) \mu(\dd \omega)
\end{equation} and for all $t\geq0$ and regular test function $( \theta, \omega) \mapsto h( \theta, \omega)$ independent of $x$
\begin{equation}
\label{eq:xit}
\partial_{t} \cro{\xi_{t}}{h} = \cro{\xi_{t}}{\frac{1}{2}\Delta_{\theta}h + \nabla_{\theta}h \cdot c}+ \cro{\nu_{t}}{\nabla_{\theta}h \cdot \Big[\Gamma\Psi, \nu_{t}\Big]}.
\end{equation}
Since by assumption, $\nu_{ t}(\dd\theta, \dd\omega, \dd x)= \xi_{ t}(\dd \theta, \omega)\dd x$, 
\begin{align}
\Big[\Gamma\Psi, \nu_{t}\Big]( \theta, \omega, x)&= \int_{ \bbX\times\bbY\times\bbS} \Gamma(\theta, \omega, \ttheta, \tomega) \Psi(x, \tx) \xi_{ t}(\dd \ttheta, \tomega)\dd \tx,\nonumber\\
&=  \left(\int_{ \bbX\times\bbY} \Gamma(\theta, \omega, \ttheta, \tomega) \xi_{ t}(\dd \ttheta, \tomega)\right) \left(\int_{ \bbS}\Psi(x, \tx) \dd \tx\right).\label{eq:GamPsix}
\end{align}
As already noticed in the proof of Lemma~\ref{lem:fluct_Psi}, the integral $\int_{ \bbS}\Psi(x, \tx) \dd \tx = 2 \int_{0}^{1/2} \frac{ \dd u}{ \left\vert u \right\vert^{ \alpha}}= \frac{ 2^{ \alpha}}{ 1- \alpha}$, is \emph{independent of $x$}. Consequently, \eqref{eq:xit} becomes
\begin{equation}
\label{eq:xit_bis}
\partial_{t} \cro{\xi_{t}}{h} = \cro{\xi_{t}}{\frac{1}{2}\Delta_{\theta}h + \nabla_{\theta}h \cdot c}+ \cro{\xi_{t}}{\frac{ 2^{ \alpha}}{ 1- \alpha}\nabla_{\theta}h \cdot \Big[\Gamma, \xi_{t}\Big]}.
\end{equation}
Equation \eqref{eq:xit_bis} endowed with $ \xi_{ 0}(\dd \theta, \dd\omega)= \zeta(\dd \theta) \mu(\dd \omega)$ is a standard McKean-Vlasov equation with regular coefficients, and hence, admits a unique solution (see for example \cite{daiPra96} or \cite{Oelsch1984}, Lemma~10). Defining $\lambda_{ t}(\dd \theta, \dd \omega, \dd x):= \xi_{ t}(\dd \theta, \dd \omega) \dd x$ and observing that $\frac{ 2^{ \alpha}}{ 1- \alpha} \Big[\Gamma, \xi_{t}\Big]= \Big[\Gamma \Psi, \lambda_{t}\Big]$, one readily sees that $t\mapsto \lambda_{ t}$ is a solution of \eqref{eq:nut}. By uniqueness in \eqref{eq:nut}, $\lambda=\nu$ and \eqref{eq:decomp_nu_xi} follows. 

We focus now on the regularity of $ \xi_{ t}$ solution of \eqref{eq:xit_bis}. We follow here the strategy developed in \cite{MR3207725}, Proposition~7.1: fix $T>0$, $\omega\in\Supp(\mu)$ and $t\mapsto \xi_t$ the unique solution in $\cC([0, T], \cM_1(\bbX\times\bbY))$ to \eqref{eq:xit_bis}. Define $A_{ t}(\theta, \omega):= c(\theta, \omega)  + \frac{ 2^{ \alpha}}{ 1- \alpha} \left[ \Gamma \Psi, \xi_{ t}\right](\theta, \omega)$ and consider the linear equation
\begin{equation}
 \label{eq:At}
\partial_t p_t(\theta,\omega)\, =\, \frac{1}{2} \Delta p_t(\theta,\omega) -\div_\theta \Big(p_t(\theta,\omega)A_{ t}(\theta, \omega)\Big)\, ,
\end{equation}
such that for $\mu$-a.e. $\omega$, for all $f\in\cC(\bbX)$, 
\begin{equation}
\label{eq:init_edp}
\int_{ \bbX} f(\theta)p_0(\theta, \omega) \dd\theta = \int_{ \bbX} f(\theta) \zeta(\dd\theta). 
\end{equation}
For fixed $\omega\in\Supp(\mu)$, $A_{ \cdot}(\cdot, \omega)$ is continuous in $t$ and regular in $\theta$, by assumption on $ \Gamma$. Suppose for a moment that we have found a solution $p_t(\theta,\omega)$ to \eqref{eq:At}-\eqref{eq:init_edp} such that for
$\mu$-a.e. $\omega$, $p_t(\cdot, \omega)$ is strictly positive on $(0,T]\times\bbX$. In particular for such a solution $p$, the quantity
$\int_{ \bbX} p_t(\theta, \omega)\dd\theta$ is conserved for $t>0$, so that $p_t(\cdot, \omega)$ is indeed a probability density for all $t>0$. Then
both probability measures $\xi_t(\dd\theta, \dd\omega)$ and $p_t(\theta,\omega)\dd\theta\mu(\dd\omega)$ solve
\begin{align}
\int_{ \bbX\times\bbY} f(\theta, \omega)\nu_t(\dd\theta, \dd\omega)&= \int_{ \bbX\times\bbY} f(\theta, \omega)\zeta(\dd\theta) \mu(\dd\omega) +\frac12\int_0^t \int_{ \bbX\times\bbY} \Delta_{ \theta}f(\theta, \omega)\nu_s(\dd\theta,\dd\omega)\dd s\nonumber\\
&+ \int_0^t \int_{ \bbX\times\bbY} \nabla_{ \theta}f(\theta, \omega) A_{ s}(\theta, \omega)\nu_s(\dd\theta,\dd\omega)\dd s.
\label{eq:weak_At}
\end{align}
By \cite{Lucon2011} or \cite{Oelsch1984}, Lemma~10, uniqueness in \eqref{eq:xit_bis} is precisely a consequence of uniqueness in \eqref{eq:weak_At}. Hence, by uniqueness in \eqref{eq:weak_At}, $\xi_t(\dd\theta, \dd\omega)=p_t(\theta,\omega)\dd\theta\mu(\dd\omega)$, which \eqref{eq:xit_pt_mu}. So it suffices to exhibit a weak solution $p_t(\theta,\omega)$ to \eqref{eq:At} such that \eqref{eq:init_edp} is satisfied. This fact can be deduced from standard results for uniform parabolic PDEs (see \cite{MR0435594, Friedman1964} for precise definitions). In particular, a standard result (see \cite{Friedman1964}, Theorem~2, page~251 or \cite{MR0435594}, Section~7) states that \eqref{eq:At} admits a fundamental solution $G(\theta, t; \theta', s, \omega)$ ($t>s$), satisfying, for all $j\in\{0, 1\}$, for any differential operator $D_{ j, \theta}$ in $ \theta$ of order $j$
\begin{equation}
\label{eq:control_G}
\left\vert D_{j, \theta}G(\theta, t; \theta', s, \omega) \right\vert \leq \frac{c_{ 1}(\omega)}{(c_{ 2}(\omega)(t-s))^{ \frac{ m+j}{ 2}}} \exp\left( - \frac{ \left\vert \theta - \theta' \right\vert^2}{c_{ 2}(\omega)(t-s)}\right),
\end{equation}
for positive constants $c_{ 1}$ and $c_{ 2}$ that only depend on the bound and the modulus of continuity of $A_{ t}$ (see \cite{Friedman1964}, Theorem~1, page~241). In particular, by the assumptions made in Section~\ref{sec:assumptions}, one can suppose without loss of generality that for all $i=1,2$, $c_{ i}(\omega)= \alpha_{ i}(1 \vee \left\vert \omega \right\vert^{ \iota})$, for some $ \alpha_{ i}>0$. Thanks to Theorem~3, page~256 in \cite{Friedman1964}, the following expression of $p_t(\theta, \omega)$
\begin{equation}
\label{eq:pt Gamma}
p_t(\theta, \omega)=\int_{ \bbX} G(\theta, t; \theta', 0, \omega)\zeta(\theta')\dd \theta'
\end{equation}
defines a classical solution of \eqref{eq:At} on $(0,T]\times\bbS$ such that \eqref{eq:init_edp} is satisfied. The positivity and boundedness of $p_t(\cdot, \omega)$ for $t>0$ is a consequence of \cite{MR0435594}, Theorem~7, page~661. Applying \eqref{eq:control_G} with $j=0$ gives
\begin{align}
0\leq p_{ t}(\theta, \omega) &\leq \frac{c_{ 1}(\omega)}{(c_{ 2}(\omega)t)^{ \frac{ m}{ 2}}}  \int_{ \bbX} \exp\left( - \frac{ \left\vert \theta - \theta' \right\vert^2}{c_{ 2}(\omega)t}\right) \zeta(\theta')\dd \theta',\nonumber\\
&\leq \frac{c_{ 1}(\omega)}{(c_{ 2}(\omega)t)^{ \frac{ m}{ 2}}}  \left(\int_{ \bbX} \exp\left( - \frac{ q\left\vert \theta - \theta' \right\vert^2}{c_{ 2}(\omega)t}\right)\dd \theta'\right)^{ \frac{ 1}{ q}}  \left(\int_{ \bbX} \zeta(\theta)^{ p}\dd \theta\right)^{ \frac{ 1}{ p}}, \label{aux:pt_holder}\\
&\leq \frac{C c_{ 1}(\omega)}{(c_{ 2}(\omega)t)^{ \frac{ m}{ 2}}}  \left(\int_{ \bbX} \exp\left( - \frac{ q\left\vert \theta - \theta' \right\vert^2}{c_{ 2}(\omega)t}\right)\dd \theta'\right)^{ \frac{ 1}{ q}}\leq \frac{C c_{ 1}(\omega)}{(c_{ 2}(\omega)t)^{ \frac{ m}{ 2}(1- \frac{ 1}{ q})}}.\nonumber
\end{align}
where in \eqref{aux:pt_holder} $ \frac{ 1}{ p} + \frac{ 1}{ q}=1$ (recall \eqref{eq:integ_zeta}).
Similarly, using again \eqref{eq:control_G}, one has
\begin{equation}
\label{aux:div_pt}
\left\vert \div_{ \theta} p_{ t}(\theta, \omega) \right\vert\leq C \frac{ c_{ 1}(\omega)}{ c_{ 2}(\omega)^{ \frac{ 1}{ 2}+\frac{m}{ 2}(1 - \frac{ 1}{ q})}} \frac{ 1}{ t^{ \frac{ 1}{ 2}+\frac{m}{ 2}(1 - \frac{ 1}{ q})}}.
\end{equation}
The parameter $p$ in \eqref{aux:div_pt} is such that $\frac{m}{ 2}(1 - \frac{ 1}{ q})< \frac{ 1}{ 2}$ (recall \eqref{eq:integ_zeta}). By definition of the $c_{ i}$, $i=1,2$ and by \eqref{eq:moments_mu_xi0}, the estimates \eqref{eq:est_pt} and \eqref{eq:est_div_pt} follow. Proposition~\ref{prop:decomp_nu} is proven.\qed

\subsection{Proof of Proposition~\ref{prop:ttaVSnonlin}}
\label{sec:proof_tta_vs_nonlin}
Since the initial condition and the Brownian motion are the same for $\theta_{i}$ and $\btheta_{i}$, one has: for all $i\in \Lambda_{ N}$,
\begin{multline}
\label{aux:ttavsnlin}
\left\vert \theta_{i,t}-\btheta_{i, t} \right\vert^{8} \leq 32 \left(\int_{0}^{t} (\theta_{i,s}-\btheta_{i,s})\cdot(c(\theta_{i,s}, \omega_{i})- c(\btheta_{i,s}, \omega_{i}))\dd s\right)^{4}\\+ 32 \left(\int_{0}^{t} (\theta_{i,s}-\btheta_{i,s})\cdot\left(\Big[\Gamma\Psi, \nuN{s}\Big](\theta_{i,s}, \omega_{i}, x_{i})- \Big[\Gamma\Psi, \nu_{s}\Big](\btheta_{i,s}, \omega_{i}, x_{i})\right)\dd s\right)^{4}.
\end{multline}
Using the one-sided Lipschitz condition on $c$, one can bound the first term in \eqref{aux:ttavsnlin} by $32L^{4} \left(\int_{0}^{t} \left\vert \theta_{i,s}-\btheta_{i,s} \right\vert^{2}\dd s\right)^{4}\leq 32L^{4} T^{ 3}\int_{0}^{t} \left\vert \theta_{i,s}-\btheta_{i,s} \right\vert^{8}\dd s\leq32L^{4} T^{ 3}\int_{0}^{t}\max_{ i\in \Lambda_{ N}}\sup_{u\leq s}\left\vert \theta_{i,u}-\btheta_{i,u} \right\vert^{8}\dd s$.
Now, introducing the empirical measure of the nonlinear processes $(\btheta_{ 1}, \ldots, \btheta_{ N})$
\begin{equation}
\label{eq:nuNbar}
t \in[0, T]\mapsto\bnuN{t}:= \frac{1}{\left\vert \Lambda_{N} \right\vert} \sum_{j\in\Lambda_{N}} \delta_{(\btheta_{j,t}, \omega_{j}, x_{j})}= \frac{ 1}{ \left\vert \Lambda_{ N} \right\vert}\sum_{ j\in \Lambda_{ N}} \delta_{ \tau_{ j, t}},\ N\geq1
\end{equation} the last term in \eqref{aux:ttavsnlin} can estimated above by $C(A_{N}+ B_{N}+ C_{N})$ ($C$ being a numerical constant) where
\begin{align}
A_{N} &:= \left(\int_{0}^{t} \left\vert \theta_{i,s}-\btheta_{i,s} \right\vert \left\vert \Big[\Gamma\Psi, \nuN{s}\Big](\theta_{i,s}, \omega_{i}, x_{i})- \Big[\Gamma\Psi, \nuN{s}\Big](\btheta_{i,s}, \omega_{i}, x_{i}) \right\vert\dd s\right)^{4},\nonumber\\
B_{N} &:= \left(\int_{0}^{t} \left\vert \theta_{i,s}-\btheta_{i,s} \right\vert \left\vert \Big[\Gamma\Psi, \nuN{s} - \bnuN{s}\Big](\btheta_{i,s}, \omega_{i}, x_{i}) \right\vert\dd s\right)^{4},\nonumber\\
C_{N} &:= \left(\int_{0}^{t} \left\vert \theta_{i,s}-\btheta_{i,s} \right\vert \left\vert \Big[\Gamma\Psi, \bnuN{s}- \nu_{s}\Big](\btheta_{i,s}, \omega_{i}, x_{i}) \right\vert\dd s\right)^{4}.\label{aux:def_CN}
\end{align}
Concerning $A_{ N}$ and $B_{ N}$, we have, by the Lipschitz continuity of $\Gamma$, for some constant $C>0$,
\begin{align*}
A_{N}&\leq \left(\int_{0}^{t} \left\vert \theta_{i,s}-\btheta_{i,s} \right\vert\int \left\vert \Gamma(\theta_{i,s}, \omega_{i}, \theta, \omega)- \Gamma(\btheta_{i,s}, \omega_{i}, \theta, \omega) \right\vert\Psi(x_{i}, x) \nuN{s}(\dd\theta, \dd\omega, \dd x)\dd s\right)^{4},\\
&\leq  \N{\Gamma}_{Lip}^{4} \left(\frac{1}{ \left\vert \Lambda_{N} \right\vert}\sum_{j\in\Lambda_{N}} \Psi(x_{i}, x_{j})\right)^{4} \left(\int_{0}^{t} \left\vert \theta_{i,s}-\btheta_{i,s} \right\vert^{2}\dd s\right)^{4},\\
&\leq  CT^{ 3}\N{\Gamma}_{Lip}^{4} \int_{0}^{t} \max_{i\in\Lambda_{N}}\sup_{u\leq s} \left\vert \theta_{i,u}-\btheta_{i,u} \right\vert^{8}\dd s,\ \text{(by \eqref{eq:Psi_bounded})},
\end{align*}
and,
\begin{align*}
B_{N}&= \left(\frac{1}{ \left\vert \Lambda_{N} \right\vert} \sum_{j\in\Lambda_{N}}  \int_{0}^{t} \left\vert \theta_{i,s}-\btheta_{i,s} \right\vert \left\vert \Gamma(\btheta_{i,s}, \omega_{i}, \theta_{j,s}, \omega_{j}) - \Gamma(\btheta_{i,s}, \omega_{i}, \btheta_{j,s}, \omega_{j}) \right\vert\Psi(x_{i}, x_{j})\dd s\right)^{4},\\
&\leq \N{\Gamma}_{Lip}^{4} \left(\frac{1}{ \left\vert \Lambda_{N} \right\vert} \sum_{j\in\Lambda_{N}}  \Psi(x_{i}, x_{j})\int_{0}^{t} \left\vert \theta_{i,s}-\btheta_{i,s} \right\vert \left\vert \theta_{j,s} - \btheta_{j,s} \right\vert\dd s\right)^{4},\\
&\leq \N{\Gamma}_{Lip}^{4} \left(\frac{1}{ \left\vert \Lambda_{N} \right\vert} \sum_{j\in\Lambda_{N}}  \Psi(x_{i}, x_{j})\right)^{4} \left(\int_{0}^{t} \max_{i\in\Lambda_{N}} \sup_{u\leq s}\left\vert \theta_{i,u}-\btheta_{i,u} \right\vert^{2}\dd s\right)^{4},\\
&\leq  CT^{ 3}\N{\Gamma}_{Lip}^{4} \int_{0}^{t} \max_{i\in\Lambda_{N}} \sup_{u\leq s} \left\vert \theta_{i,u}-\btheta_{i,u} \right\vert^{8}\dd s, \quad \text{(recall \eqref{eq:Psi_bounded}).}
\end{align*}
It remains to study the term $C_{N}$ \eqref{aux:def_CN}. Since we obviously have
\begin{multline}
C_{N}\leq \frac{ T^{ 3}}{ 2}\int_{0}^{t} \left\vert \theta_{i,s}-\btheta_{i,s} \right\vert^{8}\dd s + \frac{ T^{ 3}}{ 2} \int_{0}^{t} \left\vert \Big[\Gamma\Psi, \bnuN{s}- \nu_{s}\Big](\btheta_{i,s}, \omega_{i}, x_{i}) \right\vert^{8}\dd s\\
\leq \frac{ T^{ 3}}{ 2}\int_{0}^{t} \max_{i\in\Lambda_{N}}\sup_{u\leq s} \left\vert \theta_{i,u}-\btheta_{i,u} \right\vert^{8}\dd s + \frac{ T^{ 3}}{ 2} \int_{0}^{t} \left\vert \Big[\Gamma\Psi, \bnuN{s}- \nu_{s}\Big](\btheta_{i,s}, \omega_{i}, x_{i}) \right\vert^{8}\dd s,
\label{aux:CN}
\end{multline}
we concentrate on the term $c_{N}:= \left\vert \Big[\Gamma\Psi, \bnuN{s}- \nu_{s}\Big](\btheta_{i,s}, \omega_{i}, x_{i}) \right\vert$ in \eqref{aux:CN}. Using the decomposition $ \nu( \dd \theta, \dd\omega, \dd x)= \xi(\dd \theta, \dd \omega) \dd x$ (recall \eqref{eq:decomp_nu_xi}), we have $c_{ N}\leq d_{ N}+ e_{ N}$ where
\begin{align}
d_{N}&:= \left\vert \frac{1}{ \left\vert \Lambda_{N} \right\vert} \sum_{j \in \Lambda_{ N}} \Psi(x_{i}, x_{j})\left(\Gamma(\btheta_{i, s}, \omega_{i}, \btheta_{j, s}, \omega_{j}) - \int \Gamma(\btheta_{i,s}, \omega_{i}, \theta, \omega)\xi_{s}(\dd\theta, \dd\omega)\right) \right\vert, \label{aux:def_dN}\\
e_{ N}&:= \left\vert \frac{1}{ \left\vert \Lambda_{N} \right\vert} \sum_{j\in\Lambda_{N}} \int \Gamma(\btheta_{i,s}, \omega_{i}, \theta, \omega)\xi_{s}(\dd\theta, \dd\omega)\Psi(x_{i}, x_{j})   - \int \Gamma(\btheta_{i, s}, \omega_{i}, \theta, \omega)\xi_{s}(\dd\theta, \dd\omega)\Psi(x_{i}, x)\dd x \right\vert.\label{aux:def_eN}
\end{align}
We treat first the term $d_{N}$ \eqref{aux:def_dN}. For simplicity, we use the notation $T_{i, j}:=\Gamma(\btheta_{i}, \omega_{i}, \btheta_{j}, \omega_{j}) - \int \Gamma(\btheta_{i}, \omega_{i}, \theta, \omega)\xi_{s}(\dd\theta, \dd\omega)$. Note that, by independence of the $ \bar \theta_{i}$ and by definition of $T_{ i, j}$, for any $J\neq k_{1}, k_{2}, k_{3}$ and $i\neq J$
\begin{align*}
\bE \left((T_{i,J}\cdot T_{i, k_{1}})(T_{i,k_{2}}\cdot T_{i, k_{3}})\right)&= \bE \left(\bE \left((T_{i, J}\cdot T_{i, k_{1}}) (T_{i, k_{2}}\cdot T_{i, k_{3}})\vert \btheta_{r}, r\neq J\right)\right),\\
&=\bE \left( \left(T_{i, k_{1}}\cdot\underbrace{\bE \left(T_{i, J}\vert \btheta_{r}, r\neq l\right)}_{=0}\right) (T_{i, k_{2}}\cdot T_{i, k_{3}})\right)=0.
\end{align*}
Consequently,
\begin{align}
\bE \left(d_{N}^{8}\right)&= \bE \left\vert \frac{1}{ \left\vert \Lambda_{N} \right\vert} \sum_{j\neq i} T_{i,j}\Psi(x_{i}, x_{j}) \right\vert^{8},\nonumber\\
&= \frac{1}{ \left\vert \Lambda_{N} \right\vert^{8}} \sum_{k_{1}, k_{2}, k_{3}, k_{4}\neq i} \left(\prod_{l=1}^{4} \Psi(x_{i}, x_{k_{l}})^{ 2}\right)\bE ( \left\vert T_{i, k_{1}} \right\vert^{ 2} \left\vert T_{i, k_{2}} \right\vert^{ 2} \left\vert T_{i, k_{3}} \right\vert^{ 2} \left\vert T_{i, k_{4}} \right\vert^{ 2})\nonumber\\
&= \frac{1}{ \left\vert \Lambda_{N} \right\vert^{8}} \sum_{\substack{k_{1}, k_{2}, k_{3}, k_{4}\\\text{distincts}}} \left(\prod_{l=1}^{4} \Psi(x_{i}, x_{k_{l}})^{ 2}\right)\bE ( \left\vert T_{i, k_{1}} \right\vert^{ 2} \left\vert T_{i, k_{2}} \right\vert^{ 2} \left\vert T_{i, k_{3}} \right\vert^{ 2} \left\vert T_{i, k_{4}} \right\vert^{ 2}),\nonumber\\
&+\frac{6}{ \left\vert \Lambda_{N} \right\vert^{8}} \sum_{\substack{j,k,l\\\text{distincts}}} \Psi(x_{i}, x_{j})^{2}\Psi(x_{i}, x_{k})^{2} \Psi(x_{i}, x_{l})^{4} \bE( \left\vert T_{i,j} \right\vert^{2}\left\vert T_{i,k} \right\vert^{2} \left\vert T_{i,l} \right\vert^{4})\nonumber\\
&+\frac{2}{ \left\vert \Lambda_{N} \right\vert^{8}} \sum_{k\neq l} \Psi(x_{i}, x_{k})^{4} \Psi(x_{i}, x_{l})^{4} \bE( \left\vert T_{i,k} \right\vert^{4} \left\vert T_{i,l} \right\vert^{4})+ \frac{1}{ \left\vert \Lambda_{N} \right\vert^{8}} \sum_{k} \bE \left\vert T_{i,k} \right\vert^{8}\Psi(x_{i}, x_{k})^{8}\nonumber\\
&\leq C\Ninf{\Gamma}^{8}\Bigg(\frac{1}{ \left\vert \Lambda_{N} \right\vert^{8}} \sum_{\substack{k_{1}, k_{2}, k_{3}, k_{4}\\\text{distincts}}} \prod_{l=1}^{4} \Psi(x_{i}, x_{k_{l}})^{ 2} +\frac{1}{ \left\vert \Lambda_{N} \right\vert^{8}} \sum_{\substack{j,k,l\\\text{distincts}}} \Psi(x_{i}, x_{j})^{2}\Psi(x_{i}, x_{k})^{2} \Psi(x_{i}, x_{l})^{4}\nonumber\\&\qquad+ \frac{1}{ \left\vert \Lambda_{N} \right\vert^{4}} \sum_{k \neq l} \Psi(x_{i}, x_{k})^{4\alpha} \Psi(x_{i}, x_{l})^{4\alpha} + \frac{1}{ \left\vert \Lambda_{N} \right\vert^{8}} \sum_{k} \Psi(x_{i}, x_{k})^{8\alpha}\Bigg),\label{aux:dN}
\end{align}
where we used that $ \left\vert T_{ i, j} \right\vert\leq 2 \N{ \Gamma}_{ \infty}$ and where $C$ is a numerical constant.
An application of Lemma~\ref{lem:Riem} on \eqref{aux:dN} leads to an upper bound of $\bE \left(d_{N}^{8}\right)$ of the form $\frac{C}{N^{4}}$ if $0\leq\alpha< \frac{1}{2}$, (respectively  $\frac{C}{N^{8(1-\alpha)}}$ if $\frac{1}{2}< \alpha<1$), for some constant $C$.
It remains to treat the term $e_{N}$ in \eqref{aux:def_eN}: for simplicity, we use the notation $\pi_{ i, s}:= \int \Gamma(\btheta_{i,s}, \omega_{i}, \theta, \omega)\xi_{s}(\dd\theta, \dd\omega)$, so that
\begin{align*}
e_{N}&= \left\vert \sum_{j\in\Lambda_{N}} \int_{\Delta_{j}} \pi_{i, s} \left( \Psi(x_{i}, x_{j}) - \Psi(x_{i}, x)\right)\dd x \right\vert
\leq\Ninf{\Gamma}\underbrace{\sum_{j\in\Lambda_{N}}\int_{\Delta_{j}}\left\vert \Psi(x_{i}, x_{j})- \Psi(x_{i}, x) \right\vert\dd x}_{\leq C N^{-(1-\alpha)}},\\
&\leq \frac{C\Ninf{ \Gamma}}{N^{1-\alpha}}.
\end{align*}
Taking the expectation in \eqref{aux:ttavsnlin} and applying Gronwall's lemma ends the proof of Proposition~ \ref{prop:ttaVSnonlin}. \qed

\subsection{Proof of Proposition~\ref{prop:etaNbounded1}}
\label{sec:proof_etaNbounded1}
Since we work in this proof with the same Sobolev space $\bW_{ q}^{ \kappa_{ 1}, \iota_{ 1}}$ (and its dual), we write in this proof $ \N{\cdot}_{ q}$ instead of $ \N{\cdot}_{ q, \kappa_{ 1}, \iota_{ 1}}$, for simplicity of notations. Introducing the nonlinear process $ \bar\tau_{ i, t} =( \ttheta_{ i, t}, \omega_{ i}, x_{ i})$ (Proposition~\ref{prop:ttaVSnonlin}), one can decompose $\cro{\cH_{ N,t}}{g}$ into the sum of three terms:
\begin{align}
\cI_{N, t}(g)&:= \frac{a_{N}}{ \left\vert \Lambda_{ N} \right\vert^{ 2}} \sum_{i, j\in\Lambda_{N}} \Psi(x_{ i}, x_{ j})\left(g(\tau_{ i, t}, \tau_{ j,t}\right)-g(\bar\tau_{ i,t}, \bar\tau_{ j, t})),\label{eq:def_cIN}\\
\cJ_{N, t}(g)&:=a_{N} \left(\frac{1}{ \left\vert \Lambda_{ N} \right\vert^{ 2}} \sum_{i, j\in\Lambda_{N}} \Psi(x_{ i}, x_{ j})g(\bar\tau_{ i,t}, \bar\tau_{ j, t}) - \frac{ 1}{ \left\vert \Lambda_{ N} \right\vert} \sum_{ i\in \Lambda_{ N}} \int \Psi(x_{ i}, \tx)g(\bar\tau_{ i, t}, \ttau) \nu_{t}(\dd \ttau)\right),\label{eq:def_cKN}\\
\cK_{ N,t}(g)&:= \frac{ a_{N}}{ \left\vert \Lambda_{ N} \right\vert} \sum_{ i\in \Lambda_{ N}} \int \Psi(x_{ i}, \tx)\left(g(\bar\tau_{ i, t}, \ttau) - g(\tau_{ i, t}, \ttau)\right) \nu_{t}(\dd \ttau),\label{eq:def_cJN}
\end{align}
Intuitively, the boundedness of $\cI_{ N}$ and $\cK_{ N}$ in this decomposition comes from Proposition~\ref{prop:ttaVSnonlin}: each particle $ \theta_{ i}$ is well approximated by its corresponding nonlinear process $\bar \theta_{ i}$. The most technical part of the proof concerns the evaluation of $\cJ_{ N}$ and relies on cancellation arguments in the expression of $\cJ_{ N}(g)^{ 2}$, using the independence of the $(\bar \theta_{ i})_{ i\in \Lambda_{ N}}$. The main point of the proof is that we need an estimate that is uniform in $g$. To do so, we use Hibertian techniques similar to \cite{Fernandez1997,Lucon2011}.

\medskip
\noindent
Concerning $\cI_{N, t}(g)$, one has 
\begin{align*}
\cI_{N, t}(g)^{2}&= \frac{a_{N}^{2}}{ \left\vert \Lambda_{ N} \right\vert^{ 4}} \sum_{i, j, k, l\in\Lambda_{N}} \Psi(x_{ i}, x_{ j})\Psi(x_{ k}, x_{ l})\left(g(\tau_{ i, t}, \tau_{ j,t}\right)-g(\bar\tau_{ i,t}, \bar\tau_{ j, t}))\left(g(\tau_{ k, t}, \tau_{ l,t}\right)-g(\bar\tau_{ k,t}, \bar\tau_{ l, t}))\\
&=\frac{a_{N}^{2}}{ \left\vert \Lambda_{ N} \right\vert^{ 4}} \sum_{i, j, k, l\in\Lambda_{N}} \Psi(x_{ i}, x_{ j})\Psi(x_{ k}, x_{ l})\cR_{ i,j, t}(g)\cR_{ k,l, t}(g),
\end{align*}
where we have written, for simplicity $\cR_{ i,j, t}=\cR_{\tau_{ i, t}, \tau_{ j,t}, \bar\theta_{ i, t}, \bar\theta_{ j, t}} $ (recall \eqref{eq:def_cR}).
Applying the previous equality to an orthonormal system $(g_{p})_{p\geq1}$ in the Hilbert space $\bW^{ \kappa_{ 1}, \iota_{ 1}}_{ q}$ and using Parseval's identity, one obtains
\begin{align*}
\N{\cI_{N, t}}_{-q}^{4} &= \left(\sum_{ p\geq1} \left\vert \cI_{ N,t}(g_{ p}) \right\vert^{ 2}\right)^{2}\\
&= \left(\frac{a_{N}^{2}}{ \left\vert \Lambda_{ N} \right\vert^{ 4}} \sum_{i, j, k, l\in\Lambda_{N}} \Psi(x_{ i}, x_{ j})\Psi(x_{ k}, x_{ l})\sum_{ p\geq1}\cR_{ i,j, t}(g_{ p})\cR_{ k,l, t}(g_{ p})\right)^{ 2}\\
&\leq \left(\frac{a_{N}^{2}}{ \left\vert \Lambda_{ N} \right\vert^{ 4}} \sum_{i, j, k, l\in\Lambda_{N}}\Psi(x_{ i}, x_{ j})\Psi(x_{ k}, x_{ l})\N{\cR_{ i,j, t}}_{ -q}\N{\cR_{ k,l, t}}_{ -q}\right)^{ 2},\\
&= \frac{a_{N}^{4}}{ \left\vert \Lambda_{ N} \right\vert^{ 8}} \sum_{\substack{i_{ 1}, \ldots, i_{ 4}\\ j_{ 1}, \ldots, j_{ 4}}} \prod_{ l=1}^{ 4}\Psi(x_{ i_{ l}}, x_{ j_{ l}}) \prod_{ l=1}^{ 4}\N{\cR_{ i_{ l},j_{ l}, t}}_{ -q}.
\end{align*}
By Proposition~ \ref{prop:linearforms} and Cauchy-Schwarz inequality, 
\begin{align*}
\bE \left(\prod_{ l=1}^{ 4}\N{\cR_{ i_{ l},j_{ l}, t}}_{ -q}\right)&\leq C \bE \left(\prod_{ l=1}^{ 4} \chi_{ i_{ l}, j_{ l}, t} \prod_{ l=1}^{ 4}\left( \left\vert \theta_{ i_{ l}, t} - \bar \theta_{ i_{ l}, t} \right\vert + \left\vert \theta_{ j_{ l}, t} - \bar \theta_{ j_{ l}, t} \right\vert\right)\right)\\
&\leq C\prod_{ l=1}^{ 4}\bE \left(\chi_{ i_{ l},j_{ l}, t}^{ 8}\right)^{ \frac{ 1}{ 8}}  \prod_{ l=1}^{ 4}\bE \left( \left(\left\vert \theta_{ i, t}- \bar\theta_{ i, t} \right\vert + \left\vert \theta_{ j, t}- \bar\theta_{ j, t} \right\vert\right)^{ 8}\right)^{ \frac{ 1}{ 8}}\leq \frac{ C}{ a_{ N}^{ 4}},
\end{align*}
where $\chi_{ i,j, t}$ is a shortcut for $\chi_{ \kappa_{ 1}, \iota_{ 1}}( \tau_{ i, t}, \tau_{ j, t}, \bar\theta_{ i, t}, \bar \theta_{ j, t})$ (recall \eqref{eq:chi}) and where we used \eqref{eq:ttaVSnonlin} and Proposition~\ref{prop:moment_particles}. Consequently, using \eqref{eq:Psi_bounded},
\begin{align*}
\sup_{1\leq N, t\leq T}\bE \left(\N{\cI_{N, t}}_{-q}^{4}\right)&\leq\sup_{1\leq N} \left(\frac{C}{ \left\vert \Lambda_{ N} \right\vert^{ 4}} \sum_{i_{ 1}, \ldots, i_{ 4}\in\Lambda_{N}} \prod_{ l=1}^{ 4} \left( \frac{ 1}{ \left\vert \Lambda_{ N}\right\vert} \sum_{ j\in \Lambda_{ N}}\Psi(x_{ i_{ l}}, x_{ j})\right)\right)<+\infty.
\end{align*}
We now turn to the estimation of $\cJ_{ N, t}$ in \eqref{eq:def_cKN}: for all $N\geq1$, $t\in[0, T]$
\begin{align}
\cJ_{N, t}(g)&= \frac{ a_{ N}}{ \left\vert \Lambda_{ N} \right\vert} \sum_{ i\in \Lambda_{ N}}\left(\frac{1}{ \left\vert \Lambda_{ N} \right\vert} \sum_{j\in\Lambda_{N}} \Psi(x_{ i}, x_{ j})g(\bar\tau_{ i,t}, \bar\tau_{ j, t}) -  \int \Psi(x_{ i}, \tx)g( \bar\tau_{ i, t}, \ttau) \nu_{t}(\dd \ttau)\right),\nonumber\\
&= \frac{ a_{ N}}{ \left\vert \Lambda_{ N} \right\vert^{ 2}} \sum_{ i, j\in \Lambda_{ N}}\Psi(x_{ i}, x_{ j}) \left( g(\bar\tau_{ i,t}, \bar\tau_{ j, t}) -  \int g( \bar\tau_{ i, t}, \ttheta, \tomega, x_{ j}) \xi_{t}(\dd \ttheta, \dd \tomega)\right)\nonumber\\&+
\frac{ a_{ N}}{ \left\vert \Lambda_{ N} \right\vert} \sum_{ i\in \Lambda_{ N}} \left( \frac{ 1}{ \left\vert \Lambda_{ N} \right\vert}\sum_{ j\in \Lambda_{ N}} \Psi(x_{ i}, x_{ j}) \int g( \bar\tau_{ i, t}, \ttheta, \tomega, x_{ j}) \xi_{t}(\dd \ttheta, \dd \tomega)-  \int \Psi(x_{ i}, \tx)g( \bar\tau_{ i, t}, \ttau) \nu_{t}(\dd \ttau)\right),\nonumber\\
&:= J_{ N,t}^{ (1)}(g) + J_{ N,t}^{ (2)}(g).\label{eq:JN}
\end{align}
The linear form $J_{ N}^{ (1)}$ in \eqref{eq:JN} captures the random fluctuations of the nonlinear processes around their mean value (for fixed positions), whereas $J_{ N}^{ (2)}$ captures the variations of the positions of the particles (for a fixed randomness). The first term in \eqref{eq:JN} can be simply written as $J_{ N,t}^{ (1)}(g)=  \frac{ a_{ N}}{ \left\vert \Lambda_{ N} \right\vert^{ 2}} \sum_{ i, j\in \Lambda_{ N}}\Psi(x_{ i}, x_{ j}) A_{ i, j, t}(g)$, where 
\begin{equation}
\label{eq:def_Aij}
A_{ i, j, t}(g):= g(\bar\tau_{ i,t}, \bar\tau_{ j, t}) -  \int g( \bar\tau_{ i, t}, \ttheta, \tomega, x_{ j}) \xi_{ t}(\dd \ttheta, \dd \tomega)= g(\bar\tau_{ i,t}, \bar\tau_{ j, t}) - \bE \left(  g(\bar\tau_{ i,t}, \bar\tau_{ j, t})  \vert \bar\tau_{ r, t}, r\neq j\right).
\end{equation} 
Using again a complete orthonormal system $(g_{ p})_{ p\geq1}$ in $\bW^{ \kappa_{ 1}, \iota_{ 1}}_{ q}$,
\begin{align}
\N{J_{ N,t}^{ (1)}}^{ 4}_{ -q}&= \left(\sum_{ r\geq1}J_{ N,t}^{ (1)}(g_{ r})^{ 2}\right)^{ 2},\nonumber\\
&= \left(\frac{ a_{ N}^{ 2}}{ \left\vert \Lambda_{ N} \right\vert^{ 4}}\sum_{ r\geq1} \sum_{i, j, k, l \in \Lambda_{ N}}\Psi(x_{ i}, x_{ j})\Psi(x_{ k}, x_{ l}) A_{ i, j, t}(g_{ r}) A_{ k, l, t}(g_{ r})\right)^{ 2},\nonumber\\
&=\frac{ a_{ N}^{ 4}}{ \left\vert \Lambda_{ N} \right\vert^{ 8}}\sum_{ r, s\geq1} \sum_{ \substack{i_{ 1}, \ldots, i_{ 4}\\  j_{ 1}, \ldots j_{ 4}}} \prod_{ l=1}^{ 4}\Psi(x_{ i_{ l}}, x_{ j_{ l}}) A_{ i_{ 1}, j_{ 1}, t}(g_{ r})A_{ i_{ 2}, j_{ 2}, t}(g_{ r}) A_{ i_{ 3}, j_{ 3}, t}(g_{ s})A_{ i_{ 4}, j_{ 4}, t}(g_{ s}).\label{aux:JN1_4}
\end{align}
Taking the expectation in \eqref{aux:JN1_4}, we need to estimate the quantity \begin{equation}
\label{eq:def_E1234}E(I_{ 1}, I_{ 2}, I_{ 3}, I_{ 4}):=\bE \left(A_{ i_{ 1}, j_{ 1}, t}(g_{ r})A_{ i_{ 2}, j_{ 2}, t}(g_{ r}) A_{ i_{ 3}, j_{ 3}, t}(g_{ s})A_{ i_{ 4}, j_{ 4}, t}(g_{ s})\right),
\end{equation}
for every sets of couples $I_{ l}:=(i_{ l}, j_{ l})$, $l=1,\ldots, 4$. One crucial observation here is that for all $j\in \Lambda_{N}$,
\begin{equation}
\label{eq:cancel_Aij}
\bE \left(A_{ i, j, t}(g) \vert \bar\tau_{ r, t}, r \neq j\right)=0,
\end{equation}
by definition of the $A_{ i, j}$ in \eqref{eq:def_Aij}. \eqref{eq:cancel_Aij} leads to cancellations in \eqref{eq:def_E1234}. 
\begin{figure}
\includegraphics[width=\textwidth]{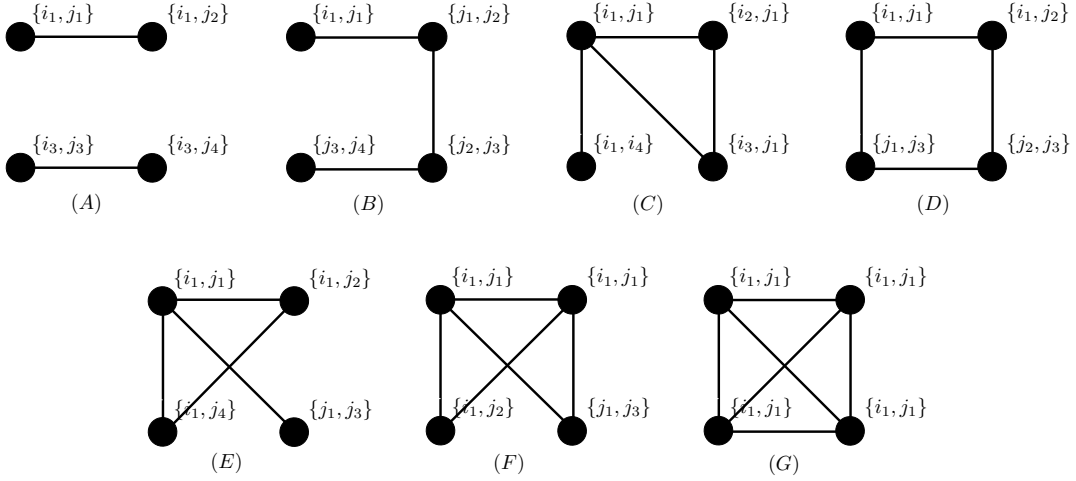}
\caption{The only nontrivial contributions to \eqref{eq:def_E1234} corresponds to indices given by the above diagrams.}
\label{fig:cancell_cases}
\end{figure}
We give here a simple combinatoric argument in order to derive the relevant contributions to the estimation of the expectation of \eqref{aux:JN1_4}: associate a vertex to each $\{i_{ l}, j_{ l}\}$ ($l=1, \ldots, 4$) and draw an edge between two vertices labelled $k$ and $l$ ($k,l=1, \ldots, 4$) if and only if $\{i_{ l}, j_{ l}\}\cap\{i_{ k}, j_{ k}\}\neq \emptyset$ (see Figure~\ref{fig:cancell_cases}). Note that when a vertex (labelled for example by $I_{ 1}=\{i_{ 1}, j_{ 1}\}$) is isolated (that is $\{i_{ 1}, j_{ 1}\}\cap\{i_{ 2}, i_{ 3}, i_{ 4}, j_{ 2}, j_{ 3}, j_{ 4}\}= \emptyset$), by independence of the random variables $\bar \tau_{ i}$, one obtains that 
\begin{align*}
E(I_{ 1}, I_{ 2}, I_{ 3}, I_{ 4}) = \bE(A_{ i_{ 2}, j_{ 2}}A_{ i_{ 3}, j_{ 3}} A_{ i_{ 4}, j_{ 4}}\bE(A_{ i_{ 1}, j_{ 1}}\vert \bar \tau_{ r}, r\neq j_{ 1}))=0.
\end{align*} 
Using this last observation, it is easy to enumerate by inspection all the possible nontrivial contributions to the expectation of \eqref{aux:JN1_4}: they are listed in Figure~\ref{fig:cancell_cases}. We only give here a detailed analysis of the two extreme cases (i.e. the largest $(A)$ and smallest $(G)$) and leave the remaining cases to the reader. The largest contribution of $E(I_{ 1}, I_{ 2}, I_{ 3}, I_{ 4})$ in $\bE\left(\N{J_{ N,t}^{ (1)}}^{ 4}_{ -q}\right)$ corresponds to the case $(A)$ in Figure~\ref{fig:cancell_cases} where the set of indices $\{I_{ 1}, I_{ 2}, I_{ 3}, I_{ 4}\}$ is separated into $\{I_{ 1}, I_{ 2}\}$ and $\{I_{ 3}, I_{ 4}\}$ with the property that $\{i_{ 1}, j_{ 1}, i_{ 2}, j_{ 2}\}\cap\{i_{ 3}, j_{ 3}, i_{ 4}, j_{ 4}\}=\emptyset$. By independence of the random variables $(\bar \tau_{ i})_{ i\in \Lambda_{ N}}$, the contribution of this term to $\bE\left(\N{J_{ N,t}^{ (1)}}^{ 4}_{ -q}\right)$ is then 
\begin{align}
\left(\sum_{ p\geq1}\frac{ a_{ N}^{ 2}}{ \left\vert \Lambda_{ N} \right\vert^{ 4}} \sum_{ i, j,k,l\in \Lambda_{ N}}\Psi(x_{ i}, x_{ j})\Psi(x_{ k}, x_{ l}) \bE \left(A_{ i, j, t}(g_{ p})A_{ k, l, t}(g_{ p})\right)\right)^{ 2}.\label{eq:AijAkl}
\end{align}
First note that, thanks to \eqref{eq:cancel_Aij}, whenever $i\neq l$ and $j \neq l$, $\bE \left(A_{ i, j, t}(g) A_{ k, l, t}(g)\right)= \bE \left( A_{ i, j, t}(g)\bE \left( A_{ k, l, t}(g)\vert \bar\tau_{ r, t}, r\neq l\right)\right)=0$, so that we can suppose that the indices $(i,j,k,l)$ in \eqref{eq:AijAkl} are such that $i=l$ or $j=l$. 

We first treat the case $i=l$: here, we can further assume that $k=j$, since if $k\neq j$, one has that $\bE(A_{ i,j} A_{ k,i})=\bE(A_{ k,i}\bE(A_{ i,j} \vert \bar \tau_{ r}, r\neq j)=0$. Hence, for fixed $i, j$,
\begin{align*}
\sum_{ p\geq 1} \bE \left(A_{ i,j ,t}(g_{ p}) A_{ j, i, t}(g_{ p})\right)&\leq \sum_{ p\geq 1}\bE \left(A_{ i,j,t}(g_{ p})^{ 2}\right)^{ \frac{ 1}{ 2}} \bE \left(A_{ j,i,t}(g_{ p})^{ 2}\right)^{ \frac{ 1}{ 2}},\\
&\leq C \sum_{ p\geq 1}\bE \left(g_{ p}(\bar \tau_{ i, t}, \bar \tau_{ j, t})^{ 2}\right)^{ \frac{ 1}{ 2}} \bE \left(g_{ p}(\bar\tau_{ j, t}, \bar\tau_{ i, t})^{ 2}\right)^{ \frac{ 1}{ 2}},\\
&\leq C \left(\sum_{ p\geq 1}\bE \left(g_{ p}(\bar \tau_{ i, t}, \bar \tau_{ j, t})^{ 2}\right)\right)^{ \frac{ 1}{ 2}} \left(\sum_{ p\geq1} \bE \left(g_{ p}(\bar\tau_{ j, t}, \bar\tau_{ i, t})^{ 2}\right)\right)^{ \frac{ 1}{ 2}},\\
&= C \bE \left(\N{\cS_{ \bar \tau_{ i, t}, \bar \tau_{ j, t}}}^{ 2}_{ -q}\right)^{ \frac{ 1}{ 2}} \bE \left(\N{\cS_{ \bar \tau_{ j, t}, \bar \tau_{ i, t}}}^{ 2}_{ -q}\right)^{ \frac{ 1}{ 2}},
\end{align*}
where $\cS$ is defined in \eqref{eq:def_cS}. Thanks to Proposition \ref{prop:linearforms} and Proposition~\ref{prop:moment_particles}, the last quantity is bounded by a constant that is uniform in $i, j$. Consequently, the contribution of the case $i=l$, $j=k$ to \eqref{eq:AijAkl} can be bounded above by 
\begin{equation}
C \left(\frac{ a_{ N}^{ 2}}{ \left\vert \Lambda_{ N} \right\vert^{ 4}} \sum_{ i, j\in \Lambda_{ N}}\Psi(x_{ i}, x_{ j})^{ 2}\right)^{ 2}= C \left(\frac{ a_{ N}^{ 2}}{ \left\vert \Lambda_{ N} \right\vert^{ 4}} \sum_{ i, j\in \Lambda_{ N}} d(x_{ i}, x_{ j})^{ -2\alpha}\right)^{ 2}.
\end{equation}
When $ \alpha< \frac{ 1}{ 2}$, the latter quantity is bounded above by $ \left(C \frac{ a_{ N}^{ 2}}{ N^{ 2}}\right)^{ 2}= \frac{ C}{ N^{ 2}}$, by \eqref{eq:Psi_bounded} . When $ \alpha> \frac{ 1}{ 2}$, this contribution is of order $ \left(\frac{ N^{ 2(1- \alpha)}}{ N^{ 3}} N^{ 2 \alpha} \right)^{ 2} = \frac{ C}{ N^{ 2}}$.

We now treat the case where $j=l$ in \eqref{eq:AijAkl}. Proceeding as before and using again \eqref{eq:Psi_bounded}, one easily sees that the contribution of this case to \eqref{eq:AijAkl} can be bounded above by
\begin{equation}
C\left(\frac{ a_{ N}^{ 2}}{ \left\vert \Lambda_{ N} \right\vert^{ 4}} \sum_{ i, j, k \in \Lambda_{ N}}\Psi(x_{ i}, x_{ j})\Psi(x_{ k}, x_{ j}) \right)^{ 2} \leq C\left(\frac{ a_{ N}^{ 2}}{ N} \right)^{ 2},
\end{equation}
which is uniformly bounded in $N$, in both cases $ \alpha< \frac{ 1}{ 2}$ and $ \alpha> \frac{ 1}{ 2}$, by definition of $ a_{ N}$. One can conclude that \eqref{eq:AijAkl} is uniformly bounded in $N$.

We now treat the case of the smallest contribution to $\bE\left(\N{J_{ N,t}^{ (1)}}^{ 4}_{ -q}\right)$ which corresponds to the case where all the four couples of indices are equal (see the case $(G)$ in Figure~\ref{fig:cancell_cases}): $I_{ 1}= I_{ 2}= I_{ 3}= I_{ 4}= (i, j)$. This case boils down to estimating the following quantity
\begin{align}
\frac{ a_{ N}^{ 4}}{ \left\vert \Lambda_{ N} \right\vert^{ 8}} \sum_{ i, j \in \Lambda_{ N}} \Psi(x_{ i}, x_{ j})^{ 4}  \sum_{ r, s\geq1}\bE \left( A_{ i, j, t}(g_{ r})^{ 2} A_{ i, j, t}(g_{ s})^{ 2}\right)&= \frac{ a_{ N}^{ 4}}{ \left\vert \Lambda_{ N} \right\vert^{ 8}} \sum_{ i, j \in \Lambda_{ N}} \Psi(x_{ i}, x_{ j})^{ 4} \bE \left( \N{A_{ i, j, t}}_{ -q}^{ 4}\right)\label{eq:Aij4}\\
&\leq \frac{ C a_{ N}^{ 4}}{ \left\vert \Lambda_{ N} \right\vert^{ 8}} \sum_{ i, j \in \Lambda_{ N}} d(x_{ i}, x_{ j})^{ -4 \alpha}. \nonumber
\end{align}
Using again \eqref{eq:Psi_bounded}, we see that if $ 0\leq \alpha < \frac{ 1}{ 4}$, the last quantity is bounded above by $C \frac{ a_{ N}^{ 4}}{ N^{ 8} N^{ 2}}= \frac{ C}{ N^{ 4}}$. If $ \alpha= \frac{ 1}{ 4}$, we obtain an upper-bound of order $ \frac{\ln N}{ N^{ 4}}$. If $ \frac{ 1}{ 4} < \alpha < \frac{ 1}{ 2}$, the upper-bound is of order $ \frac{N^{ 4 \alpha} }{ N^{ 5}} \leq\frac{ 1}{ N^{ 3}}$ and in the case $ \frac{ 1}{ 2} < \alpha< 1$, one has an upper-bound of order $ \frac{ N^{ 4(1- \alpha)}}{ N^{ 7}} N^{ 4 \alpha} = \frac{ 1}{ N^{ 3}}$. In any case, this term is bounded in $N$. From all this we can conclude that 
\begin{equation}
\label{aux:bound_JN1}
\sup_{t\leq T}\bE\left( \N{J_{ N, t}^{ (1)}}_{ -q}^{ 4}\right) \leq C \left(\frac{ a_{ N}^{ 2}}{ N}\right)^{ 2}.
\end{equation}

We now turn to the second term $J_{ N, t}^{ (2)}$ in \eqref{eq:JN}. Setting $\left[g, \xi_{ t}\right](\tau, \tx)= \left[g, \xi_{ t}\right]_{ 1}(\tau, \tx)= \int g(\tau, \ttheta, \tomega, \tx) \xi_{t}(\dd \ttheta, \dd \tomega)$ (recall \eqref{eq:cro_1}) and defining $ \Delta_{j}:=[x_{ j}, x_{ j+1}[\subset \bbS$, $j\in \Lambda_{ N}$, one obtains
\begin{align*}
J_{ N, t}^{ (2)}(g)&= \frac{ a_{ N}}{ \left\vert \Lambda_{ N} \right\vert} \sum_{ i\in \Lambda_{ N}} \left( \frac{ 1}{ \left\vert \Lambda_{ N} \right\vert}\sum_{ j\in \Lambda_{ N}} \Psi(x_{ i}, x_{ j}) \left[g, \xi_{ t}\right](\bar\tau_{ i, t}, x_{ j})-  \int \Psi(x_{ i}, \tx)\left[g, \xi_{ t}\right](\bar\tau_{ i, t}, \tx)\dd \tx\right),\\
&=\frac{ a_{ N}}{ \left\vert \Lambda_{ N} \right\vert} \sum_{ i, j\in \Lambda_{ N}}  \int_{ \Delta_{ j}} \left(\Psi(x_{ i}, x_{ j}) \left[g, \xi_{ t}\right](\bar\tau_{ i, t}, x_{ j})-  \Psi(x_{ i}, \tx)\left[g, \xi_{ t}\right](\bar\tau_{ i, t}, \tx)\right)\dd \tx,\\
&= \frac{ a_{ N}}{ \left\vert \Lambda_{ N} \right\vert} \sum_{ i, j\in \Lambda_{ N}}  \int_{ \Delta_{ j}} \left[g, \xi_{ t}\right](\bar\tau_{ i, t}, x_{ j}) \left(\Psi(x_{ i}, x_{ j}) - \Psi(x_{ i}, \tx) \right)\dd \tx\\
&+ \frac{ a_{ N}}{ \left\vert \Lambda_{ N} \right\vert} \sum_{ i, j\in \Lambda_{ N}} \int_{ \Delta_{ j}} \Psi(x_{ i}, \tx)\left( \left[g, \xi_{ t}\right](\bar\tau_{ i, t}, x_{ j})- \left[g, \xi_{ t}\right](\bar\tau_{ i, t}, \tx)\right)\dd \tx,\\&:= J_{ N, t}^{ (3)}(g) + J_{ N, t}^{ (4)}(g).
\end{align*}
Concerning $J_{ N,t}^{ (3)}$, setting $u_{ i,j}(g) := \left[g, \xi_{ t}\right](\bar\tau_{ i, t}, x_{ j})\int_{ \Delta_{ j}}  \left(\Psi(x_{ i}, x_{ j}) - \Psi(x_{ i}, \tx) \right)\dd \tx$, we have
\begin{align*}
\bE \left(\N{J_{ N,t}^{ (3)}}^{ 4}_{ -q}\right)&= \frac{ a_{ N}^{ 4}}{ \left\vert \Lambda_{ N} \right\vert^{ 4}} \sum_{ \substack{i_{ 1}, \ldots, i_{ 4}\\ j_{ 1}, \ldots, j_{ 4}}} \bE \left( \sum_{ r, s\geq 1} u_{ i_{ 1}, j_{ 1}}(g_{ r}) u_{ i_{ 2}, j_{ 2}}(g_{ r})u_{ i_{ 3}, j_{ 3}}(g_{ s})u_{ i_{ 4}, j_{ 4}}(g_{ s})\right),\\
&\leq \left(\frac{ a_{ N}^{ 2}}{ \left\vert \Lambda_{ N} \right\vert^{ 2}}\sum_{ i, j, k, l\in \Lambda_{ N}}  \left\lbrace \bE \left(\sum_{ p\geq 1} u_{ i,j}(g_{ p})^{ 2}\right) \left(\sum_{ p\geq1}u_{ k,l}(g_{ p})^{ 2}\right)\right\rbrace ^{ \frac{ 1}{ 2}}\right)^{ 2}.
\end{align*}
Using \eqref{eq:regPsi}, it is easy to see that $\int_{ \Delta_{ j}}  \left\vert\Psi(x_{ i}, x_{ j}) - \Psi(x_{ i}, \tx) \right\vert\dd \tx \leq \frac{ N^{ \alpha-1}}{ \left\vert i-j \right\vert^{ \alpha+1}} $. Moreover,
\begin{align*}
\sum_{ p\geq1}   \bE \left(\left[g_{ p}, \xi_{ t}\right](\bar\tau_{ i, t}, x_{ j})^{ 2}\right) &\leq \sum_{ p\geq1}   \bE \left(\int g_{ p}(\bar\tau_{ i, t},  \ttheta, \tomega, x_{ j})^{ 2} \xi_{ t}(\dd \ttheta, \dd \tomega)\right),\\ &= \bE \left( \int \N{\cS_{ \bar\tau_{ i,t}, \ttheta, \tomega, x_{ j}}}^{ 2}_{ -q} \xi_{ t}(\dd \ttheta, \dd \tomega)\right),
\end{align*} which is uniformly bounded in $i, j$ and $N$ (recall Proposition~\ref{prop:linearforms}). Consequently, for some constant $C>0$,
\begin{align*}
\left\lbrace \bE \left(\sum_{ p\geq 1} u_{ i,j}(g_{ p})^{ 2}\right) \left(\sum_{ p\geq1}u_{ k,l}(g_{ p})^{ 2}\right)\right\rbrace ^{ \frac{ 1}{ 2}}&\leq \frac{ CN^{ 2(\alpha-1)}}{ \left\vert i-j \right\vert^{ \alpha+1} \left\vert k-l \right\vert^{ \alpha+1}}.
\end{align*}
Thus, one obtains that 
\begin{align*}
\sup_{t\leq T } \bE \left(\N{J_{ N, t}^{ (3)}}^{ 4}_{ -q}\right)&\leq C\left(\frac{ a_{ N}^{ 2}}{ \left\vert \Lambda_{ N} \right\vert^{ 2}}\sum_{ i, j, k, l\in \Lambda_{ N}}  \frac{ N^{ 2(\alpha-1)}}{ \left\vert i-j \right\vert^{ \alpha+1} \left\vert k-l \right\vert^{ \alpha+1}}\right)^{ 2}\leq C\frac{ a_{ N}^{ 4}}{ N^{ 4(1- \alpha)}},
\end{align*}
which is bounded, uniformly in $N$, by definition of $a_{ N}$.
Turning to $J_{ N, t}^{ (4)}$ and setting $\cU_{ i, j}(g):=\cU_{\bar \tau_{ i,t}, x_{ j}, x_{ i}, \Delta_{ j}, \xi_{ t}}(g)=\int_{ \Delta_{ j}} \Psi(x_{ i}, \tx)\left( \left[g, \xi_{ t}\right](\bar\tau_{ i, t}, x_{ j})- \left[g, \xi_{ t}\right](\bar\tau_{ i, t}, \tx)\right)$ (recall the definition of the form $\cU$ in Proposition~\ref{prop:contU}), we have
\begin{align*}
\bE \left(\N{J_{ N, t}^{ (4)}}^{ 4}_{ -q}\right)&\leq \frac{ a_{ N}^{ 4}}{ \left\vert \Lambda_{ N} \right\vert^{ 4}} \sum_{ \substack{ i_{ 1}, \ldots, i_{ 4}\\j_{ 1}, \ldots, j_{ 4}}}\bE \left(\prod_{ l=1}^{ 4}\N{\cU_{ i_{ l}, j_{ l}}}_{ -q}\right).
\end{align*}
Hence, using the boundedness of the linear form $\cU$ in \eqref{eq:contU}
\begin{align*}
\bE \left(\N{J_{ N, t}^{ (4)}}^{ 4}_{ -q}\right)&\leq C\frac{ a_{ N}^{ 4}}{ \left\vert \Lambda_{ N} \right\vert^{ 4}} \sum_{ \substack{ i_{ 1}, \ldots, i_{ 4}\\j_{ 1}, \ldots, j_{ 4}}} \bE\prod_{ l=1}^{ 4} \left(\left(1+ \left\vert \theta_{ i_{ l}, t} \right\vert^{ \kappa_{ 1}} + \left\vert \omega_{ i_{ l}} \right\vert^{\iota_{ 1}}\right) \left(\sup_{ x\in \Delta_{ j_{ l}}}\left\vert x_{ j_{ l}}-x \right\vert \right)\int_{ \Delta_{ j_{ l}}} \Psi(x_{ i_{ l}}, \tx)  \dd \tx\right),\\
&\leq C\frac{ a_{ N}^{ 4}}{ \left\vert \Lambda_{ N} \right\vert^{ 4}} \frac{ 1}{ N^{ 4}}\sum_{ \substack{ i_{ 1}, \ldots, i_{ 4}\\j_{ 1}, \ldots, j_{ 4}}} \prod_{ l=1}^{ 4} \left(\int_{ \Delta_{ j_{ l}}} \Psi(x_{ i_{ l}}, \tx)  \dd \tx\right) \bE\prod_{ l=1}^{ 4} \left(1+ \left\vert \theta_{ i_{ l}, t} \right\vert^{ \kappa_{ 1}} + \left\vert \omega_{ i_{ l}} \right\vert^{\iota_{ 1}}\right),\\
&\leq C\frac{ a_{ N}^{ 4}}{N^{ 4}} \frac{ 1}{ \left\vert \Lambda_{ N} \right\vert^{ 4}}\sum_{i_{ 1}, \ldots, i_{ 4}} \prod_{ l=1}^{ 4} \sum_{ j\in \Lambda_{ N}} \int_{ \Delta_{ j}} \Psi(x_{ i_{ l}}, \tx)  \dd \tx\leq C \frac{ a_{ N}^{ 4}}{ N^{ 4}},
\end{align*}
so that $\sup_{N\geq 1, t\leq T} \bE \left(\N{J_{ N, t}^{ (4)}}^{ 4}_{ -q}\right)< +\infty$.

The same estimate about the last term $\cK_{ N, t}$ can be proven along the same lines as we have done for the first term $\cI_{ N, t}$. We leave the proof to the reader. Proposition~ \ref{prop:etaNbounded1} is proven.\qed

\subsection{Proof of Proposition~\ref{prop:initial_cond_HN}}
\label{sec:proof_initial_cond_HN}
\begin{definition}
\label{def:cP_cQ}
Define the linear forms on the space of functions $(\theta, \omega, x, \ttheta, \tomega, \tx)\mapsto g(\theta, \omega, x, \ttheta, \tomega, \tx)$ on $(\bbX\times\bbY\times\bbS)^{ 2}$:
\begin{align}
\label{eq:cP}
\cP g(\theta, \omega, \ttheta, \tomega, x)&:= g (\theta, \omega, x, \ttheta, \tomega, x),\\
\label{eq:cQ}
\cQ g&:= g- \cP g.
\end{align}
\end{definition}
If we fix $(\theta, \omega, \ttheta, \tomega, x)$ and we see $g$ as a function of the second spatial variable $\tx$ only, we have to think of $\cP g$ as the first (constant) term in the Taylor decomposition of $\tx \mapsto g(\tx)$ around $x$, $\cQ g$ being the remainder.
\begin{proof}[Proof of Proposition~\ref{prop:initial_cond_HN}]
We decompose $ \eta_{ N, 0}$ \eqref{eq:fluct} and $\cH_{ N, 0}$ \eqref{eq:secorderfluc} into the sum of terms of different scalings (see Remark~\ref{rem:decomp0} below): for all test function $ (\theta, \omega, x)\mapsto f(\theta, \omega, x)$
\begin{align}
\left\langle \eta_{ N, 0}\, ,\, f\right\rangle&=  \frac{ a_{ N}}{ \left\vert \Lambda_{ N} \right\vert} \sum_{ i\in \Lambda_{ N}}\left(f(\theta_{ i, 0}, \omega_{ i}, x_{ i}) - \left[ f, \xi_{ 0}\right](x_{ i}) \right)\nonumber\\
&+ a_{ N} \left(\frac{ 1}{ \left\vert \Lambda_{ N} \right\vert} \sum_{ i\in \Lambda_{ N}} \left[ f, \xi_{ 0}\right](x_{ i})  - \int \left[ f, \xi_{ 0}\right](x)\dd x\right):= \left\langle \eta_{ N, 0}^{ (1)}\, ,\, f\right\rangle + \left\langle \eta_{ N, 0}^{ (2)}\, ,\, f\right\rangle,\label{eq:decomp_eta0}
\end{align}
where $ \left[f, \xi_{ 0}\right](x)=\int_{ \bbX\times\bbY} f(\theta, \omega, x) \xi_{ 0}(\dd \theta, \dd \omega)$ and for every test functions $(\tau, \ttau)\mapsto g(\tau, \ttau)$
\begin{align}
\left\langle \cH_{ N, 0}\, ,\, g\right\rangle&= \frac{ a_{ N}}{ \left\vert \Lambda_{ N} \right\vert^{ 2}} \sum_{ i, j\in \Lambda_{ N}} \Psi(x_{ i}, x_{ j})\left( g( \tau_{ i, 0}, \tau_{ j, 0}) - \left[g, \xi_{ 0}\right](\tau_{ i, 0}, x_{ j})\right)\nonumber\\
\begin{split}
&+ \frac{ a_{ N}}{ \left\vert \Lambda_{ N} \right\vert}\sum_{ i\in \Lambda_{ N}}\bigg( \frac{1}{ \left\vert \Lambda_{ N} \right\vert} \sum_{ j\in \Lambda_{ N}} \Psi(x_{ i}, x_{ j}) \left[g, \xi_{ 0}\right](\tau_{ i, 0}, x_{ j}) - \int_{ \bbS} \Psi(x_{ i}, \tx)\left[g, \xi_{ 0}\right](\tau_{ i, 0}, \tx)\dd \tx\bigg),
\end{split}\nonumber\\
&:= \left\langle \cH_{ N, 0}^{ (1)}\, ,\, g\right\rangle + \left\langle \cH_{ N, 0}^{ (2)}\, ,\, g\right\rangle,\label{eq:decomp_HN0}
\end{align}
where (recall \eqref{eq:cro_1}) $ \left[g, \xi_{ 0}\right](\tau, \tx)=\left[g, \xi_{ 0}\right]_{ 1}(\theta, \omega, x, \tx)= \int_{ \bbX\times\bbY} g( \tau, \ttheta, \tomega, \tx) \xi_{ 0}(\dd \ttheta, \dd \tomega)$.
\begin{remark}
\label{rem:decomp0}
In \eqref{eq:decomp_eta0} and \eqref{eq:decomp_HN0}, $ \eta_{ N, 0}^{ (1)}$ and $ \cH_{ N, 0}^{ (1)}$ capture the fluctuations of the i.i.d. particles and disorder at $t=0$ (and hence, should scale as $ \frac{ a_{ N}}{ \sqrt{N}}$ for large $N$ with Gaussian limits in the case $a_{ N}= \sqrt{N}$). On the other hand, the processes $ \eta_{ N, 0}^{ (2)}$ and $\cH_{ N, 0}^{ (2)}$ capture the variations of the system w.r.t. the spatial variables (for a fixed randomness). Hence, the scaling of $ \eta_{ N, 0}^{ (2)}$ and $\cH_{ N, 0}^{ (2)}$ is governed by the regularity of the weight $ \Psi$ and the test functions. This scaling is in any case different from the Gaussian scaling and yields to deterministic limits as $N\to\infty$.
\end{remark}
We first make the observation that in \eqref{eq:decomp_eta0}, $ \eta_{ N, 0}^{ (2)}$ converges  to $0$ as $N\to\infty$, regardless of the value of $ \alpha\in[0, 1)$. Indeed, since $f$ in \eqref{eq:decomp_eta0} is at least of class $\cC^{ 1}$ (recall \eqref{eq:embed_one_var}), one has $ \left\vert \left\langle \eta_{ N, 0}^{ (2)}\, ,\, f\right\rangle \right\vert \leq C \Ninf{ \nabla f}\frac{ a_{ N}}{ N}$, which goes to $0$ as $N\to\infty$ in both cases $ \alpha< \frac{ 1}{ 2}$ and $ \alpha> \frac{ 1}{ 2}$. Secondly, using in \eqref{eq:decomp_HN0} the decomposition of Definition~\ref{def:cP_cQ}, we write
\begin{equation}
\label{aux:HN02_PQ}
\left\langle \cH_{ N, 0}^{ (2)}\, ,\, g\right\rangle=\left\langle \cH_{ N, 0}^{ (2)}\, ,\, \cP g\right\rangle+ \left\langle \cH_{ N, 0}^{ (2)}\, ,\, \cQ g\right\rangle.
\end{equation}
The intuition for \eqref{aux:HN02_PQ} is that the only (possibly) nontrivial contribution to $\cH_{ N, 0}^{ (2)}$ only comes from $\cP g$, the remainder $\cQ g$ being already sufficiently regular in $(x, \tx)$ to compensate for the singularity of the kernel $ (x, \tx)\mapsto\Psi(x, \tx)$ in \eqref{eq:decomp_HN0}: regardless of the value of $ \alpha\in[0, 1)$,
\begin{equation}
\label{eq:conv_Qg}
\left\langle \cH_{ N, 0}^{ (2)}\, ,\, \cQ g\right\rangle\to0,\text{ as $N\to\infty$}.
\end{equation} Indeed, since $g$ is at least of class $\cC^{ 2}$ (recall \eqref{eq:embed_one_var}), one can write
\begin{align*}
 \left[\cQ g, \xi_{ 0}\right](\theta, \omega, x, \tx) &= \int_{ \bbX\times\bbY} \left(g( \theta, \omega, x, \ttheta, \tomega, \tx) - g( \theta, \omega, x, \ttheta, \tomega, x) \right) \xi_{ 0}(\dd \ttheta, \dd \tomega),\\
 &=\int_{ \bbX\times\bbY} \int_{x}^{\tx}\partial_{ u}g( \theta, \omega, x, \ttheta, \tomega, u) \dd u\xi_{ 0}(\dd \ttheta, \dd \tomega),\\
  &= d(x, \tx)\underbrace{\left( \frac{ 1}{ d(x, \tx)}\int_{x}^{\tx} \left[\partial_{ u}g, \xi_{ 0}\right]_{ 1}(\theta, \omega, x, u) \dd u\right)}_{ :=Q(\theta, \omega, x, \tx)},
\end{align*}
where, by assumption on $g$, the function $Q$ is bounded and of class $\cC^{ 1}$. Writing $Q_{ i}(x_{ i}, x_{ j})=Q(\theta_{ i, 0}, \omega_{ i}, x_{ i}, x_{ j})$, $ \left\langle \cH_{ N, 0}^{ (2)}\, ,\, \cQ g\right\rangle$  can be written as
\begin{align}
\left\langle \cH_{ N, 0}^{ (2)}\, ,\, \cQ g\right\rangle&= \frac{ a_{ N}}{ \left\vert \Lambda_{ N} \right\vert}\sum_{ i\in \Lambda_{ N}}\bigg( \frac{1}{ \left\vert \Lambda_{ N} \right\vert} \sum_{ j\in \Lambda_{ N}} \Psi(x_{ i}, x_{ j}) \left[\cQ g, \xi_{ 0}\right](\theta_{ i, 0}, \omega_{ i}, x_{ i}, x_{ j})\nonumber\\&\qquad\qquad\qquad - \int_{ \bbS} \Psi(x_{ i}, \tx) \left[\cQ g, \xi_{ 0}\right](\theta_{ i, 0}, \omega_{ i}, x_{ i}, \tx)\dd \tx\bigg),\nonumber\\
&= \frac{ a_{ N}}{ \left\vert \Lambda_{ N} \right\vert}\sum_{ i\in \Lambda_{ N}}\bigg( \frac{1}{ \left\vert \Lambda_{ N} \right\vert} \sum_{ j\in \Lambda_{ N}} d(x_{ i}, x_{ j})^{ 1- \alpha} Q_{ i}(x_{ i}, x_{ j})- \int_{ \bbS} d(x_{ i}, \tx)^{ 1- \alpha} Q_{ i}( x_{ i}, \tx)\dd \tx\bigg),\nonumber\\
&= \frac{ a_{ N}}{ \left\vert \Lambda_{ N} \right\vert}\sum_{ i\in \Lambda_{ N}}\bigg(  \sum_{ j\in \Lambda_{ N}} \int_{ \Delta_{ j}} \left(d(x_{ i}, x_{ j})^{ 1- \alpha} Q_{ i}(x_{ i}, x_{ j})-  d(x_{ i}, \tx)^{ 1- \alpha} Q_{ i}( x_{ i}, \tx)\right)\dd \tx\bigg),\nonumber\\
\begin{split}
&= \frac{ a_{ N}}{ \left\vert \Lambda_{ N} \right\vert}\sum_{ i\in \Lambda_{ N}}\bigg(  \sum_{ j\in \Lambda_{ N}} \int_{ \Delta_{ j}}Q_{ i}(x_{ i}, x_{ j}) \left(d(x_{ i}, x_{ j})^{ 1- \alpha} -  d(x_{ i}, \tx)^{ 1- \alpha}\right)\dd \tx\bigg)\\
&+ \frac{ a_{ N}}{ \left\vert \Lambda_{ N} \right\vert}\sum_{ i\in \Lambda_{ N}}\bigg(  \sum_{ j\in \Lambda_{ N}} \int_{ \Delta_{ j}}d(x_{ i}, \tx)^{ 1- \alpha} \left(Q_{ i}(x_{ i}, x_{ j})-  Q_{ i}( x_{ i}, \tx)\right)\dd \tx\bigg),
\end{split}\nonumber\\
\label{aux:Q}&= \left\langle \cH_{ N, 0}^{ (5)}\, ,\, g\right\rangle + \left\langle \cH_{ N, 0}^{ (6)}\, ,\, g\right\rangle
\end{align}
where $ \Delta_{ j}:=[x_{ j}, x_{ j+1})$. We treat the two terms in \eqref{aux:Q} separately. First of all,
\begin{align*}
\left\vert \left\langle \cH_{ N, 0}^{ (5)}\, ,\, g\right\rangle \right\vert&\leq \Ninf{\partial_{ u}g}\frac{ a_{ N}}{ \left\vert \Lambda_{ N} \right\vert}\sum_{ i\in \Lambda_{ N}}\bigg(  \sum_{ j\in \Lambda_{ N}} \int_{ \Delta_{ j}} \left\vert d(x_{ i}, x_{ j})^{ 1- \alpha} -  d(x_{ i}, \tx)^{ 1- \alpha} \right\vert\dd \tx\bigg).
\end{align*}
A long but easy calculation shows that the term within the parentheses is of order $ \frac{ 1}{ N}$, uniformly in $i\in \Lambda_{ N}$ (it is the speed of convergence of the Riemann sum associated to $x \mapsto \left\vert x \right\vert^{ 1- \alpha}$ towards its integral), so that $\left\vert \left\langle \cH_{ N, 0}^{ (5)}\, ,\, g\right\rangle \right\vert\leq C \frac{ a_{ N}}{ N}$, which, in any case in $ \alpha$, goes to $0$ as $N\to \infty$. As far as the second term in \eqref{aux:Q} is concerned, one has
\begin{align*}
\left\vert \left\langle \cH_{ N, 0}^{ (6)}\, ,\, g\right\rangle \right\vert&\leq \Ninf{\partial_{ u}^{ 2}g} \frac{ a_{ N}}{ \left\vert \Lambda_{ N} \right\vert}\sum_{ i\in \Lambda_{ N}}\bigg(  \sum_{ j\in \Lambda_{ N}} \int_{ \Delta_{ j}} \left\vert x_{ j} - \tx \right\vert\dd \tx\bigg),
\end{align*}that is of order $ \frac{ a_{ N}}{ N}\to_{ N\to\infty} 0$ too. This proves \eqref{eq:conv_Qg}.

We specify now our analysis to cases $ \alpha< \frac{ 1}{ 2}$ and $ \alpha> \frac{ 1}{ 2}$. We first treat the case $ \alpha< \frac{ 1}{ 2}$. First observe that in this case the whole process $\cH_{ N, 0}^{ (2)}$ (not only $ \left\langle \cH_{ N, 0}^{ (2)}\, ,\, \cQ g\right\rangle$) converges to $0$ as $N\to\infty$. Indeed, since one has trivially $ \left[\cP g, \xi_{ 0}\right](\tau, \tx)= \left[g, \xi_{ 0}\right](\tau, x)$, we have
\begin{align}
\label{eq:HN02_Pg}
\left\langle \cH_{ N, 0}^{ (2)}\, ,\, \cP g\right\rangle&=\frac{ 1}{ \left\vert \Lambda_{ N} \right\vert}  \sum_{ i\in \Lambda_{ N}}\left[g, \xi_{ 0}\right](\tau_{ i, 0}, x_{ i}) \left\lbrace a_{ N}\bigg( \frac{1}{ \left\vert \Lambda_{ N} \right\vert} \sum_{ j\in \Lambda_{ N}} \Psi(x_{ i}, x_{ j})  - \int_{ \bbS} \Psi(x_{ i}, \tx) \dd \tx\bigg)\right\rbrace,
\end{align}which by Lemma~\ref{lem:fluct_Psi} goes to $0$ as $N\to\infty$. From this and the remarks made above, everything boils down to the identification of the limit of the process $( \eta_{ N, 0}^{ (1)}, \cH_{ N, 0}^{ (1)})_{ N\geq1}$. To do so, we identify the limit of $X_{ N}:= u \left\langle \eta_{ N, 0}^{ (1)}\, ,\, f\right\rangle + v\left\langle \cH_{ N, 0}^{ (1)}\, ,\, g\right\rangle$ for any scalars $u$ and $v$ and any test functions $f$ and $g$. By a density argument, it suffices to identify the limit for test functions $g$ of the type $g(\theta, \omega, x, \ttheta, \tomega, \tx)= g_{ 1}(\theta, \omega)g_{ 2}(\ttheta, \tomega) h(x, \tx)$. In such a case, (recall that $a_{ N}= \sqrt{N}$ and $ \left\vert \Lambda_{ N} \right\vert= 2N$)
\begin{align}
\label{aux:HN01}
\left\langle \cH_{ N, 0}^{ (1)}\, ,\, g\right\rangle&=\frac{ 1}{ \sqrt{N}} \sum_{ j\in \Lambda_{ N}} \bar g_{2}(\theta_{ j, 0}, \omega_{ j}) \left(\frac{ 1}{ 4N} \sum_{ i\in \Lambda_{ N}}g_{1}(\theta_{ i, 0}, \omega_{ i})h(x_{ i}, x_{ j})\Psi(x_{ i}, x_{ j})\right),
\end{align}
where
\begin{equation}
\bar g_{2}(\theta, \omega):= g_{2}(\theta, \omega) -\int g_{2}(\ttheta, \tomega)   \xi_{ 0}(\dd \ttheta, \dd \tomega).
\end{equation}
Intuitively, the term within brackets in \eqref{aux:HN01} should be, for large $N$, close to $ \Xi(x_{ j})$ where 
\begin{equation}
\label{eq:Xi}
\Xi(\tx):= \frac{ 1}{ 2} \int g_{1}(\theta, \omega) h(x, \tx) \Psi(x, \tx) \xi_{ 0}(\dd \theta, \dd \omega)\dd x.
\end{equation} Hence, one can decompose $\cH_{ N, 0}^{ (1)}$ into
\begin{equation}
\label{aux:HN01_2}
\begin{split}
\left\langle \cH_{ N, 0}^{ (1)}\, ,\, g\right\rangle&=\frac{ 1}{ \sqrt{N}} \sum_{ j\in \Lambda_{ N}} \bar g_{2}(\theta_{ j, 0}, \omega_{ j}) \Xi(x_{ j})\\& + \frac{ 1}{ \sqrt{N}} \sum_{ j\in \Lambda_{ N}} \bar g_{2}(\theta_{ j, 0}, \omega_{ j}) \left( \Xi(x_{ j}) - \frac{ 1}{ 4N} \sum_{ i\in \Lambda_{ N}}g_{1}(\theta_{ i, 0}, \omega_{ i})h(x_{ i}, x_{ j})\Psi(x_{ i}, x_{ j})\right)\\&:= \left\langle \cH_{ N, 0}^{ (3)}\, ,\, g\right\rangle + \left\langle \cH_{ N, 0}^{ (4)}\, ,\, g\right\rangle.
\end{split}
\end{equation}
Consequently, $X_{ N}$ may be written as
\begin{align}
X_{ N}&= u \left\langle \eta_{ N, 0}^{ (1)}\, ,\, f\right\rangle + v\left\langle \cH_{ N, 0}^{ (3)}\, ,\, g\right\rangle + v\left\langle \cH_{ N, 0}^{ (4)}\, ,\, g\right\rangle,\nonumber\\
&= \frac{ 1}{ \sqrt{2N}} \sum_{ i\in \Lambda_{ N}} \left(\frac{ u}{ \sqrt{2}}\bar f(\theta_{ i, 0}, \omega_{ i}, x_{ i}) + v \sqrt{2}\bar g_{2}(\theta_{ i, 0}, \omega_{ j}) \Xi(x_{ i}) \right)+ v\left\langle \cH_{ N, 0}^{ (4)}\, ,\, g\right\rangle \label{aux:XN_decomp}
\end{align}
Let us admit for a moment that $\cH_{ N, 0}^{ (4)}$ converges in $L^{ 2}$ to $0$, as $N\to\infty$. Then, Proposition~\ref{prop:initial_cond_HN} for $ \alpha< \frac{ 1}{ 2}$ is a consequence of Lyapounov Central Limit Theorem (see \cite{Billingsley1995}, Theorem~27.3, page 362) applied to \eqref{aux:XN_decomp} and the convergence as $N\to\infty$ of \[s_{ N}^{ 2}:= \frac{ 1}{ N}\sum_{ j\in \Lambda_{ N}} \bE \left(\left(\frac{ u}{ \sqrt{2}}\bar f(\theta_{ i, 0}, \omega_{ i}, x_{ i}) + v \sqrt{2}\bar g_{2}(\theta_{ i, 0}, \omega_{ j}) \Xi(x_{ i}) \right)^{ 2}\right)\] to $u^{ 2}C_{ \eta}(f, f) +2uv C_{ \eta, \cH}(f, g) + C_{ \cH}(g, g)$ (recall \eqref{eq:covariances_0} and \eqref{eq:Xi}). We are now left with proving that $\cH_{ N, 0}^{ (4)}$ in \eqref{aux:HN01_2} goes to $0$ in $L^{ 2}$ as $N\to\infty$. Indeed, using that $(\theta_{ i, 0}, \omega_{ i})_{ i\in \Lambda_{ N}}$ are \iid and that $\bE(\bar g_{2}(\theta_{ i, 0}, \omega_{ i}))=0$ for all $i\in \Lambda_{ N}$, we obtain
\begin{align*}
\bE \left( \left\langle \cH_{ N, 0}^{ (4)}\, ,\, g\right\rangle^{ 2}\right)&= \frac{ 1}{ N} \sum_{ i\in \Lambda_{ N}} \bE \left(\bar g_{2}(\theta_{ i, 0}, \omega_{ i})^{ 2}\right) \Xi(x_{ i})^{ 2}\\
\begin{split}
 - \frac{ 1}{ 2N^{ 2}} \sum_{ i, k\in \Lambda_{ N}} &\bE \left(\bar g_{2}(\theta_{ i, 0}, \omega_{ i})^{ 2}\right)\bE \left(g_{1}(\theta_{ k, 0}, \omega_{ k})\right)h(x_{ k}, x_{ i}) \Psi(x_{ k}, x_{ i}) \Xi(x_{ i})\\
+ \frac{ 1}{ 16 N^{ 3}} \sum_{ i, j, k\in \Lambda_{ N}} &\bE \left(\bar g_{2}(\theta_{ i, 0}, \omega_{ i})^{ 2}\right)\bE \left(g_{1}(\theta_{ j, 0}, \omega_{ j})\right)\bE \left(g_{1}(\theta_{ k, 0}, \omega_{ k})\right)h(x_{ j}, x_{ i}) \Psi(x_{ j}, x_{ i})h(x_{ k}, x_{ i}) \Psi(x_{ k}, x_{ i})
\end{split}\\
&= \bE \left(\bar g_{2}(\theta_{ 1, 0}, \omega_{ 1})^{ 2}\right) \bigg( \frac{ 1}{ N} \sum_{ i\in \Lambda_{ N}}  \Xi(x_{ i})^{ 2}\\
\begin{split}
&- \bE \left(g_{1}(\theta_{ 1, 0}, \omega_{ 1})\right)\frac{ 1}{ 2N^{ 2}} \sum_{ i, k\in \Lambda_{ N}} h(x_{ k}, x_{ i}) \Psi(x_{ k}, x_{ i}) \Xi(x_{ i})\\
&+ \bE \left(g_{1}(\theta_{ 1, 0}, \omega_{ 1})\right)^{ 2} \frac{ 1}{ 16 N^{ 3}} \sum_{ i, j, k\in \Lambda_{ N}} h(x_{ j}, x_{ i}) \Psi(x_{ j}, x_{ i})h(x_{ k}, x_{ i}) \Psi(x_{ k}, x_{ i})\bigg).
\end{split}
\end{align*}
As $N\to\infty$, this quantity goes to 
\begin{align*}
\lim_{ N\to\infty}\bE \left( \left\langle \cH_{ N, 0}^{ (4)}\, ,\, g\right\rangle^{ 2}\right) &=2 \int  \Xi(\tx)^{ 2} \dd \tx - 4\bE \left(g_{1}(\theta_{ 1, 0}, \omega_{ 1})\right)^{ 2} \int \left(\frac{ 1}{ 2}\int h(x, \tx) \Psi(x, \tx)\dd x\right)^{ 2}\dd \tx\\ &+ 2\bE \left(g_{1}(\theta_{ 1, 0}, \omega_{ 1})\right)^{ 2} \int \left(\frac{ 1}{ 2}\int h(x, \tx) \Psi(x, \tx)\dd x\right)^{ 2}\dd \tx\\&=2 \int  \Xi(\tx)^{ 2} \dd \tx - 4\int  \Xi(\tx)^{ 2} \dd \tx + 2\int  \Xi(\tx)^{ 2} \dd \tx =0.
\end{align*} 
which is the desired result. This concludes the proof of Proposition~\ref{prop:initial_cond_HN} in the case $ \alpha< \frac{ 1}{ 2}$.
\medskip

It remains to treat the case $ \alpha> \frac{ 1}{ 2}$ (where $a_{ N}=N^{ 1- \alpha}$): as already proven in the case $ \alpha< \frac{ 1}{ 2}$, the quantity $ \frac{ \sqrt{N}}{ \left\vert \Lambda_{ N} \right\vert} \sum_{ i\in \Lambda_{ N}}\left(f(\theta_{ i, 0}, \omega_{ i}, x_{ i}) - \left[ f, \xi_{ 0}\right](x_{ i}) \right)$ converges in law, so that $ \left\langle \eta_{ N, 0}^{ (1)}\, ,\, f\right\rangle= \frac{ a_{ N}}{ \sqrt{N}}\frac{ \sqrt{N}}{ \left\vert \Lambda_{ N} \right\vert} \sum_{ i\in \Lambda_{ N}}\left(f(\theta_{ i, 0}, \omega_{ i}, x_{ i}) - \left[ f, \xi_{ 0}\right](x_{ i}) \right)$ converges in law to $0$ as $N\to\infty$. Since it is also the case for $ \left\langle \eta_{ N, 0}^{ (2)}\, ,\, f\right\rangle$, the whole process $ \eta_{ N, 0}$ converges in law to $0$. It remains to identify the limit for $\cH_{ N, 0}$: note that $ \left\langle \cH_{ N, 0}^{ (1)}\, ,\, g\right\rangle = J_{ N, 0}^{ (1)}(g)$, where $J_{ N, t}^{ (1)}$ was defined in \eqref{eq:JN}. In particular, we see from \eqref{aux:bound_JN1} and the definition of $a_{ N}$ in \eqref{eq:aN} that $\cH_{ N, 0}^{ (1)}$ vanishes to $0$ as $N\to\infty$ when $ \alpha> \frac{ 1}{ 2}$. By the remarks made at the beginning of this proof, it only remains to identify the limit of $\left\langle \cH_{ N, 0}^{ (2)}\, ,\, \cP g\right\rangle$ in \eqref{aux:HN02_PQ}. An application of Lemma~\ref{lem:fluct_Psi} to \eqref{eq:HN02_Pg} clearly shows that $\left\langle \cH_{ N, 0}^{ (2)}\, ,\, \cP g\right\rangle$ converges as $N\to\infty$ to the quantity defined in \eqref{eq:H0_large_alpha}. This concludes the proof of Proposition~\ref{prop:initial_cond_HN}.
\end{proof}

\subsection{Proof of Proposition~\ref{prop:unique_L1}}
\label{sec:proof_unique_L1}
The proof of this result is standard and uses techniques developed by Kunita in \cite{MR876080} and used by Mitoma \cite{MR820620} and Jourdain and M\'el\'eard \cite{Jourdain1998}. Since the following is mostly an adaptation of the approach of \cite{Jourdain1998} to our case, we sketch the main lines of proof and refer to the mentioned references for technical details.

Since $ \theta \mapsto c(\theta, \omega)$ and $ \Gamma$ are supposed to be bounded as well as their derivatives up to order $3P+9$, the function $ \theta \mapsto v(t, \theta, \omega)$ is regular and bounded (with bounded derivatives), uniformly in $t\in[0, T]$.
According to \cite{MR876080}, Theorem~4.4, p.277, the flow $X_{ s, t}(\tau, \ttau)_{ 0\leq s\leq t \leq T}$ defines a $C^{ 3P+8}$ diffeomorphism. By \cite{MR876080}, for any differential operator $D$ on $(\bbR^{ m})^{ 2}$ of order smaller that $3P+8$
\begin{equation}
\sup_{ (\tau, \ttau)\in\bbX\times\bbY\times\bbS} \sup_{ 0\leq s\leq t\leq T} \bE_{ B, \tB} \left( \left\vert D X_{ s, t}(\tau, \ttau) \right\vert^{ r}\right)<+\infty,\ r>0.
\end{equation}
Using this estimate and backward Ito's formula (\cite{MR876080}, Theorem~1.1, p.256), it is possible to prove (see \cite{Jourdain1998}, p.761) that for all $g\in C_{ b}^{ 2}$, for all $0\leq s\leq t\leq T$, for all $(\tau, \ttau)$, 
\begin{equation}
\label{eq:back_Kolm_U}
U(t, s)g(\tau, \ttau) - g(\tau, \ttau) = \int_{s}^{t} \mathscr{L}_{ r}^{ (1)}(U(t, r)g)(\tau, \ttau)\dd r.
\end{equation}
The next step is to prove that \eqref{eq:back_Kolm_U} is also valid in the space $C_{ 3(P+2)}^{ 0, 2 \iota}$. This relies on the following lemma (we refer to \cite{Jourdain1998}, Lemma~3.11 for a proof of this result):
\begin{lemma}
\label{lem:regularity_U}
Under the assumptions of Section~\ref{sec:assumptions}, the operator $\mathscr{L}_{ t}^{ (1)}$ is continuous from $C_{ 3P+8}^{ 0, \iota}$ to $C_{ 3(P+2)}^{ 0, 2 \iota}$ and
\begin{align*}
\N{\mathscr{L}_{ t}^{ (1)}g}_{ C_{ 3(P+2)}^{ 0, 2 \iota}}&\leq C \N{g}_{ C_{ 3P+8}^{ 0, \iota}},\ t\in[0, t]\\
\N{\mathscr{L}_{ s}^{ (1)}g - \mathscr{L}_{ t}^{ (1)}g}_{ C_{ 3(P+2)}^{ 0, 2 \iota}}&\leq C \N{g}_{ C_{ 3P+8}^{ 0, \iota}} \left\vert t-s \right\vert,\ s,t\in[0, T].
\end{align*}
For any $j\leq 3P+8$, the operator $U(t, s)$ is a linear operator from $C_{ j}^{ 0, 0}$ to $C_{ j}^{ 0, \iota}$ such that
\begin{align*}
\N{U(t, s) g}_{ C_{ j}^{ 0, \iota}}&\leq C \N{g}_{ C_{ j}^{ 0, 0}},\ 0\leq s\leq t\leq T,\\
\N{U(t, s) g-U(t, s^{ \prime}) g}_{ C_{ j}^{ 0, \iota}}&\leq C \N{g}_{ C_{ j+1}^{ 0, 0}} \sqrt{s^{ \prime}-s},\ 0\leq s\leq s^{ \prime}\leq t\leq T.
\end{align*}
\end{lemma}
Note in particular that there exists $C>0$ such that for any differential operator $D$ of order smaller than $3(P+2)$, $ \N{D \left[ \Gamma \Psi, \nu_{ s}\right] - D \left[ \Gamma \Psi, \nu_{ t}\right]}_{ \infty}\leq C \left\vert t-s \right\vert$, for any $s, t\in[0, T]$. The change in the parameter $ \iota$ in the spaces $ C_{j}^{ 0, \iota}$ in Lemma~\ref{lem:regularity_U} comes from the fact that $(\theta, \omega)\mapsto c(\theta, \omega)$ is possibly unbounded in $ \omega$: there exists $C>0$ such that for all differential operators $D_{ 1}$ and $D_{ 2}$, with sum of orders smaller than $3(P+2)$, $ \frac{ \left\vert D_{ 1}g D_{ 2}c \right\vert}{ 1+ \left\vert \omega \right\vert^{ 2 \iota}}= \frac{ \left\vert D_{ 1}g\right\vert}{ 1+ \left\vert \omega \right\vert^{ \iota}} \frac{\left\vert  D_{ 2}c \right\vert (1+ \left\vert \omega \right\vert^{ \iota})}{ 1+ \left\vert \omega \right\vert^{ 2\iota}}\leq C\frac{ \left\vert D_{ 1}g\right\vert}{ 1+ \left\vert \omega \right\vert^{ \iota}}$, by assumption on $c$. As a consequence of Lemma~\ref{lem:regularity_U}, we know that if $g\in C_{ 3P+9}^{ 0, 0}$, $ s \mapsto \mathscr{L}_{ s}^{ (1)}(U(t, s)g)$ is continuous in $C_{ 3(P+2)}^{ 0, 2 \iota}$ and hence that $ \int_{0}^{t} \mathscr{L}_{ s}^{ (1)}(U(t, s)g)\dd s$ makes sense as a Riemann integral in $C_{ 3(P+2)}^{ 0, 2 \iota}$. In particular, we obtain, for every $g\in C_{ 3P+9}^{ 0, 0}$
\begin{equation}
\label{eq:back_Kolm_U_2}
U(t, s)g - g = \int_{s}^{t} \mathscr{L}_{ r}^{ (1)}(U(t, r)g)\dd r,\ \text{in $C_{ 3(P+2)}^{ 0, 2 \iota}$}.
\end{equation}
We are now ready to establish the representation \eqref{eq:repr_cE}: let $\cE$ a solution to \eqref{eq:cE_linear_eq} in $\cC([0, T], \bW_{ -2(P+2)}^{ \underline{ \kappa}+ \gamma, \underline{ \iota}+ \gamma})$. Since $\bW_{ -2(P+2)}^{ \underline{ \kappa}+ \gamma, \underline{ \iota}+ \gamma} \hookrightarrow \bW_{ -3(P+2)}^{ \underline{ \kappa}, \underline{ \iota}}$, we have, for all $g\in \bW_{ 3(P+2)}^{ \underline{ \kappa}, \underline{ \iota}}$, 
\begin{equation}
\label{eq:cE_linear_eq_2}
\left\langle \cE_{ t}\, ,\, g\right\rangle = \int_{0}^{t} \left\langle \cE_{ s}\, ,\, \mathscr{L}_{ s}^{ (1)}g\right\rangle \dd s + \int_{0}^{t} \left\langle R_{ s}\, ,\, g\right\rangle\dd s.
\end{equation}
Both relations \eqref{eq:back_Kolm_U_2} and \eqref{eq:cE_linear_eq_2} are in particular true for every $g\in C_{ 3P+9}^{ 0, 0}$, since $C_{ 3(P+2)}^{ 0, 2 \iota}\hookrightarrow \bW_{ 3(P+2)}^{ \underline{ \kappa}, \underline{ \iota}}$ (by \eqref{eq:embed_cont} and since by assumption $\underline{ \kappa}> m$ and $ \underline{ \iota}> n + 2 \iota$, recall \eqref{eq:kappas} and \eqref{eq:iotas}). Combining \eqref{eq:back_Kolm_U_2} and \eqref{eq:cE_linear_eq_2}, one obtains for $g\in C_{ 3P+9}^{ 0, 0}$
\begin{align}
\left\langle \cE_{ t}\, ,\, g\right\rangle&= \int_{0}^{t} \left\langle (\mathscr{L}_{ s}^{ (1)})^{ \ast} \cE_{ s} + R_{ s}\, ,\, U(t, s) g - \int_{s}^{t} \mathscr{L}_{ r}^{ (1)}(U(t, r)g)\dd r\right\rangle \dd s \nonumber\\
&= \int_{0}^{t} \left\langle (\mathscr{L}_{ s}^{ (1)})^{ \ast} \cE_{ s} + R_{ s}\, ,\, U(t, s) g\right\rangle \dd s - \int_{0}^{t} \int_{s}^{t}\left\langle (\mathscr{L}_{ s}^{ (1)})^{ \ast} \cE_{ s} + R_{ s}\, ,\,  \mathscr{L}_{ r}^{ (1)}(U(t, r)g)\right\rangle \dd r \dd s,\nonumber\\
&= \int_{0}^{t} \left\langle (\mathscr{L}_{ s}^{ (1)})^{ \ast} \cE_{ s} + R_{ s}\, ,\, U(t, s) g\right\rangle \dd s - \int_{0}^{t} \int_{0}^{r}\left\langle (\mathscr{L}_{ s}^{ (1)})^{ \ast} \cE_{ s} + R_{ s}\, ,\,  \mathscr{L}_{ r}^{ (1)}(U(t, r)g)\right\rangle \dd s \dd r,\nonumber\\
&= \int_{0}^{t} \left\langle (\mathscr{L}_{ s}^{ (1)})^{ \ast} \cE_{ s} + R_{ s}\, ,\, U(t, s) g\right\rangle \dd s - \int_{0}^{t} \left\langle \cE_{ r}\, ,\, \mathscr{L}_{ r}^{ (1)}(U(t, r)g)\right\rangle \dd r,\label{aux:repr_cE}\\
&= \int_{0}^{t} \left\langle R_{ s}\, ,\, U(t, s) g\right\rangle \dd s,\nonumber
\end{align}
where we used again \eqref{eq:cE_linear_eq_2} in \eqref{aux:repr_cE}. Since $C_{ 3P+9}^{ 0, 0}$ is dense in $C_{ 2(P+2)}^{0,0}$, the identity $\cE_{ t}= \int_{0}^{t} U(t, s)^{ \ast} R_{ s}\dd s$ holds in $C_{ -2(P+2)}^{0, 0}$. Since $C_{ 2(P+2)}^{0, 0}$ is dense in $\bW_{ 2(P+2)}^{ \underline{ \kappa}+ \gamma, \underline{ \iota}+ \gamma}$, uniqueness holds for \eqref{eq:cE_linear_eq} in $\cC([0, T], \bW_{ -2(P+2)}^{ \underline{ \kappa}+ \gamma, \underline{ \iota}+ \gamma})$. Proposition~\ref{prop:unique_L1} is proven.\qed

\section*{Acknowledgements}
Part of this work was made during the stay of E.L. at the Institut f\"ur Mathematik at the Technische Universit\"at and at the Bernstein Center for Computational Neuroscience in Berlin. E.L. thanks the members of the Probability group for the warm hospitality. This work has been partially supported by the BMBF, FKZ 01GQ 1001B.

\bibliographystyle{siam}
\def\cprime{$'$}

\end{document}